\documentclass{amsart}
\usepackage{amsmath, amsthm, amssymb, amscd, mathrsfs, eucal, epsfig}
\usepackage{algorithm} 
\usepackage[noend]{algpseudocode}
\usepackage{enumitem}
\usepackage{pinlabel}
\usepackage{tikz-qtree}
\usepackage{anyfontsize}

\DeclareMathOperator{\PSL}{\mathrm{PSL}}

\input xy
\xyoption{all}

\usepackage{hyperref}

\usepackage{pinlabel}

\begin{document}

\newtheorem{theorem}{Theorem}[subsection]
\newtheorem{lemma}[theorem]{Lemma}
\newtheorem{corollary}[theorem]{Corollary}
\newtheorem{conjecture}[theorem]{Conjecture}
\newtheorem{proposition}[theorem]{Proposition}
\newtheorem{question}[theorem]{Question}
\newtheorem{problem}[theorem]{Problem}
\newtheorem*{main_thm}{Main Theorem~\ref{thm:main_theorem}}
\newtheorem*{claim}{Claim}
\newtheorem*{criterion}{Criterion}
\theoremstyle{definition}
\newtheorem{definition}[theorem]{Definition}
\newtheorem{construction}[theorem]{Construction}
\newtheorem{notation}[theorem]{Notation}
\newtheorem{convention}[theorem]{Convention}
\newtheorem*{warning}{Warning}

\newtheorem*{disconnected_is_schottky}{Theorem~\ref{theorem:disconnected_is_Schottky} (Disconnected is Schottky)}
\newtheorem*{interior_is_dense}{Theorem~\ref{theorem:interior_dense} (Interior is almost dense)}
\newtheorem*{limit_of_holes}{Theorem~\ref{theorem:hole_limit} (Limit of holes)}
\newtheorem*{renormalizable_traps}{Theorem~\ref{theorem:similarity} (Renormalizable traps)}
\newtheorem*{hole_in_M_0}{Theorem~\ref{theorem:hole_in_SetB} (Hole in $\SetB$)}

\theoremstyle{remark}
\newtheorem{remark}[theorem]{Remark}
\newtheorem{example}[theorem]{Example}
\newtheorem*{case}{Case}

\def\Z{{\mathbb Z}}
\def\R{{\mathbb R}}
\def\C{{\mathbb C}}
\def\D{{\mathbb D}}
\def\DD{{\mathcal D}}
\def\H{{\mathbb H}}
\def\L{{\mathcal L}}
\def\M{{\mathcal M}}
\def\T{{\mathcal T}}
\def\E{{\mathcal E}}
\def\SetA{{\mathcal M}}
\def\SetB{{\mathcal M_0}}
\def\SetC{{\mathcal M_1}}
\def\SetAA{{\overline{\mathcal M}}}
\def\SetBB{{\overline{\mathcal M}_0}}
\def\SetCC{{\overline{\mathcal M}_1}}
\def\fz{f_z}
\def\gz{g_z}
\def\Lz{\Lambda_z}
\def\Gz{G_z}
\def\P{{\mathcal P}}

\def\tdLz{L_z}
\def\tdL{L}

\newcommand\numberthis{\addtocounter{equation}{1}\tag{\theequation}}
\newcommand{\marginal}[1]{\marginpar{\tiny #1}}

\title{Roots, Schottky semigroups, and a proof of Bandt's Conjecture}
\author{Danny Calegari}
\address{Department of Mathematics \\ University of Chicago \\
Chicago, IL, 60637}
\email{dannyc@math.uchicago.edu}
\author{Sarah Koch}
\address{Department of Mathematics \\ University of Michigan \\
Ann Arbor, MI, 48109}
\email{kochsc@umich.edu}
\author{Alden Walker}
\address{Department of Mathematics \\ University of Chicago \\
Chicago, IL, 60637}
\email{akwalker@math.uchicago.edu}
\date{\today}

\begin{abstract}
In 1985, Barnsley and Harrington defined a ``Mandelbrot Set'' $\SetA$ for pairs of similarities 
--- this is the set of complex numbers $z$ with $0<|z|<1$ for which the limit set of the 
semigroup generated by the similarities $$x \mapsto zx \text{ and } x \mapsto z(x-1)+1$$ is connected. 
Equivalently, $\SetA$ is the closure of the set of roots of polynomials with coefficients in 
$\lbrace -1,0,1 \rbrace$. Barnsley and Harrington already noted the (numerically apparent) 
existence of infinitely many small ``holes'' in $\SetA$, and conjectured that these holes 
were genuine. These holes are very interesting, since they are ``exotic'' components of 
the space of (2 generator) Schottky semigroups. The existence of at least one hole was 
rigorously confirmed by Bandt in 2002, and he conjectured that the interior points are dense away from the real axis.  We introduce the 
technique of {\em traps} to construct and certify interior points of $\SetA$, and use them 
to prove Bandt's Conjecture. Furthermore, our techniques let us certify the existence 
of infinitely many holes in $\SetA$.
\end{abstract}

\maketitle

\setcounter{tocdepth}{1}
\tableofcontents

\section{Introduction}

Consider the similarity transformations $f,g:\C\to\C$ given by 
\[
f:x\mapsto zx\quad\text{and}\quad g:x\mapsto z(x-1)+1,
\]
where $z\in \D^*:=\{z\in\C\;|\;0<|z|<1\}$. Because these maps are contractions, there is a nonempty compact attractor $\Lz\subseteq\C$ associated with the {\em iterated function system} (or IFS) given by the pair $\{f,g\}$. The attractor $\Lz$ coincides with the set of accumulation points of the forward orbit of any $x\in\C$ under the semigroup $G_z:=\langle f,g\rangle$. 

In this article, we study the topology of certain subsets of the parameter space $\D^*$ for $G_z$. The first set we consider is the {\em connectedness locus}, denoted by $\SetA$; that is, the set of parameters $z$ for which $\Lz$ is connected. Standard IFS arguments prove that the limit set $\Lz$ is either connected, or it is a Cantor set (for details, see Lemma \ref{lemma:disconnected_Cantor}). 

The second subset of the parameter space we examine is related to the geometry of $\Lz$. For all values of the parameter $z\in\D^*$, the map $f$ fixes $0$, and the map $g$ fixes $1$. As both of these maps are contracting by the same factor (in fact, by a factor of $z$) around their respective fixed points, the limit set $\Lz$ has a center of symmetry about the point $1/2$ in the dynamical plane. The set $\SetB$ is defined to be the set of parameters $z$ for which $\Lz$ contains the point $1/2$. 


The sets $\SetA$ and $\SetB$ have been studied by various mathematicians over the past 30 years:  Barnsley-Harrington \cite{Barnsley_Harrington}, Bousch \cite{Bousch1, Bousch2}, Bandt \cite{Bandt}, Solomyak \cite{Solomyak_local,Solomyak}, Shmerkin-Solomyak \cite{Shmerkin_Solomyak}, and Solomyak-Xu \cite{Solomyak_Xu}, to name a few. 

There is a profound and unexpected connection between the sets $\SetA$ and $\SetB$ and the set of roots of power series with prescribed coefficients (see Section \ref{section:regular}). In particular, $\SetA$ can be identified with the closure of the set of roots of polynomials with coefficients in $\{-1,0,1\}$ (which are in $\D^*)$, and $\SetB$ can be identified with the closure of the set of roots of polynomials with coefficients in $\{-1,1\}$ (which are in $\D^*$). Via this formulation, the set $\SetB$ is related to roots of the minimal polynomials associated to the {\em core entropy} of real quadratic polynomials as defined by Thurston \cite{Thurston_entropy}, and established by Tiozzo \cite{Tiozzo}. 
We further elaborate on the history of $\SetA$ and $\SetB$ in Section \ref{subsection:history}. 

In \cite{Bousch1} and \cite{Bousch2}, Bousch proved that the sets $\SetA$ and $\SetB$ are connected and locally connected. However, the complement of $\SetA$  and the complement of $\SetB$ are {\em disconnected}. The complement of $\SetA$ and the complement of $\SetB$ both contain a prominent central component (see Figure \ref{mandelbrot} and Figure \ref{Set_B}). In 1985, Barnsley and Harrington numerically observed other connected components of the complement, or ``holes'' in $\SetA$, and they conjectured that these holes are genuine. In 2002, Bandt rigorously established the existence of one hole in $\SetA$. In Theorem \ref{theorem:hole_limit}, we prove that there are {\em infinitely many} holes in $\SetA$.  

These ``exotic holes'' in $\SetA$ are quite interesting and somewhat mysterious; they appear to be very well-organized in parameter space, suggesting that there may be a combinatorial classification of them. We currently have found no such classification. 

\subsection{Statement of results} 

We prove that all of the connected components of $\D^*\setminus \SetA$ are {\em Schottky}, in the sense 
that if $z$ in $\D^*\setminus \SetA$, there is a topological disk $D$ containing $\Lz$, so that $f(D)\cap g(D) =\varnothing$, and $f(D)$ and $g(D)$ are contained in the interior of $D$. 
\begin{disconnected_is_schottky}
The semigroup $G_z$ has disconnected $\Lambda_z$ if and only if $G_z$ is Schottky.
\end{disconnected_is_schottky}

To prove that these exotic components in the complement of $\SetA$ exist, we introduce the method of {\em traps} (see Section \ref{subsection:traps_holes}), which allows us to numerically certify that a parameter $z\in \SetA$. This technique is different from Bandt's proof of the existence of these exotic holes. In fact, the existence of a trap is an open condition, so if there is a trap for the parameter $z\in\D^*$, then necessarily $z\in \mathrm{int}(\SetA)$. Traps therefore allow us to  access the interior points of $\SetA$. In \cite{Bandt}, Bandt conjectured that the interior of $\SetA$ is dense away from $\SetA\cap \R$ (see Conjecture \ref{conjecture:Bandt}). In Theorem \ref{theorem:interior_dense}, we prove Bandt's conjecture using traps. 
\begin{interior_is_dense}
The interior of $\SetA$ is dense away from the real axis; that is, 
\[
\SetA=\overline{\mathrm{int}(\SetA)}\cup (\SetA\cap \R).
\]
\end{interior_is_dense}

Interestingly, the proof of Theorem \ref{theorem:interior_dense} requires a complete characterization of the set of $z\in\SetA$ for which the limit set $\Lz$ is convex. This is established in Lemma \ref{lemma:convex_zonohedra}. 

In Section~\ref{section:renormalization}, we examine families of exotic holes in $\SetA$ which appear to spiral down and limit on a distinguished point $z\in \partial \SetA$ (see Figure \ref{figure:hexahole_spiral}). 
\begin{limit_of_holes}
Let $\omega \sim 0.371859+0.519411i$ be the root of the
polynomial $1-2z+2z^2-2z^5+2z^8$ with the given approximate value. Then
\begin{enumerate}
\item{$\omega$ is in $\SetA$, and $\SetB$; in fact, the intersection of $f\Lambda_\omega$
and $g\Lambda_\omega$ is exactly the point $1/2$;}
\item{there are points in the complement of $\SetA$ arbitrarily close to $\omega$; and}
\item{there are infinitely many rings of concentric loops in the interior of $\SetA$ which
nest down to the point $\omega$.}
\end{enumerate}
Thus, $\SetA$ contains infinitely many holes which accumulate at the point $\omega$.
\end{limit_of_holes}

We continue Section \ref{section:renormalization} by generalizing the methods of 
Theorem~\ref{theorem:hole_limit}.  We define the notion of {\em renormalization} and {\em limiting traps} to show that at certain renormalization points $z\in \SetA$, the set $\SetA$ is asymptotically similar to $\Gamma_z$, where $\Gamma_z$ is the limit set of the 3 generator IFS
\[
x\mapsto z(x+1)-1\quad x\mapsto zx\quad x\mapsto z(x-1)+1. 
\]
Previous results of Solomyak established this asymptotic similarity at certain `landmark points' in $\partial \SetA$. We reprove his results with a more algorithmic approach using traps, and as a consequence, we obtain ``asymptotic interior.''

\begin{renormalizable_traps}
Suppose that $\omega$ is a renormalization point.
There are constants $A$ and $B$, depending only on $\omega$, such that 
\begin{enumerate}
\item If $C \in (A+B\Gamma_\omega)$, then for all $\epsilon>0$, there is a 
$C'$ such that $|C-C'|<\epsilon$ and for all sufficiently 
large $n$, there is a trap for $\omega+C'\omega^{bn}$.
\item If $f\Lz\cap g\Lz$ is a single point, then there is $\delta>0$ such that for all $C \notin (A+B\Gamma_\omega)$ 
with $|C| <\delta$, the limit set for the parameter 
$\omega + C\omega^{bn}$ is disconnected for all sufficiently large $n$.
\end{enumerate}
\end{renormalizable_traps}
In Section \ref{section:holes_in_m0}, we prove that the complement of $\SetB$ is also disconnected by numerically certifying a loop in $\SetB$ which bounds a component of the complement. 
\begin{hole_in_M_0}
There is a hole in $\SetB$.
\end{hole_in_M_0}

\subsection{Outline} In Section \ref{Semigroups of similarities}, we establish key definitions and survey some previous results about $\SetA$ and $\SetB$. In Section \ref{elementary_estimate}, we collect a few elementary estimates about the geometry of $\Lz$. In Section \ref{section:regular}, we explore the connection the sets $\SetA$ and $\SetB$ have with roots of power series with prescribed coefficients in a more general context involving regular languages. In Section \ref{section:topology_of_lambda}, we establish some important results about the topology and geometry of the limit set, and we prove Theorem \ref{theorem:disconnected_is_Schottky}. We also present an algorithm (similar to an algorithm used by Bandt in \cite{Bandt}) to certify that the limit set $\Lz$ is disconnected. 

In Section \ref{section:differences}, we examine the set of differences between points in $\Lz$. This set of differences is actually the limit set $\Gamma_z$ of the 3 generator IFS 
\[
x\mapsto z(x+1)-1\quad x\mapsto zx\quad x\mapsto z(x-1)+1.
\]
In Section \ref{section:interior}, we introduce the notion of traps, and characterize the set of $z\in\SetA$ for which $\Lz$ is convex in Lemma \ref{lemma:convex_zonohedra}. In Theorem \ref{theorem:interior_dense}, we prove that the interior of $\SetA$ is dense away from the real axis, establishing Bandt's Conjecture \ref{conjecture:Bandt}. 

In Section \ref{section:holes}, we describe our trap-finding algorithm and prove the estimates required to certify that $\SetA$ has holes. In Section \ref{section:renormalization}, we introduce the notions of renormalization and limiting traps, and we prove Theorem \ref{theorem:hole_limit} and Theorem \ref{theorem:similarity}. In Section \ref{section:whiskers}, we discuss the ``real whiskers'' of $\SetA$, and we use a 2-dimensional real IFS for this analysis. And lastly, in Section \ref{section:holes_in_m0}, we prove that there is a hole in $\SetB$; that is, we prove that the complement of $\SetB$ is disconnected. 

\subsection{Acknowledgements}

We would like to thank Christoph Bandt, Emmanuel Breuillard, Giulio Tiozzo
and especially Boris Solomyak for comments, corrections, pointers to references, and
enthusiasm and interest in this project.
Danny Calegari was supported by NSF grant DMS 1405466. Sarah Koch was supported by NSF grant DMS 1300315 and a Sloan research fellowship. 
Alden Walker was supported by NSF grant DMS 1203888.

\section{Semigroups of similarities}\label{Semigroups of similarities}

\subsection{Definitions}

\begin{definition}
A {\em contracting similarity} (or just a {\em similarity}) with {\em center} $c \in \C$ and
{\em dilation} $z \in \C$ with $0<|z|<1$ is the complex affine map $\C\to\C$ given by 
$$x \mapsto z(x-c) + c.$$
\end{definition}

The composition of any positive number of similarities is again a similarity. The set of all
similarities is topologized as $\C \times \D^*$.
We are concerned in the sequel with semigroups generated by finitely many similarities. 

\begin{definition}
Let $G$ be a finitely generated semigroup of contracting similarities. The {\em limit set} $\Lambda$
(also called the {\em attractor}) is the closure of the set of fixed points of elements of $G$.
\end{definition}

The limit set of $G$ is the unique compact, nonempty invariant subset of $\C$ for the action of $G$.
In particular the action of $G$ on $\Lambda$ is minimal (every orbit is dense).

\begin{example}[Middle third Cantor set]\label{example:middle_third}
The semigroup $f:x \mapsto \frac 1 3 x$, $g:x \mapsto \frac 1 3 (x-1) + 1$ has the middle third Cantor set
as limit set. 
\end{example}

\begin{example}[Sierpinski carpet]\label{example:Sierpinski_carpet}
The semigroup $f:x \mapsto \frac 1 2 x$, $g:x \mapsto \frac 1 2 (x-1) + 1$, $h:x \mapsto \frac 1 2 (x-\omega) + \omega$
for $\omega = e^{i\pi/3}$ has the Sierpinski triangle as limit set.
\end{example}

\begin{definition}[Schottky semigroup]\label{definition:schottky}
Let $S$ be a finite set of contracting similarities, and let $G$ be the semigroup they
generate. We say that $G$ is a {\em Schottky semigroup} if there is an embedded loop 
$\gamma \subseteq \C$ bounding a closed (topological) disk $D$, so that the elements of $S$ take $D$
to disjoint disks contained in the interior of $D$.

A loop $\gamma$ with this property, and the disk $D$ it bounds is said to be {\em good} for $G$.
\end{definition}

\begin{lemma}
The Schottky semigroup $G$ is free (on $S$) and discrete as a subset of $\C \times \D^*$.
\end{lemma}
\begin{proof}
Actually, {\em every} finitely generated semigroup which is strictly contracting is
discrete, since the set of dilations accumulates only at $0$; so the point is to
prove freeness. This follows from Klein's ping-pong argument applied to a 
good disk $D$ and its translates.
\end{proof}

Note that if $S$ generates a Schottky semigroup, the centers of generators are distinct. Indeed, a good disk $D$ must contain all of the centers, and since the generators map $D$ to disjoint disks, the centers must be distinct. 
Thus for a Schottky semigroup $G$, the limit set is a Cantor set, which is the intersection of the
images of a good disk $D$ under elements of $G$, and which can be identified (topologically)
with the set of right-infinite words in the generators. Thus, any two Schottky semigroups
with the same number of generators have topologically conjugate actions on their limit sets.
In fact, we can say more:

\begin{lemma}\label{lemma:conjugacy}
Any two isomorphic Schottky semigroups $G$, $G'$ are topologically conjugate on their
restriction to good disks $D$, $D'$.
\end{lemma}
\begin{proof}
If $S$ and $S'$ are the generators of $G$ and $G'$, then choose any homeomorphism
$h:D-S(D) \to D' - S'(D')$ which extends a conjugacy on their boundaries, and extend
to $h:D-\Lambda  \to D' - \Lambda'$ using $D-\Lambda=G(D-S(D))$ and $D'-\Lambda' = G'(D'-S'(D'))$.
Then extend to $h:D \to D'$ by the canonical (abstract) isomorphism $h:\Lambda \to \Lambda'$
coming from the identification of these limit sets with the right-infinite words in the generators.
\end{proof}

\begin{remark}
Note that Schottky semigroups $G$, $G'$ are very rarely topologically conjugate on all of $\C$;
for, they are invertible on $\C$, and therefore any conjugacy would extend to a conjugacy 
between the {\em groups} they generate. But these are indiscrete, and indiscrete subgroups 
of $\PSL(2,\C)$ are rarely topologically conjugate.
\end{remark}

\subsection{Pairs of similarities}

For the remainder of the paper we focus almost entirely on semigroups generated by a pair of
similarities with the same dilation $z$. After conjugation by a similarity of $\C$ we may assume that
the two centers of the generators are at $0$ and $1$ respectively. Thus the space of conjugacy
classes of such semigroups is parameterized by $z \in \D^*$.

\begin{notation}
For $z \in \D^*$, let $\Gz$ denote the semigroup with generators
$$f:x \mapsto zx, \quad g:x \mapsto z(x-1) + 1,$$
and let $\Lz$ denote the limit set of $\Gz$. We omit the subscript $z$ from $f$ and $g$ to lighten notation. 
\end{notation}

\begin{figure}[htb]
\includegraphics[scale=0.2]{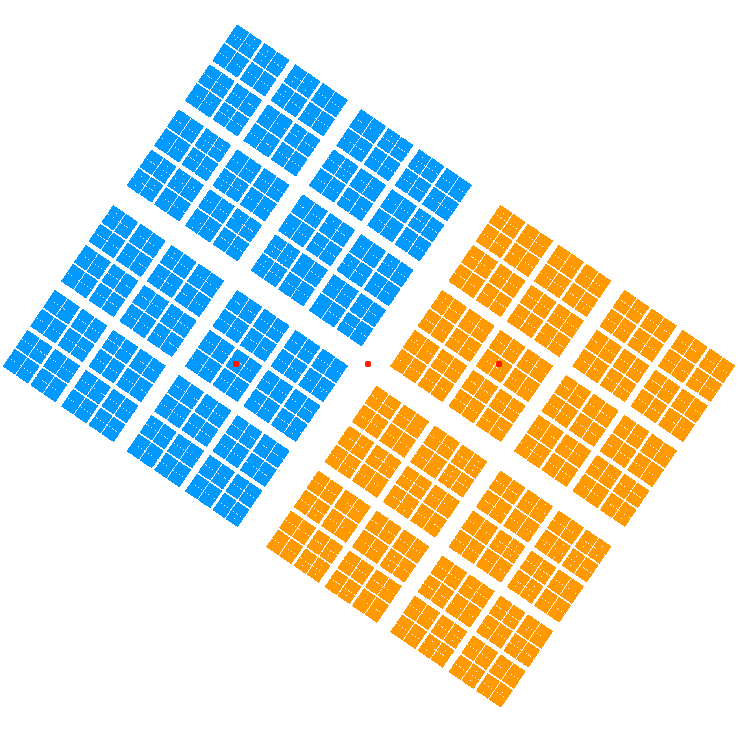}\hspace{1mm}
\includegraphics[scale=0.227]{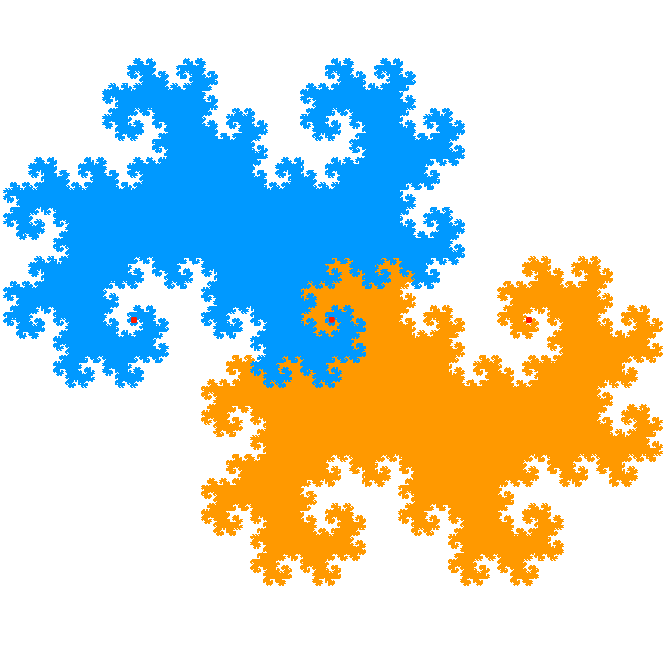}\hspace{1mm}
\includegraphics[scale=0.28]{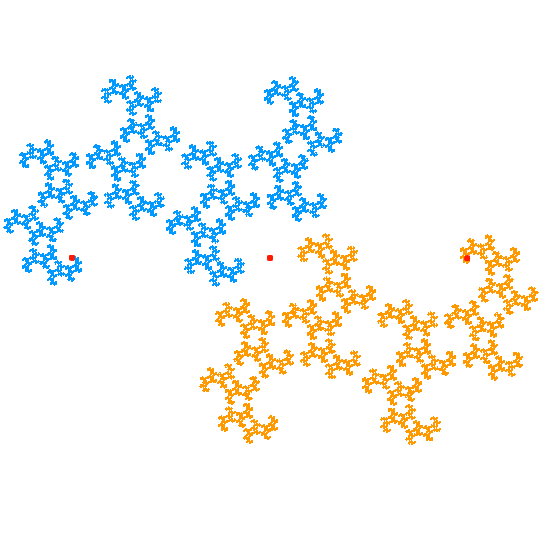}

\includegraphics[scale=0.22]{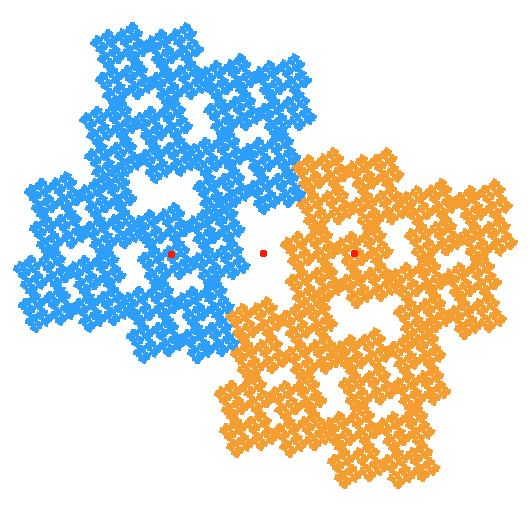}\hspace{2mm}
\includegraphics[scale=0.24]{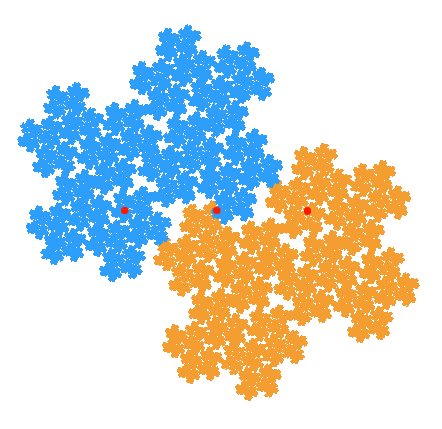}\hspace{1mm}
\includegraphics[scale=0.26]{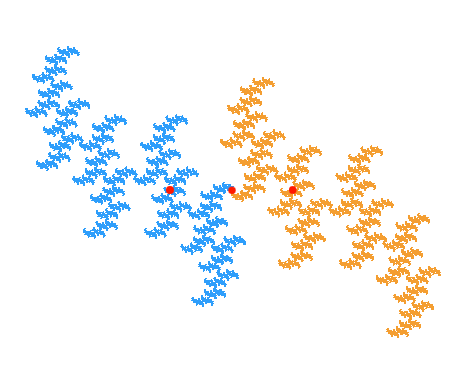}
\caption{Some limit sets $\Lz$ for various parameters.  
In each case, we show the decomposition of $\Lz$ as the 
union of $f\Lz$ (blue) and $g\Lz$ (orange).
The points $0$, $1/2$ and $1$ are marked in red.
Along the bottom from left to right, the parameters 
lie in $\SetA - \SetB$, $\SetB - \SetC$, and 
$\SetC$, respectively.}
\label{figure:limit_set_example}
\end{figure}

Other normalizations have some nice features. 
Barnsley and Harrington \cite{Barnsley_Harrington},
Bousch \cite{Bousch1} and others use the normalization
$$f: x \mapsto zx +1, \quad g: x \mapsto zx -1,$$
and Solomyak \cite{Solomyak} uses
$$f: x \mapsto zx, \quad g: x \mapsto zx + 1.$$
Our normalization has the convenient property that $0$ and $1$ are always in $\Lambda$ as the
centers of the two generators, independent of $z$.

\subsection{Basic symmetries}

Complex conjugation ``conjugates'' $\Gz$ to $G_{\overline{z}}$. Thus $\Lz$ and 
$\Lambda_{\overline{z}}$ are mirror images of each other. 
In particular, they are homeomorphic, and are therefore connected,
simply connected etc. for the same values of $z$.

The semigroup $\Gz$ has another basic symmetry: rotation through $\pi$ about the point $1/2$ 
interchanges the two generators $f$ and $g$. Thus the limit set $\Lz$ is invariant under
this symmetry: $\Lz = 1-\Lz$. On the other hand, by definition
$$\Lz = (z\Lz) \cup \left(z\Lz + (1-z)\right).$$
Using the relation $\Lz = 1-\Lz$ we obtain the identity
$$\Lz = (z\Lz) \cup \left(-z\Lz + 1\right)$$
which is the limit set of the semigroup $H_z$ with generators
$$f: x \mapsto zx, \quad g: x \mapsto 1-zx.$$
Thus, although $\Gz$ and $H_z$ are {\em not conjugate} (not even topologically, and in general
not even when restricted to $\Lz$), they have the same limit set. Now, from the definition,
the limit sets of $H_z$ and $H_{-z}$ are similar. It follows that the same is true for $\Gz$ and
$G_{-z}$.

We record this observation as a lemma:

\begin{lemma}[Similar limit sets]\label{lemma:similar_limit_sets}
The limit sets $\Lz,\Lambda_{-z},\Lambda_{\overline{z}}$ and $\Lambda_{-\overline{z}}$ are similar or mirror images of each other. 
\end{lemma}

\subsection{Three sets}

We now define three subsets in parameter space $\D^*$ of our semigroups $\Gz$. These sets
are the basic objects of interest in this paper. 
\begin{enumerate}
\item{$\SetA$ is the set of $z$ such that $\Lz$ is connected;}
\item{$\SetB$ is the set of $z$ such that $\Lz$ contains $1/2$; and}
\item{$\SetC$ is the set of $z$ such that $\Lz$ is connected and full.}
\end{enumerate}
Recall that a set is {\em full}  if its complement is connected. These sets are all closed.

As far as we know, the set $\SetA$ was first introduced by Barnsley-Harrington \cite{Barnsley_Harrington}, 
and the set $\SetB$ was first introduced by Bousch \cite{Bousch1}. 
We are not aware of any previous explicit mention of $\SetC$, although
Bandt \cite{Bandt}, Solomyak \cite{Solomyak} and others have studied the 
(closely related) set of $z$ for which $\Lz$ is a dendrite. Figure~\ref{mandelbrot} 
is a picture of $\SetA$, and Figure~\ref{Set_B} is a picture
of $\SetB$. The set $\SetC$ is much less substantial, and it is harder to draw a good picture.


\begin{figure}[htpb]
\centering
\includegraphics[scale=0.4]{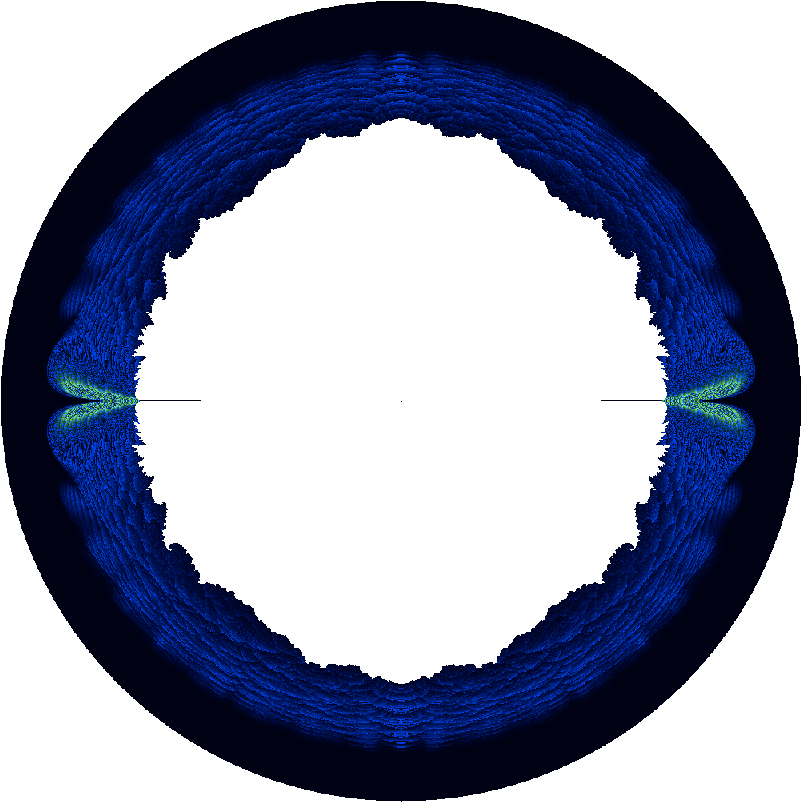}
\caption{$\SetA$ drawn in $\D^*$.}
\label{mandelbrot}
\end{figure}

\begin{figure}[htpb]
\centering
\includegraphics[scale=0.33]{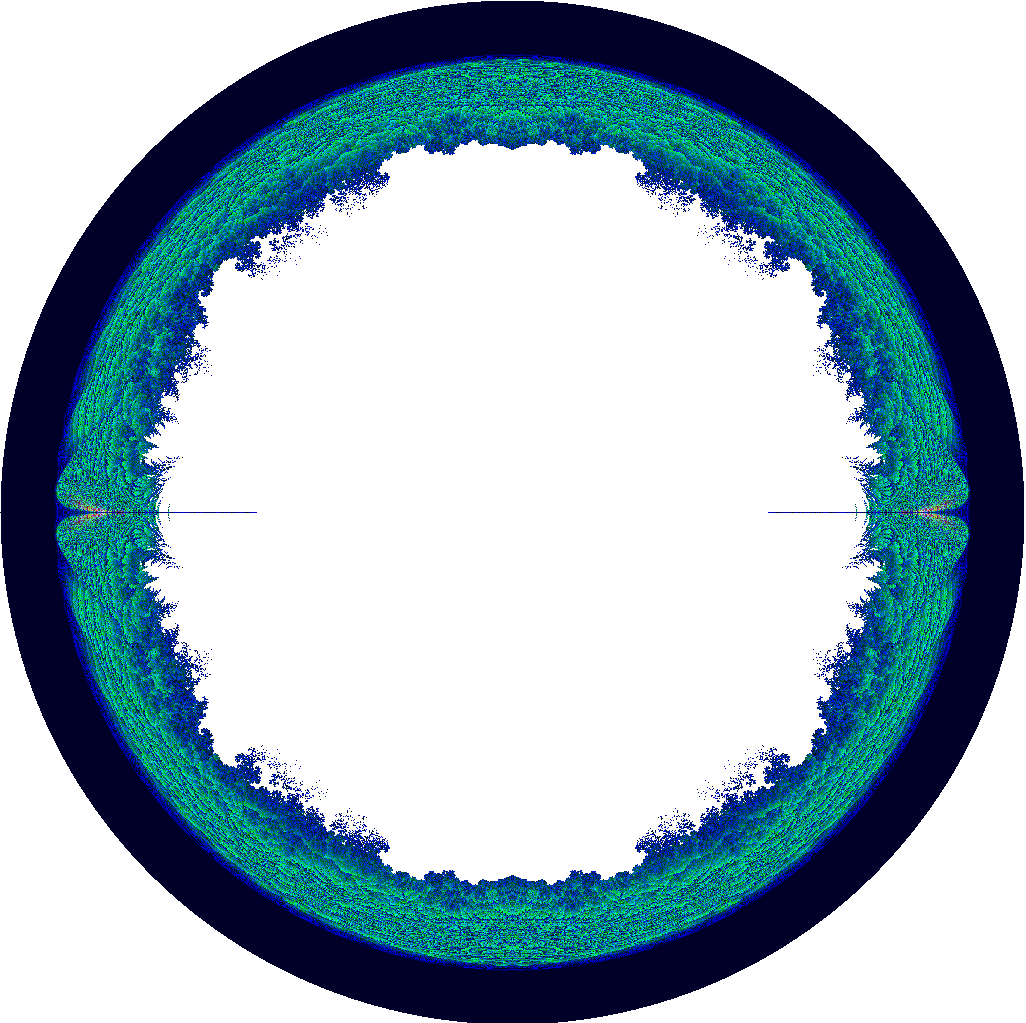}
\caption{$\SetB$ drawn in $\D^*$.}
\label{Set_B}
\end{figure}

\begin{proposition}\label{proposition:C_in_B_in_A}
We have $\SetC\subsetneq\SetB\subsetneq\SetA$.
\end{proposition}
\begin{proof}
It is straightforward to show (see Lemma~\ref{lemma:disconnected_Cantor}) 
that $z\in \SetA$ --- i.e.\/ the limit set $\Lz$ is connected ---
if and only if $f\Lz:=f(\Lz)$ intersects $g\Lz:=g(\Lz)$. Since $\Lz$ is rotationally
symmetric about the
point $1/2$, it follows that $\SetB$ is contained in $\SetA$. Likewise, if $\Lz$ is connected
and simply-connected, then because it is rotationally symmetric about $1/2$, it follows that
$\Lz$ contains $1/2$. No two of these sets are equal; see Figure \ref{figure:limit_set_example}. 
\end{proof} 

We will focus on the sets $\SetA$ and $\SetB$ for the remainder of the paper.

\subsection{Holes}

We will show (see Theorem~\ref{theorem:disconnected_is_Schottky}) 
that $z$ is in the complement of $\SetA$ if and only if
$\Gz$ is Schottky. We have already observed that all Schottky semigroups are topologically
conjugate when restricted to good disks. The set of $z$ for which $\Gz$ is Schottky is
evidently open. However, an examination of Figure~\ref{mandelbrot} with a microscope reveals the
apparent existence of tiny ``holes'' in $\SetA$, corresponding to ``exotic'' components of
Schottky space. 

One hole in $\SetA$ is clearly visible in Figure~\ref{mandelbrot}; it is shaped
approximately like a round disk except for two ``whiskers'' of $\SetA$ along the real axis.
But it turns out that there are also much smaller holes in $\SetA$, 
which can be thought of as exotic components of Schottky space. 
This is in stark contrast to the situation of Kleinian
{\em groups}, where the (Teichm\"uller) spaces of (quasifuchsian) representations of
a surface of fixed topological type are connected, as can be proved by means of
the measurable Riemann mapping theorem. 

Figure~\ref{hole} shows a collection of holes in $\SetA$ centered near the point
$0.372368+0.517839i$, which we refer to colloquially as {\em hexaholes}. 
The diameter of the picture is approximately $0.0005$, so these
holes are much too small to see in Figure~\ref{mandelbrot}. It is one of the main goals
of this paper to prove rigorously that infinitely many holes such as these really do
exist in $\SetA$.

\begin{figure}[htpb]
\centering
\includegraphics[scale=0.3]{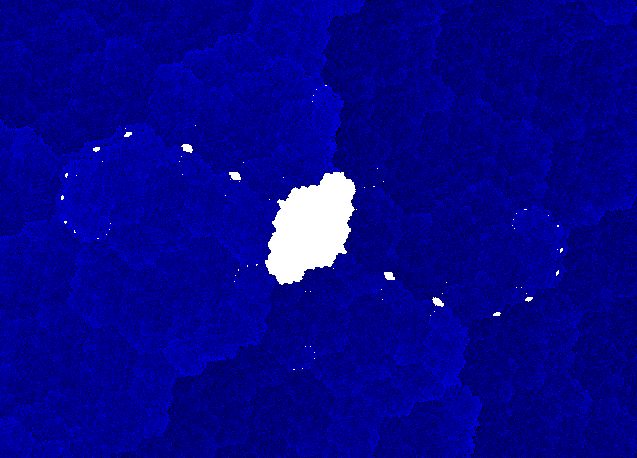}
\caption{Apparent holes in $\SetA$ near the point $z=0.372368+0.517839i$. The width of 
the figure is about $0.0005$.}
\label{hole}
\end{figure}

\subsection{Some history}\label{subsection:history}

The sets $\SetA$ and $\SetB$ have a long history, and these sets (and some close relatives)
were discovered independently several times by people working in
quite different areas of mathematics. In fact, we ourselves did not learn of the work of
Bandt and Solomyak until an advanced stage of our investigations. Therefore we believe it
would be useful to briefly mention some of the important papers on this subject that have
appeared over the last 30 years, and say something about their contents.

\begin{itemize}[leftmargin=*]
\item{
In 1985, Barnsley and Harrington \cite{Barnsley_Harrington} initiated a (mainly numerical)
study of $\SetA$. They discovered much structure evident in this set, most significantly the
presence of apparent holes, whose rigorous existence they conjectured.
Another phenomenon they discovered was the real whiskers in $\SetA$, and they proved rigorously that $\SetA$ is entirely
real in some definite neighborhood of the endpoints $\pm 0.5$ of these whiskers:

\begin{theorem}[Barnsley-Harrington, whiskers]\label{theorem:Barnsley_Harrington_whiskers}
There is a neighborhood of the points $\pm 0.5$ in which $\SetA$ is contained in $\R$.
\end{theorem}

Let $\alpha$ be the supremum of the real numbers $t$ for which $\SetA$ intersects
some neighborhood of $[0.5,t]$ only in real points. Barnsley-Harrington obtained a
rigorous estimate $\alpha > 0.53$ but observed that this estimate is far from sharp.}

\item{
In 1988 Thierry Bousch began a systematic study of 
$\SetA$ and $\SetB$ in his unpublished papers \cite{Bousch1} and \cite{Bousch2}. Bousch proved
many remarkable theorems about $\SetA$ and $\SetB$, including the following:

\begin{theorem}[Bousch, connectivity]\label{theorem:bousch_connectivity}
$\SetA$ and $\SetB$ are both connected and locally connected.
\end{theorem}

Bousch interpreted both sets as the zeros of power series with coefficients of a particular form; we will
return to this perspective in Section~\ref{section:regular}.
}

\item{
In 1993, Odlyzko and Poonen \cite{Odlyzko_Poonen} studied zeroes of polynomials with $0,1$
coefficients (a set closely related to $\SetB$) and showed the closure of this set is path connected;
their techniques are similar to those of Bousch. They also noted the presence of apparent 
holes, and conjectured that they really existed.
}

\item{
In 2002  Bandt \cite{Bandt} developed some fast algorithms to draw accurate pictures
of $\SetA$, and managed to rigorously prove the existence of a hole in $\SetA$, 
thus positively answering the conjecture of Barnsely-Harrington. Bandt first realized the importance
of understanding the set of interior points in $\SetA$, and made the following conjecture:
\begin{conjecture}[Bandt, interior almost dense]\label{conjecture:Bandt}
The interior of $\SetA$ is dense away from the real axis.
\end{conjecture}
\noindent which has been at the center of much subsequent work. Note that the necessity 
to exclude the real axis from this conjecture is already implied by
Theorem~\ref{theorem:Barnsley_Harrington_whiskers}.

Bandt's algorithm explicitly related $z \in \SetA$ to the dynamics
of a $3$-generator semigroup $f:x \mapsto zx-1$, $g:x \mapsto zx$, $h:x \mapsto zx+1$ which we denote $H_z$,
and remarked on the apparent similarity of $\SetA$ and the limit set $\Gamma_z$ of $H_z$
at certain algebraic points on $\partial \SetA$ that he called {\em landmark points}.}

\item{In 2003 Solomyak \cite{Solomyak_local} and 
Solomyak-Xu \cite{Solomyak_Xu} made partial progress on Bandt's conjecture, finding some interior points
in $\SetA$ with $|z|<2^{-1/2}$, and showing
that interior points are dense in $\SetA$ in some definite neighborhood of the imaginary axis. They
also obtained strong results on the structure of the natural invariant measures on $\Lz$, relating
this to the classical study of Bernoulli convolutions, and were able to compute the Hausdorff dimension
and measure of the limit set for almost all $z$.}

\item{In 2005 Solomyak \cite{Solomyak} proved the asymptotic similarity of $\SetA$ and $\Gamma_z$
at certain points $z$ which satisfy the condition that $z$ is a root of a rational function of a particular form. 
Following Solomyak, we refer to these points as {\em landmark points}. Then Solomyak shows

\begin{theorem}[Solomyak \cite{Solomyak}]\label{theorem:Solomyak}
If $z \in \SetA-\R$ is a {landmark point} then $\SetA$ is asymptotically similar at $z$ to
the set $\Gamma_z$ at a certain specific point, and both of these sets are asympotically
self-similar at these points.
\end{theorem}

Here asymptotic similarity of two sets $X$ and $Y$ at $0$ (for simplicity) means that the Hausdorff
distance between $t^{-1}(X)$ and $t^{-1}(Y)$ restricted to balls of fixed radius (and ignoring the
boundary) goes to zero as $t \to 0$; and asymptotic self-similarity means that there is a complex $z$ 
with $|z|<1$ so that the sets $z^nX$ converge on compact subsets in the Hausdorff topology to a limit.}

\item{In 2011, Thurston \cite{Thurston_entropy} studied the set of Galois conjugates of
algebraic numbers $e^\lambda$ where $\lambda$ is the core entropy of a postcritically finite interval 
map $x\mapsto x^2+c$, for which the parameter $c$ is taken from the main ``limb'' of the Mandelbrot set (the intersection of the Mandelbrot set with $\R$). He asserted that the closure of this set of roots (in $\C$) 
is connected and path connected. In $\D^*$ the closure of this set agrees with $\SetB$, and therefore the assertion
generalizes Theorem~\ref{theorem:bousch_connectivity}. For $|z|\ge 1$ this assertion was
verified in an elegant paper by Giulio Tiozzo \cite{Tiozzo}, who also went on to plot Galois conjugates associated
with core entropies of postcritically finite maps $x\mapsto x^2+c$, where $c$ comes from other limbs of the Mandelbrot set; these sets display a ``family resemblance'' to
$\SetB$.}
\end{itemize}

These papers describe some remarkable connections related to the theory of postcritically finite 
interval maps, Perron numbers, Galois theory and so on. 
The richness and mathematical depth of these various 
sets has barely begun to be plumbed.
We emphasize that the survey above is not exhaustive, and the papers cited contain a substantial 
amount beyond the part we summarize here.

\section{Elementary estimates}\label{elementary_estimate}

In this section we collect a few elementary estimates about the geometry of $\Lz$.

\subsection{Geometry of \texorpdfstring{$\Lz$}{Lambda(z)}}

Recall our notation $\Gz$ for the semigroup generated by 
 $f:x \mapsto zx$ and $g:x \mapsto z(x-1)+1$. The map $f$ fixes $0$ and the map
$g$ fixes $1$. Any element $e \in \Gz$ of length $n$ acts as a similarity on $\C$ with dilation
$z^n$ and center some point of $\Lz$. We make some {\it a priori} estimates
on the geometry of $\Lz$.

\begin{lemma}[Diameter bound]\label{lemma:diameter_bound}
The limit set $\Lz$ is contained in the ball of radius $|z-1|/2(1-|z|)$ centered at
$1/2$.
\end{lemma}
\begin{proof}
Let $D$ denote the ball of radius $R$ centered at $1/2$. Then $fD:=f(D)$ and $gD:=g(D)$
are the balls of radius $|z|R$ centered at $z/2$ and $1-z/2$ respectively. So providing
$R\ge |z-1|/2(1-|z|)$ we have $fD,gD \subseteq D$. But this means $\Lz \subseteq D$.
\end{proof}

\begin{lemma}\label{lemma:distance_estimate}
Let $e,e'$ be words with a common prefix of length $n$. Let $x$ be contained in $D$,
the ball of radius $|z-1|/2(1-|z|)$ centered at $1/2$. Then
$$d(ex,e'x)\le \frac {|z|^n|z-1|}{1-|z|}.$$
\end{lemma}
\begin{proof}
Write $e=uv$ and $e'=uv'$. Then $vx,v'x\in D$ so $d(vx,v'x)\le |z-1|/(1-|z|)$
by Lemma~\ref{lemma:diameter_bound}. But the dilation of $u$ is $|z|^n$,
so the estimate follows.
\end{proof}


\begin{definition}[Compactification]\label{definition:compactification}
Let $\Sigma$ denote the set of finite words in the alphabet $\{f,g\}$, and let $\overline{\Sigma}$ denote all right-infinite words in this alphabet, such that if a word contains $*$, all successive letters are also $*$. Metrize $\overline{\Sigma}$
with the metric $d(e,e')=2^{-n}$ where $n$ is the length of the biggest common prefix of $e$ and
$e'$.

The set $\overline \Sigma$ decomposes naturally into the subset $\partial \Sigma$ of words
not containing the symbol $*$, and words that {do} contain the symbol $*$ which are
in natural bijection with $\Sigma$, under the map that takes a finite word in $f,g$ to the infinite
word obtained by padding with infinitely many $*$ symbols.
\end{definition}

\begin{lemma}\label{lemma:compactification}
The space $\Sigma$ is compact. The subspace $\partial \Sigma$ is homeomorphic to a Cantor
set, and $\Sigma$ is homeomorphic to a discrete set, whose accumulation points are precisely
$\partial \Sigma$.
\end{lemma}
\begin{proof}
This is immediate from the definition.
\end{proof}

There is a natural symmetry of $\overline\Sigma$ interchanging the symbols $f$ and $g$ and fixing the symbol $*$. 

Note that $\overline\Sigma$ is formally distinct from $G_z$, which is the semi-group generated by compositions of the affine maps $f$ and $g$. 
It's important to make the distinction between $\Sigma$ and $G_z$ as we are interested in how the semigroup changes as the parameter $z$ varies. 
\begin{definition} There is an obvious map 
\[
\sigma_z:\Sigma\to \Gz
\]
such that $\sigma_z(u)\in \Gz$ is the appropriate composition of the maps 
\[
f:x\mapsto zx, \quad \text{and}\quad g:x\mapsto z(x-1)+1.
\] 
\end{definition}


\begin{definition}\label{definition:wordaction}
Let $u\in \Sigma$ be a word of length $n$, and (by abusing notation), define the map $u:\D^*\times \C\to \C$ given by 
\[
u:(z,x)\mapsto \sigma_z(u)(x).
\]
We will also use the notation $u(z)(x):=u(z,x)$, and we often consider the map $\D^*\to \C$ given by $z\mapsto u(z,x)$. The map $u$ is continuous in both $z$ and $x$, which is evident in Section \ref{section:regular}. 
\end{definition}

\begin{definition} The map $\pi:\partial\Sigma\times \D^*\to \C$ is defined by 
\[
\pi(u,z)=\lim_{n\to\infty} u_n(z,x)
\]
where $u_n$ is the prefix of $u$ of length $n$, and $x\in\C$ is any point. By Lemma \ref{lemma:distance_estimate}, this limit is well-defined, independent of the point $x\in \C$. 
\end{definition}
%
%
%
%

\begin{lemma}\label{lemma:word_composition}
For $u,v \in \Sigma$ and $x \in \C$, we have $uv(z,x) = u(z,v(z,x))$.  That is, 
$uv(z) = u(z) \circ v(z)$.
For $u \in \Sigma$ and $v \in \partial \Sigma$, we have 
$\pi(uv,z) = u(z,\pi(v,z))$.
\end{lemma}
\begin{proof}
Obvious from the definitions.
\end{proof}

\begin{lemma}[H\"older continuous]\label{lemma:Holder_estimate}
The map $\pi(\cdot,z):\partial\Sigma\to\C$ is H\"older continuous with exponent
$\log{|z|}/\log(0.5)$, and the image is $\Lz$.
\end{lemma}
\begin{proof}
Evidently if $e$ is a periodic word $e:=vvvv\cdots$ then $\pi(e,z)$ is the center (i.e.\/
the fixed point) of $v$; since $\partial \Sigma$ is compact, if $\pi$ is continuous, then the image is closed
and is therefore equal to $\Lz$. So it suffices to show $\pi$ is H\"older, and
estimate the exponent.

From the definition, if $e,e'$ have a common maximal prefix of length $n$ then 
$d_{\overline{G}}(e,e')=2^{-n}$. On the other hand, by Lemma~\ref{lemma:distance_estimate} we
obtain 
$$d(\pi(e,z),\pi(e',z))\le \frac {|z|^n|z-1|}{(1-|z|)} = \frac {(0.5^n)^\alpha|z-1|}{(1-|z|)}$$
for $\alpha=\log{|z|}/\log(0.5)$. 
\end{proof}


%

\subsection{Geometry of $\SetA$}


The following result is proved in \cite{Bousch1}; we include a proof for completeness. 
\begin{lemma}[inner and outer annuli]\label{lemma:inner_outer}
$\SetA$ (the set of $z$ for which the semigroup $\Gz$ has connected $\Lz$) 
contains the region
$|z|\ge 1/\sqrt{2} = 0.7071067\cdots$ and is contained in the region
$|z|\ge 1/2$. 
\end{lemma}
\begin{proof}
We shall see (Lemma~\ref{lemma:disconnected_Cantor}) that the limit set $\Lz$ 
of the semigroup $\Gz$ is disconnected
if and only if $f\Lambda_z \cap g\Lambda_z$ is empty, in which case $\Lz$ is
a Cantor set. In this case, the Hausdorff dimension of $\Lz$ can be computed
from Moran's Theorem (see \cite{Climenhaga_Pesin}, Ch.~2), as the unique $d$ for which
$$2|z|^d = 1.$$
In fact, this is easy to see directly: for a subset of Euclidean space,
the $d$-dimensional Hausdorff measure transforms by
$\lambda^d$ when the set is scaled linearly by the factor $\lambda$. When $\Lz$ is
disconnected, it is the disjoint union of $f\Lambda_z$ and $g\Lambda_z$, which are obtained 
(up to translation) by scaling $\Lambda_z$ by $z$; the formula follows.

If $|z|>1/\sqrt{2}$ then $d>2$ which is absurd, since
$\Lz$ is a subset of $\C$. Thus $|z|>1/\sqrt{2}$ is in $\SetA$, and since this set
is closed, so is $|z|\ge 1/\sqrt{2}$.

Conversely, if $|z|<1/2$ then the round disk $D$ of radius 1 centered at
$1/2$ is a good disk, so $\Gz$ is Schottky, and $\Lz$ is disconnected.
\end{proof}

\begin{example}
The estimates in Lemma~\ref{lemma:inner_outer} are sharp. Taking
$z=1/2$, we see that $f(1)=g(0)=1/2$, so $1/2\in f\Lambda_{\frac{1}{2}}\cap g\Lambda_{\frac{1}{2}}$; in fact, 
$\Lz= [0,1]$ in this case (and for all $z\in [1/2,1)$).

Likewise, taking $z=i/\sqrt{2}$ the rectangle $R$ with corners 
$\lbrace -1,i/\sqrt{2},2,2-i/\sqrt{2}\rbrace$
satisfies $R=fR\cup gR$, so that $R=\Lambda_{{i}/{\sqrt{2}}}$, whereas for $z=it$ with $t<1/\sqrt{2}$ 
the rectangle with corners $\lbrace -1,it,2,2-it\rbrace$ is good and $\Gz$ is Schottky; see
Figure~\ref{figure:limit_set_example}, left.
\end{example}


Solomyak--Xu \cite{Solomyak_Xu} Thm.~2.8 show that the set of $z$ with $|z|<2^{-1/2}$ for which
the Hausdorff dimension of $\Lz$ is different from $d:=-\log{2}/\log{|z|}$ itself has
Hausdorff dimension less than $2$. Since (as we shall show) interior points are dense in $\M$ away from $\R$, 
this implies that the simple formula for the Hausdorff dimension of $\Lz$ is valid on a dense
subset of $\M$. Finer results about the ``exceptions'' are known, but we do not pursue that here.

\section{Roots, polynomials, and power series with regular coefficients}\label{section:regular}

The most interesting mathematical objects are those that can be defined in many different ---
and apparently unrelated --- ways.
The sets $\SetA$ and $\SetB$ can be defined in a way which is (at first glance) entirely
unconnected to dynamics, namely as the closures of the set of {\em roots} of certain classes 
of polynomials. This connection is quite sensitive to the choice of 
normalization for our semigroups, and in fact the freedom to choose 
several different normalizations is itself of some theoretical interest.

\subsection{The Barnsley-Harrington and Bousch normalization} 

Recall the normalization 
$f:x \mapsto zx+1$, $g:x \mapsto zx-1$. If $w$ is a word of length $n$ in $f$ and $g$, we can express $wx$ as a
polynomial of a particularly simple form, namely
$$w(z,x) = \sum_{j=0}^{n-1} a_jz^j + xz^n$$
where the $a_j \in \lbrace -1,1\rbrace$ are equal to $1$ or $-1$ according to whether each successive
letter of $w$ is equal to $f$ or $g$.
In particular, the limit set $\Lz$ is precisely equal to the set of values 
of power series in $z$ with coefficients in $\lbrace -1,1\rbrace$. In this normalization, the center of symmetry is
$0$ (rather than $1/2$ as in our normalization), so we obtain the following characterization of $\SetB$:

\begin{proposition}[Power series with $\lbrace -1,1\rbrace$ coefficients]\label{proposition:SetB_power_series}
The set $\SetB$ is the set of $z\in\D^*$ which are zeros of power series with coefficients in
$\lbrace -1,1\rbrace$.
\end{proposition}

Similarly, the subsets $f\Lambda_z$ and $g\Lambda_z$ are the sets of {\em values} of power series with
$\lbrace -1,1\rbrace$ coefficients which start with $1$ and $-1$ respectively. 
Thus $z \in \SetA$ if and only if $f\Lambda_z \cap g\Lambda_z$ is nonempty,
which happens if and only if $z$ is a root of
a power series with coefficients in $\lbrace -2,0,2\rbrace$ starting with $\pm 2$.
Equivalently, after dividing such a
power series by $2$, we see that $z \in \SetA$ if and only if $z$ is a root of
a power series with coefficients in $\lbrace -1,0,1\rbrace$ starting with $\pm 1$:

\begin{proposition}[Power series with $\lbrace -1,0,1\rbrace$ coefficients]\label{proposition:SetA_power_series}
The set $\SetA$ is the set of $z\in\D^*$ which are zeros of power series with coefficients in
$\lbrace -1,0,1\rbrace$ with constant term $=\pm 1$.
\end{proposition}

In either case, zeros of power series with prescribed coefficients can be approximated by zeros of
{\em polynomials} with the same constraints on the coefficients. This suggests defining an
``extended'' $\SetA$ (resp. $\SetB$) to be the closures of the set of {\em all} roots $z$ 
(not just those in $\D^*$) of polynomials
with coefficients in $\lbrace -1,0,1\rbrace$ (resp. $\lbrace -1,1\rbrace$). 
Reversing the order of the coefficients replaces a root by its reciprocal, 
so these extended sets are exactly the sets obtained by taking the union of 
$\SetA$ (resp. $\SetB$) together with its image under inversion in the unit circle. 

Using this interpretation of $\Lambda_z$ as the set of values of power series with $\lbrace -1,1\rbrace$
coefficients, Bousch noted an interesting relationship between $\SetA$ and $\SetB$. We
give the proof here for several reasons. Firstly, Bousch's paper is unpublished, and this argument
is not easy to extract from the paper. Secondly, it is short and illuminating. Thirdly, it depends
on a geometric fact which we use later in the proof of Proposition~\ref{proposition:iterated_SetA_SetB}.

\begin{proposition}[Bousch \cite{Bousch1}, Prop.~2]\label{proposition:square_Set_B}
If $z^2 \in \SetA$ then $z \in \SetB$. Consequently $\SetB$ contains the annulus
$2^{-1/4}\le |z| < 1$.
\end{proposition}
\begin{proof}
Let $\P$ denote the set of power series with coefficients in
$\lbrace -1,1\rbrace$. Then for any $p \in \P$ we can write $p(z) = p_e(z^2)+zp_o(z^2)$ for 
unique $p_e,p_o \in \P$. But this means that $\Lz = \Lambda_{z^2} + z\Lambda_{z^2}$.

Now, in this normalization, limit sets all have rotational symmetry around $0$. So if $\Lambda_{z^2}$ is
connected but doesn't contain $0$, there is some symmetric innermost loop $\gamma$ around $0$. 
If $\Lz$ doesn't contain $0$, then (since $\Lambda_{z^2} = -\Lambda_{z^2}$) it must be that
$\Lambda_{z^2}$ and $z\Lambda_{z^2}$ are disjoint, so that $z\Lambda_{z^2}$ is contained in the disk
bounded by $\gamma$, and similarly $z^2\Lambda_{z^2}$ is contained in the disk bounded by
$z\gamma$, and therefore $\Lambda_{z^2}$ is disjoint from $z^2\Lambda_{z^2}$. 
But Bousch shows this is absurd in the following way.

Write $L:=z^2\Lambda_{z^2}$ so that $\Lambda_{z^2} = (L+1) \cup (L-1)$. By hypothesis, both $L$ and
$(L+1) \cup (L-1)$ are compact and connected, so that $(L-1)$ intersects $(L+1)$. But then
$L$ must intersect $(L+1) \cup (L-1)$, since if it is disjoint from them both, the union of $L$ with
vertical rays from its top-most and bottom-most point to infinity separates $(L-1)$ from $(L+1)$.
\end{proof}

\subsection{Our normalization}\label{subsection:coefficients}

Now let's return to our normalization $f:x \mapsto zx$, $g:x \mapsto z(x-1)+1 = zx + (1-z)$.
If we fix $e\in \partial \Sigma$ and vary $z$, note that $z\mapsto \pi(e,z)$ is a function of $z$.
For any fixed $e$, we can express $\pi(e,z)$ as a very simple power series in $z$.
In fact, the set of power series that can be obtained are precisely those whose coefficients,
listed in order, are (right-infinite) words in an explicit regular language (for an introduction
to the theory of regular languages, see e.g.\/ \cite{ECHLPT}).

\begin{proposition}[Power series]\label{proposition:power_series}
For any fixed $e\in \Sigma$ of length $m$ there is a formula
$$e(z,x) = xz^m+\sum_{j=0}^m a_j z^j$$
where each $a_j \in \lbrace -1,0,1\rbrace$, and $a_m=0$ if $e$ ends with $f$ and $a_m=-1$ if $e$ ends with $g$. 

Furthermore, the string of digits $a_j$ for $j<m$ can be recursively obtained as follows.
Read the letters of $e$ from left to right, and express this as a walk on the edges of
the directed labeled graph in Figure~\ref{coefficient_graph}, starting at the vertex labeled $*$.

\begin{figure}[htpb]
\centering
\labellist
\small\hair 2pt
\pinlabel $*$ at 50 150
\pinlabel $1$ at 150 125
\pinlabel $-1$ at 250 125
\pinlabel $0$ at 200 200
\pinlabel $0$ at 200 50
\pinlabel $f$ at 125 200
\pinlabel $g$ at 105 150
\pinlabel $g$ at 155 177
\pinlabel $g$ at 155 73
\pinlabel $f$ at 245 177
\pinlabel $f$ at 245 73
\pinlabel $f$ at 200 244
\pinlabel $f$ at 200 155
\pinlabel $g$ at 200 95
\pinlabel $g$ at 200 6
\endlabellist
\includegraphics[scale=0.75]{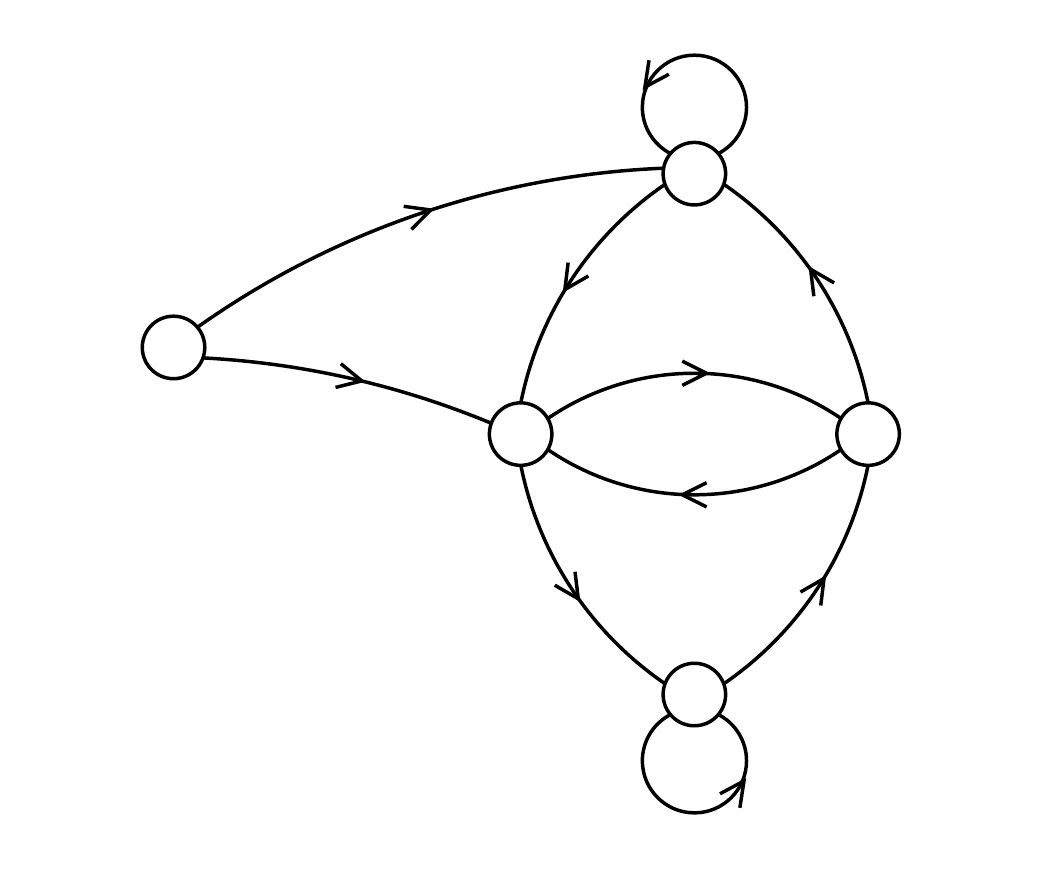}
\caption{The coefficients $a_j$ are the vertex labels visited in order on the walk associated to a word.}
\label{coefficient_graph}
\end{figure}
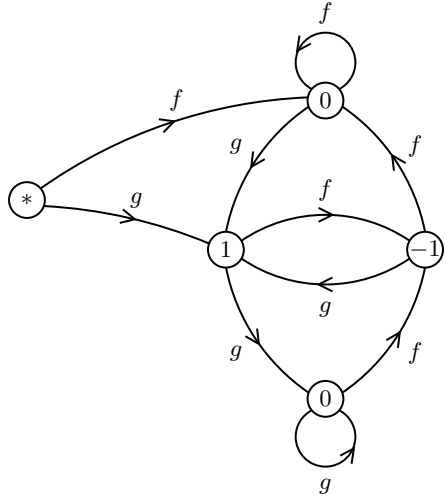

The coefficients $a_j$ for $j<m$ (in order) are the labels on the vertices visited in this path,
after the initial vertex. Thus, the sequences that occur are precisely the sequences in which
the nonzero coefficients alternate between $1$ and $-1$, starting with $1$.

Similarly, for any fixed $e\in \partial \Sigma$ we can write
$$\pi(e,z)  = \sum_{j=0}^\infty a_j z^j$$
where each $a_j \in \lbrace -1,0,1\rbrace$, and the
$a_j$ are obtained as the labels on the vertices associated to the right-infinite walk on
the graph as above.
\end{proposition}
\begin{proof}
This is immediate by induction.
\end{proof}

\begin{example}
From Proposition~\ref{proposition:power_series} we can quickly generate the formula for $e(z,x)$ for any finite
word $e$. For example, taking $e=gfgfffgg$, and writing $-1$ as $\bar{1}$,
we compute the sequence as follows:
$$\emptyset \xrightarrow{g} 1 \xrightarrow{f} 1\bar{1} \xrightarrow{g} 1\bar{1}1 \xrightarrow{f} 1\bar{1}1\bar{1}
\xrightarrow{f} 1\bar{1}1\bar{1}0 \xrightarrow{f} 1\bar{1}1\bar{1}00 \xrightarrow{g} 1\bar{1}1\bar{1}001 \xrightarrow{g} 1\bar{1}1\bar{1}0010$$
so that we obtain the formula
$$e(z,x) = 1-z+z^2-z^3+z^6+(x-1)z^8$$
\end{example}

\subsection{Regular coefficients}

We now show that in general families of semigroups of similarities with centers
depending polynomially on the (common) dilation give rise to limit sets
which are the values of power series with ``regular'' coefficients. Very similar, but
somewhat complementary observations were made by Mercat \cite{Mercat}.

\begin{definition}
Fix some finite alphabet $S$ of complex numbers, and fix a prefix-closed 
regular language $L \subseteq S^*$. Let $\overline{L}$ denote the set of right-infinite words
in $S$ whose finite prefixes are in $L$. Call a power series 
$$e(z):=a_0+a_1z+a_2z^2+ \cdots$$
{\em $L$-regular} if the sequence $(a_0,a_1,\cdots) \in \overline{L}$.
\end{definition}

\begin{proposition}[Coefficient language]\label{proposition:coefficient_language}
Let $p_i$ for $1\le i\le m$ be a finite set of complex polynomials, and define $K_z$ to be the
semigroup generated by contractions $$f_i:x \mapsto zx + p_i(z)$$
Then there is a regular language $L$ in a finite alphabet of complex numbers
so that a power series of the form
$$e(z):=a_0+a_1z+a_2z^2+\cdots$$
is $L$-regular if and only if $e\in \partial K_z$; that is $e$ is an infinite composition of the generators $f_i$, thought of as a function in $z$. 
\end{proposition}
\begin{proof}
The effect of $f_i$ on some element of $\partial K_z$ is to shift the sequence by one (the $x \mapsto zx$ part)
and to add the coefficients of $p_i$ to the first $d_i+1$ coefficients, where $d_i$ is the degree of $p_i$.
Introduce the notation
$$p_i(z) = b_{i,0} + b_{i,1}z + \cdots + b_{i,d_i}z^{d_i}$$
and pad coefficients up to $b_{i,d}$ where $d=\max_i d_i$ by defining $b_{i,j}=0$ if $d_i<j\le d$.
Then if we let $e=f_{s_1}f_{s_2}f_{s_3}\cdots$ be an arbitrary element of $\partial K_z$, the
$n$th coefficient $a_n$ of the power series expansion of $e(z)$ is given by the formula
$$a_n = b_{s_n,0} + b_{s_{n-1},1} + \cdots + b_{s_{n-d},d}$$
provided $n\ge d$, and for $n<d$ we simply omit the terms $b_{s_{n-i},d}$ for $n-i<0$ and $i\le d$.
This coefficient depends only on the last $d+1$ letters visited in order, and a finite state
automaton can store this information as a vertex. 

Explicitly, we build a finite graph with $(m^{d+1}-1)/(m-1)$ vertices in bijection with words
of length at most $d$ in the $f_i$, and with an edge from each vertex
corresponding to the word $u$ to the vertex corresponding to $v$, with the edge
labeled $f_j$ if $uf_j$ has $v$ as a suffix of length $\min(d,|uf_j|)$. 

Now at each vertex associated to a word $u$ of length $d'\le d$ of the form
$u=s_1s_2\cdots s_d'$, put the coefficient 
$$a(u):=b_{s_{d'},0} + b_{s_{d'-1},1} + \cdots + b_{s_{d'-d},d}$$

We now relabel the edges in such a way that the new label on each edge is equal to the
coefficient at the vertex it points to. The resulting directed graph is a {\em nondeterministic}
finite state automaton in a finite alphabet (the alphabet of possible coefficients), and the
set of possible edge paths is some language $L$. It
is a standard theorem in the theory of automata due to Kleene--Rabin--Scott (see
\cite{ECHLPT}, Thm.~1.2.7) that there is a {\em deterministic} finite state
automaton in the same alphabet recognizing $L$; this means (by definition) that $L$ is regular. 
Moreover by construction, $\overline{L}$ is precisely the language of coefficient sequences.
\end{proof}

We refer to the language of coefficient sequences as the {\em coefficient language} of the
parameterized family $K_z$, and denote it $L(K_z)$. In the special case that $K_z$ is generated by two elements $p_1,p_2$,
then at least on the subset where $p_1(z) \ne p_2(z)$, the semigroup $K_z$ is conjugate to
the semigroup generated by $f:x \mapsto zz$, $g:x \mapsto z(x-1)+1$ that we have been studying up to now. 

\begin{question}
Which regular languages in a finite alphabet  arise as
$L(K)$ for some $K$?
\end{question}

\begin{example}[Differences]\label{example:differences}
Let $K_z$ be a holomorphic family of semigroups of similarities, parameterized by $z$, whose
IFS has coefficient language $L(K_z)$. The set of differences $DL(K_z):=\lbrace a - b \text{ such that } a,b \in L(K_z)\rbrace$
is of the form $L(DK_z)$ for a suitable holomorphic family $DK_z$.
\end{example}

In Section~\ref{section:differences} we illustrate this difference operation 
in the context of our 2-generator
IFSs, obtaining a sequence of ``iterated Mandelbrot sets'', of which $\SetB$ and $\SetA$ are the
first two terms.


\section{Topology and geometry of the limit set}
\label{section:topology_of_lambda}

In this section we establish basic facts about the geometry and topology of
$\Lambda_z$, establishing quantitative versions of the fundamental dichotomy that
either $\Lambda_z$ is (path) connected, or $\Lambda_z$ is a Cantor set and $G_z$ is Schottky.
These facts lead to an explicit algorithm (essentially due to Bandt) to (numerically)
certify that a particular $G_z$ is Schottky. It is important to describe this algorithm and
its justification in some detail for several reasons. Firstly, this algorithm powers our
program {\tt schottky}, which provided numerical certificates for many of the assertions
we make in this paper (and produced most of the pictures!). Secondly, understanding the
{\em theoretical} behaviour of this algorithm, were it run on an ideal computer for infinite time,
leads to some of the key theoretical insights that underpin our main theorems.

\subsection{Constructing $\Lambda_z$}

Recall that $\Sigma$ is the set of all finite words in $f$ and $g$. For each 
$n\in\mathbb{N}$ define $\Sigma_n$ to be the set of words of length $n$. 

Because $\Lambda_z$ is minimal, for any $p\in\Lambda_z$,
\[
\Lambda_z=\overline{\bigcup_{n} \Sigma_n(z,p)}.
\]
Furthermore, the limit set is well-approximated by $\Sigma_n(z,p)$ for any 
$p$ which is close to $\Lambda_z$:
\begin{lemma}\label{lemma:orbits_are_uniformly_dense}
Let $p\in \C$.  Then $\Lambda_z \subseteq N_\delta(\Sigma_n(z,p))$ where
\[
\delta=|z|^n\left(\frac{|z-1|}{1-|z|} + d(p,\Lz)\right)
\]
\end{lemma}
\begin{proof}
Let $x \in \Lz$ be such that $d(p,x) = d(p,\Lz)$.  We can write $x = \pi(u,z)$ for $u \in \partial\Sigma$.  
Now let $y \in \Lz$ be given, and write $y = \pi(v,z)$ for $v \in \partial\Sigma$.  Let $v_n \in \Sigma_n$ 
be the prefix of $y$ of length $n$.  Consider $w = \pi(v_nu,z) \in \Lz$.  Note $w = v_n(z,x)$, so 
$d(v_n(z,p),w) = d(v_n(z,p),v_n(z,x)) = |z|^nd(p,x)$.  
By Lemma \ref{lemma:distance_estimate},
$d(w,y) \le |z|^n |z-1|/(1-|z|)$, so by the triangle inequality,
\[
d(v_n(z,p),y) \le d(v_n(z,p),w) + d(w,y) \le \delta
\]
\end{proof}

Let $D_z$ be any compact set containing $\Lz$ 
(for example a disk of radius $|z-1|/2(1-|z|)$ centered at $1/2$).   Let 
$D_n = \Sigma_n(z,D_z)$; this is a union of $2^n$ copies of $D_z$ scaled by the 
factor $|z|^n$.
We can construct $\Lz$ as a descending intersection:
\begin{lemma}\label{lemma:intersection_of_disks} We have
\[
\Lz = \bigcap_n D_n.
\]
\end{lemma}
\begin{proof}
Observe that $\bigcap_n D_n$ is a compact, nonempty invariant set. Since $\Lz$ is the
unique such set, they must be equal.
\end{proof}


\begin{figure}[htp]
\begin{center}
\includegraphics[scale=0.23]{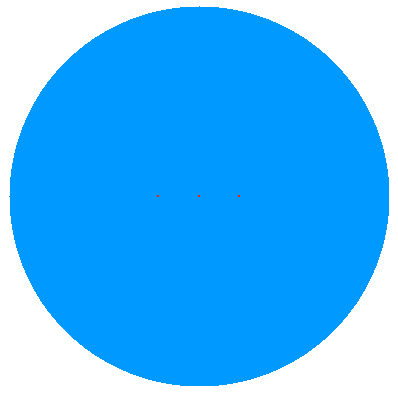}
\includegraphics[scale=0.3]{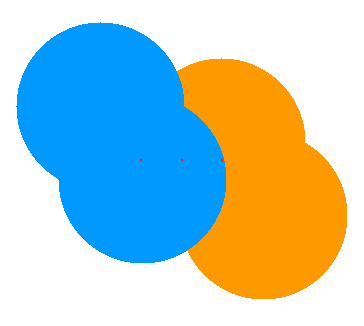}
\includegraphics[scale=0.3]{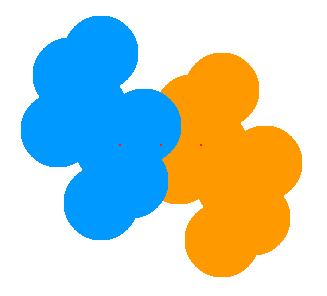}

\includegraphics[scale=0.3]{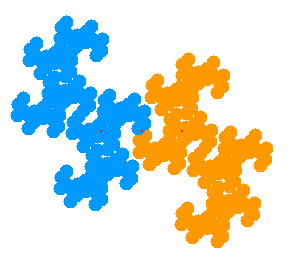}
\includegraphics[scale=0.3]{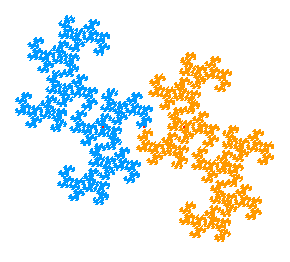}
\includegraphics[scale=0.3]{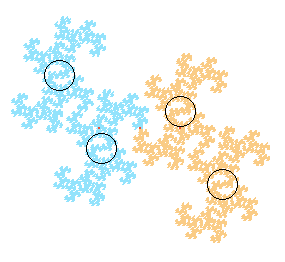}
\end{center}
\caption{Constructing $\Lz$ by intersecting the 
unions of disks $D_n$.  The bottom 
right picture indicates how $\Lz$ decomposes 
as a union of $4$ copies of itself centered at the 
indicated circles.}
\label{figure:intersection_of_disks}
\end{figure}


\subsection{Connectivity}

\begin{lemma}\label{lemma:disconnected_Cantor}
The following are equivalent:
\begin{enumerate}
\item{$\Lz$ is disconnected;}
\item{$\Lz$ is a Cantor set; or}
\item{$f\Lambda_z \cap g\Lambda_z$ is empty.}
\end{enumerate}
Moreover, any of these conditions is implied by $G_z$ Schottky.
\end{lemma}
\begin{proof}
The implications (3) $\to$ (2) $\to$ (1) are obvious, and (1) $\to$ (3)
is a standard result in the theory of IFS. For a proof (in exactly this
context) see \cite{Bousch1}, p.2 (alternately, it follows from the estimates in
Lemma~\ref{lemma:short_hop}).

$G_z$ Schottky immediately proves (3), since if $D$ is a good disk for $G_z$, then
$D$ contains $\Lambda$, but then $fD$ and $gD$ contain $f\Lambda_z$ and
$g\Lambda_z$ and are disjoint by the definition of a good disk.
\end{proof}

Lemma~\ref{lemma:disconnected_Cantor} implies that Schottky semigroups have disconnected
limit sets. The next Lemma, although elementary, is the key to proving the converse.
See Figure~\ref{figure:short_hop} for an illustration of the paths produced 
by Lemma~\ref{lemma:short_hop}.

\begin{lemma}[Short Hop Lemma]\label{lemma:short_hop}
Suppose that $f\Lambda_z$ and $g\Lambda_z$ contain points at distance $\delta$ apart.
Then the $\delta/2$ neighborhood of $\Lambda_z$ is path connected.
\end{lemma}
\begin{proof}
Since $|z|<1$ there is some $n$ so that for any two $e,e' \in \partial \Sigma$ with
a common prefix of length at least $n$ we have $d(\pi(e,z),\pi(e',z))<\delta$.

Suppose $v,v'$ are words of length $i$. Write $v\approx_i v'$ if there are right-infinite
words $u,u'$ with prefixes $v,v'$ such that $d(\pi(u,z),\pi(u',z))<\delta$. Then define
$\sim_i$ to be the equivalence relation generated by $\approx_i$.

We claim that for all $i$ the equivalence relation $\sim_i$ has a single equivalence
class; i.e.\/ that {\em any} two words of length $i$ can be joined by a sequence of
words of length $i$ related by $\approx_i$. Evidently $f \approx_1 g$ since we can
choose right-infinite words $fu$ and $gu'$ such that $\pi(fu,z)$ and $\pi(gu',z)$ are
points in $f\Lambda_z$ and $g\Lambda_z$ respectively realizing $d(\pi(fu,z),\pi(gu',z))\le \delta$.

If $v\sim_i v'$ for all words $v,v'$ of length $i$, then $fv\sim_{i+1} fv'$ and
$gv\sim_{i+1} gv'$ for all words $v,v'$ of length $i$, since $v\approx_i v'$ implies
$fv \approx_{i+1} fv'$ and $gv\approx_{i+1} gv'$. But if $fu$ and $gu'$ are as above, and
$w,w'$ are the initial words of $fu$ and $gu'$ of length $i+1$ then $w\approx_{i+1} w'$.
So the claim is proved for all $i$, by induction.

Taking $i=n$ and using the defining property of $n$ as above proves the lemma.
\end{proof}

\begin{figure}[htb]
\begin{center}
\includegraphics[scale=0.2]{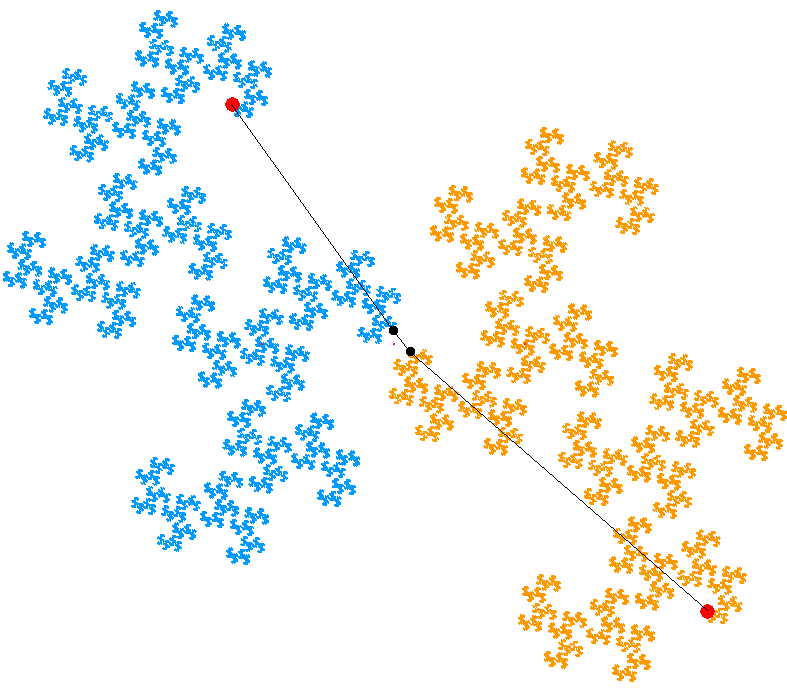}~
\includegraphics[scale=0.2]{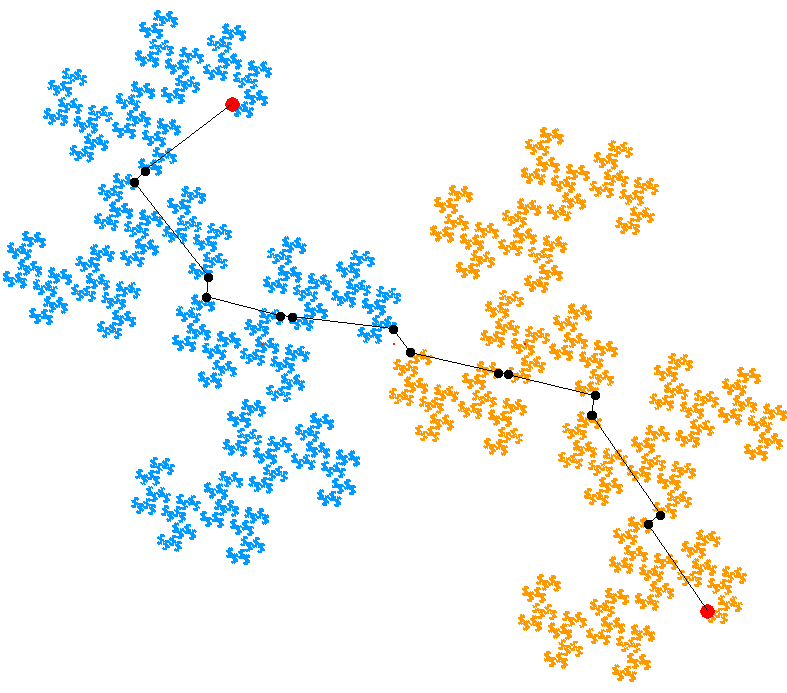}~
\includegraphics[scale=0.2]{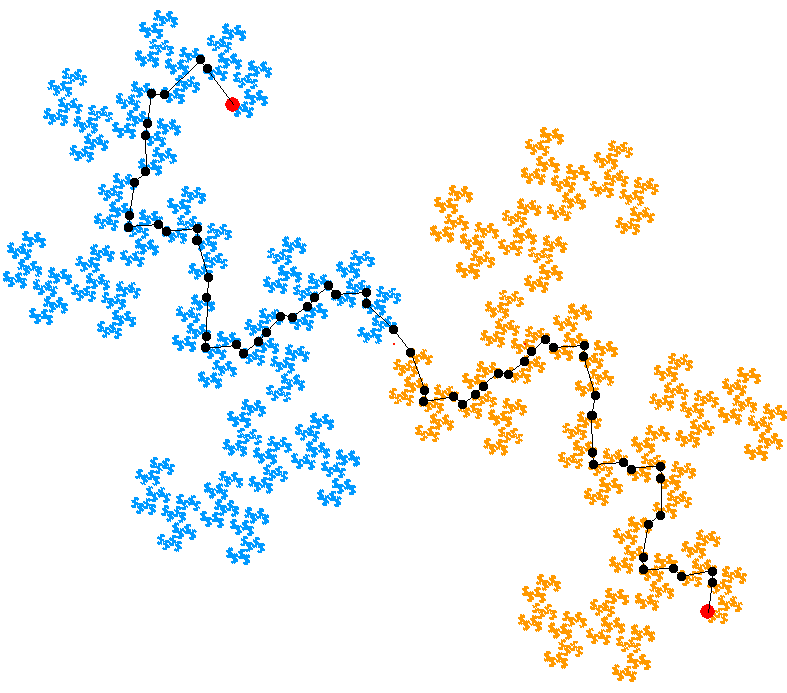}
\end{center}
\caption{Creating a path between the red points which lies entirely 
within the $\delta/2$ neighborhood of $\Lz$ by recursively ``jumping'' 
across the pair of points in $f\Lz$ and $g\Lz$ which are closest, as 
explained in the Short Hop Lemma (Lemma~\ref{lemma:short_hop}).}
\label{figure:short_hop}
\end{figure}


\begin{theorem}[Disconnected is Schottky]\label{theorem:disconnected_is_Schottky}
The semigroup $G_z$ has disconnected $\Lambda_z$ if and only if $G_z$ is Schottky.
\end{theorem}
\begin{proof}
It remains to prove that if $\Lambda_z$ is disconnected, then $G_z$ admits a good disk.
Since $\Lambda_z$ is disconnected, by Lemma~\ref{lemma:disconnected_Cantor} the
distance from $f\Lambda_z$ to $g\Lambda_z$ is some positive number $\delta$. By
Lemma~\ref{lemma:short_hop} it follows that the closed $\delta/2$ neighborhood 
$\overline{N}_{\delta/2}(\Lz)$ of $\Lambda_z$
is path connected. So $\overline{N}_{|z|\delta/2}(f\Lz)$ and $\overline{N}_{|z|\delta/2}(g\Lz)$
are path connected.

Choose some $\epsilon$ with $|z|\epsilon < \delta/2 < \epsilon$.
Let $L_n = \Sigma_n(z,0)$.
By Lemma~\ref{lemma:orbits_are_uniformly_dense}, there is an $n$ such that 
$\overline{N}_{\delta/2}(\Lambda_z) \subseteq \overline{N}_{\epsilon}(L_n)$.
Let $E = \overline{N}_{\epsilon}(L_n)$.  By definition, each connected 
component of $E$ contains a point in $\Lz$, but 
$\Lz \subseteq \overline{N}_{\delta/2}(\Lambda_z)$, which is path connected, so 
there can only be a single connected component of $E$ containing $\Lz$, so $E$ 
is path connected.  
Since $E$ is a finite union of closed disks which is path connected, 
it is homeomorphic to a disk with finitely many subdisks removed.

Furthermore, $fE$ and $gE$ are unions of round disks of radius $|z|\epsilon$ 
around points of $f\Lz$ and $g\Lz$, and
therefore are contained in $\overline{N}_{\delta/2}(\Lz)\subseteq E$. Because $|z|\epsilon < \delta/2$, 
$fE$ and $gE$ are disjoint. 

We now show that we can fill in the ``holes'' in $E$ (if any) and obtain a good disk.
By construction $E$ has finitely many holes, so there is some hole of least diameter
with boundary component $\gamma$. But then $f\gamma$ and $g\gamma$ have diameter
strictly less than $\gamma$, and are contained in the interior of $E$,
so that they must bound subdisks of $E$. So it follows that we can add to $E$ the subdisk
bounded by $\gamma$ to obtain a new closed set $E'$ with $fE'$ and $gE'$ disjoint and
contained in the interior of $E$. Add the bounded complementary components in this way
one by one until we obtain a closed topological disk $D$ with $fD$ and $gD$ disjoint
and contained in the interior of $D$. In other words, $D$ is good for $\Gz$, so that $\Gz$
is Schottky.
\end{proof}

\subsection{An algorithm to certify that \texorpdfstring{$\Lambda_z$}{Lambda(z)} is disconnected}
\label{section:numerical_schottky}

In this section, we describe a fast and practical algorithm to certify 
that the limit set $\Lz$ is disconnected for a given parameter $z$ (equivalently, to certify
that $\Gz$ is Schottky). Since this condition is open in
$z$, a careful analysis of this algorithm certifies that $\Lz$ is
disconnected on a definite open subset of parameter space. Giving a rigorous numerical
certificate that $\Lz$ is {\em connected}, especially one valid in a definite open
subset of parameter space, is more difficult, and is addressed in Section~\ref{section:interior}.
However practically speaking, the algorithm described in this section can be used to draw
fast and accurate pictures of $\SetA$. The algorithm we describe differs only in inessential
ways from that first discussed by Bandt \cite{Bandt}.

We give some notation.  Let $D_z$ be a round disk centered at $1/2$ with the property that
$fD_z$ and $gD_z$ are both contained in $D_z$; for example, we could take $D_z$ to be a disk of radius
$|z-1|/2(1-|z|)$. Let $D_n=:\Sigma_n(z,D_z)$; i.e.\/ $D_n$ is the union of the images of $D_z$ under
the set of words in $\Sigma$ of length $n$. Inductively, $D_n = fD_{n-1} \cup gD_{n-1}$.
By Lemma~\ref{lemma:intersection_of_disks}, $\Lambda_z = \cap_n D_n$, 
so $\Lz$ is disconnected if and only if $D_n$ is disconnected 
for some $n$.  

\begin{lemma}
$D_n$ is disconnected if and only if $fD_{n-1} \cap gD_{n-1} = \varnothing$.
\end{lemma}
\begin{proof}
Obviously if $fD_{n-1} \cap gD_{n-1} = \varnothing$, then $D_n$ is 
disconnected, so we must only show the converse.
Suppose $D_n$ is disconnected and 
$fD_{n-1} \cap gD_{n-1} \ne \varnothing$.  We can take $n$ to be minimal 
such that $D_n$ is disconnected and 
$fD_{n-1} \cap gD_{n-1} \ne \varnothing$.  Since $n$ is minimal, 
for $n-1$ we have either $D_{n-1}$ is connected or 
$fD_{n-2} \cap gD_{n-2} = \varnothing$.
The latter is impossible, though, because $D_{n-1} \subseteq D_{n-2}$, 
so $fD_{n-1} \cap gD_{n-1} \subseteq fD_{n-2} \cap gD_{n-2}$.  
We conclude that $D_{n-1}$ is connected.
But then $fD_{n-1}$ and $gD_{n-1}$ 
are connected, and $fD_{n-1} \cap gD_{n-1} \ne \varnothing$, 
so $D_n$ is connected, a contradiction.
\end{proof}


Naively, to
check that $fD_{n-1}$ is disjoint from $gD_{n-1}$ would take exponential time, since we need
to check the pairwise distances between elements of two sets, each with $2^{n-1}$ points.
However, there is a great deal of redundacy: if $u$ and $v$ are words of length $n$ starting
with $f$ and $g$ respectively, then if $u(z,D_z)$ is disjoint from $v(z,D_z)$, then $ux(z,D_z)$ is disjoint from
$vy(z,D_z)$ for all words $x,y$. In fact, for any fixed $u$ and $v$ words of length $n$, the
images $u(z,D_n)$ and $v(z,D_n)$ are copies of $D_n$, scaled by $z^n$ and translated relative to each
other by $u(z,1/2) - v(z,1/2)$. Thus the relevant data to keep track of is the set of numbers
$z^{-n}(u(z,1/2) - v(z,1/2))$ ranging over $u,v$ of length $n$ where $u$ starts with $f$ and $v$
starts with $g$, for which $u(z,D_z)$ and $v(z,D_z)$ intersect --- equivalently, for which there is an
inequality $|z^{-n}(u(z,1/2) - v(z,1/2))|<R:=2\,\text{radius}(D_z)$.

This discussion justifies Algorithm~\ref{algorithm:disconnectedness} to test for 
connectedness of $\Lz$.  We briefly explain the recursion in the 
context of the above observations.  
First, the algorithm initializes the set 
$V$ to contain the single number 
\[
z^{-1}(f(z,1/2)-g(z,1/2)) = z^{-1}(z/2 - (z(1/2-1)+1)) = 1-z^{-1}.
\]
Next, recall that 
for any word $u$ of length $n$, we can write $u(z,x) = z^nx + p_u(z)$, where 
$p_u(z)$ is a polynomial in $z$.  Therefore, 
$z^{-n}(u(z,1/2) - v(z,1/2)) = z^{-n}(p_u(z) - p_v(z))$.
So if we are given $\alpha = z^{-n}(u(z,1/2) - v(z,1/2)) = z^{-n}(p_u(z) - p_v(z))$, 
we can compute (for clarity, we write $u1/2$ in place of $u(z,1/2)$):
\footnotesize
\begin{align*}
z^{-(n+1)}(uf1/2 - vf1/2) &= z^{-1}z^{-n}(z^nz/2 + p_u(z) - z^nz/2 - p_v(z)) = z^{-1}\alpha \\
z^{-(n+1)}(ug1/2 - vg1/2) &= z^{-1}z^{-n}(z^n(1-z/2) + p_u(z) - z^n(1-z/2) - p_v(z)) = z^{-1}\alpha \\
z^{-(n+1)}(uf1/2 - vg1/2) &= z^{-1}z^{-n}(z^nz/2 + p_u(z) - z^n(1-z/2) - p_v(z)) = z^{-1}(\alpha +z-1) \\
z^{-(n+1)}(ug1/2 - vf1/2) &= z^{-1}z^{-n}(z^n(1-z/2) + p_u(z) - z^nz/2 - p_v(z)) = z^{-1}(\alpha -z+1)
\end{align*}
\normalsize
So given the set of differences of the form $z^{-n}(u(z,1/2)-v(z,1/2))$ which are less than $R$, 
where $u$ and $v$ may range over all words of length $n$, we can compute the set of differences 
of words of length $n+1$, discarding those which are larger than $R$.

\begin{algorithm}[htpb]
\caption{Disconnected$(z,\text{depth})$}\label{algorithm:disconnectedness}
\begin{algorithmic}
\State $V\gets \lbrace 1-z^{-1} \rbrace$
\State $d\gets 0$

\While{$V\ne \emptyset$ or $d<\text{depth}$}
	\State $W \gets \emptyset$
	\ForAll{$\alpha \in V$}
		\If{$|z^{-1}\alpha|<R$}
			$W \gets W \cup z^{-1}\alpha$
		\EndIf
		\If{$|z^{-1}(\alpha + z-1)|<R$}
			$W \gets W \cup z^{-1}(\alpha+z-1)$
		\EndIf
		\If{$|z^{-1}(\alpha -z +1)|<R$}
			$W \gets W \cup z^{-1}(\alpha-z+1)$
		\EndIf
	\EndFor
	\State	$V \gets W$
	\State $d \gets d+1$

\EndWhile

\If{$V=\emptyset$} 
\State \Return true
\Else
\State \Return false
\EndIf
\end{algorithmic}
\end{algorithm}

If this algorithm returns true, then $\Lz$ is disconnected. If it returns false,
then $\Lz$ might still be disconnected, but this would not be discovered without
increasing the ``depth'' parameter.

Algorithm~\ref{algorithm:disconnectedness} is very fast, and has been implemented in our
program {\tt schottky}, available from \cite{schottky}. 
In practice, we can check connectedness to depths exceeding $60$.  
The algorithm is faster in certain regions than others; in particular, it is quite slow
near the real axis. We follow this point up in Section~\ref{section:whiskers}.

\subsection{Paths in $\Lz$}\label{subsection:Lz_paths}

In this section, we show how to construct paths inside the 
limit set $\Lz$ and show that it is connected if and only 
if it is path connected.  We will not explicitly need 
the results in this section; however, it serves to further 
introduce the structure of $\Lz$, and we will use 
very similar ideas in Section~\ref{section:holes_in_m0}.  
These are not new results and can be derived from the general 
theory of IFSs, but this direct approach is illuminating.
Note that this is essentially a continuous 
version of the Short Hop Lemma~\ref{lemma:short_hop}.

The following construction essentially appears in \cite{Bandt}.
Suppose that $f\Lz \cap g\Lz \ne \varnothing$, 
so there are $u,v \in \partial \Sigma$ with $u_1 = f$ and $v_1 = g$ 
and $\pi(u,z) = \pi(v,z)$.  Then for any $a,b \in \partial \Sigma$, 
we will construct a continuous path within $\Lz$ between $\pi(a,z)$ and $\pi(b,z)$.

\begin{figure}[ht]
\begin{center}
\includegraphics[scale=0.5]{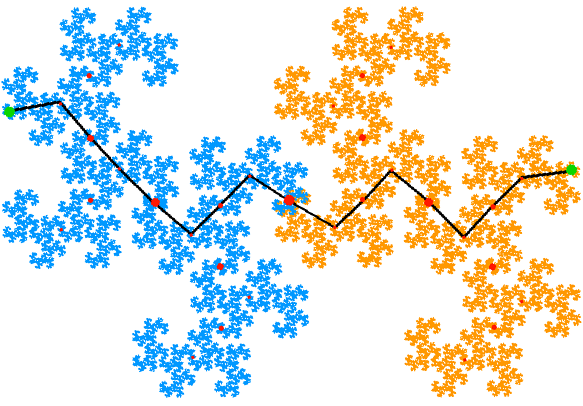}
\end{center}
\caption{Given the words $x,y \in \partial G$ such that 
$\pi(z)(x) = \pi(z)(y)$, we show an approximation 
of a path in $\Lz$ between $\pi(z)(a)$ and $\pi(z)(b)$ 
for a given $a,b \in \partial G$.  The large red disk in the middle indicates a 
point in the intersection $f\Lz \cap g \Lz$, and the other red
disks show the image of this point under words of length less than or equal to $4$.
This method is completely analogous to Figure~\ref{figure:short_hop}.}
\label{figure:path_in_lambda}
\end{figure}

First, let us narrate Figure~\ref{figure:path_in_lambda} to explain 
the construction graphically.  Suppose that the IFS 
takes the disk of radius $R$ inside itself.  We will use this fact to 
bound distances between points.
Suppose that $a_1 = f$ and $b_1 = g$.  Note $\pi(a,z)$ and 
$\pi(b,z)$ are distance at most $2R$. 
Consider the pair of words $(u,v)$.  By assumption 
Since $x_1 = f = a_1$ and $y_1 = g = b_1$, note that both $\pi(a,z)$ 
and $\pi(u,z)=\pi(v,z)$ lie in $f\Lz$, so they are distance 
at most $2|z|R$ apart.  Similarly, $\pi(b,z)$ and $\pi(u,z)=\pi(v,z)$ 
lie in $g \Lz$, so they are also distance at most $2|z|R$ apart.
That is, the point $\pi(u,z)=\pi(v,z)$ coarsely interpolates 
between $\pi(a,z)$ and $\pi(b,z)$.  Next consider the words 
$a$ and $u$.  For illustrative purposes, suppose 
$a_2 = f$ and $u_2 = g$.  Then the pair $(fu,fv)$ interpolates 
between $a$ and $u$: because $fu$ agrees with $a$ to depth $2$, and $fv$ agrees with $v$ to depth $2$, 
we have $\pi(fu,z)=\pi(fv,z)$, and 
\[
|\pi(a,z) - \pi(fu,z)| < 2|z|^2R\quad \text{and}\quad |\pi(fv,z) - \pi(u,z)| < 2|z|^2R
\]
We can continue inductively producing 
points in $\Lz$ coarsely between points in our path.  
Figure~\ref{figure:path_in_lambda} shows the coarse path which is the 
result of stopping the construction after $4$ steps.
In the limit, we will have  
a continuous path from the dyadic rationals into $\Lz$, and because 
$\Lz$ is closed, this path extends continuously to $[0,1]$.

We now describe the construction precisely.
Given $u,v$ as above, we define an interpolation function on infinite words 
$\Phi_{(u,v)}:\partial \Sigma \times \partial \Sigma \to \partial \Sigma \times \partial \Sigma$, as follows.
Given two right-infinite words $s,t$, we may rewrite them as $s=ws'$ and $t=wt'$, 
where $w$ is the maximal common prefix of $s$ and $t$.  Then
\[
\Phi_{(u,v)} (s,t) = \left\{ \begin{array}{ll} (wu,wv) & \textnormal{if $s'_1 = f$ (and thus $t'_1 = g$)} \\
                                        (wv,wu) & \textnormal{if $s'_1 = g$ (and thus $t'_1 = f$)} 
             \end{array} \right.
\]
Note that $\pi(wu,z)=\pi(wv,z)$ by Lemma \ref{lemma:word_composition}. Furthermore if the maximal common prefix of $s,t$ has length $n$, i.e. $|w|=n$, then 
the maximal common prefix of $s$ and $(\Phi_{(u,v)}(s,t))_1$ (denoting the first coordinate) is $n+1$; similarly, the maximal common prefix of $t$ and $(\Phi_{(u,v)}(s,t))_2$ is also $n+1$. 

Now we define a set $W\subseteq\partial \Sigma\times\partial\Sigma$.  The set $W$ will be indexed by 
the dyadic rational numbers, with $W_r$ denoting the element at $r$; that is $W_r=(W_{r,1},W_{r,2})$. 
First set $W_0 = (a,a)$ and $W_1 = (b,b)$.  Then recursively define
\[
W_{k2^{-i-1}+(k+1)2^{-i-1}}=\Phi_{(u,v)}(W_{k2^{-i},2}, W_{(k+1)2^{-i},1})
\]
In other words, to get the pair between $k2^{-i}$ and $(k+1)2^{-i}$, 
apply the interpolation function $\Phi_{(u,v)}$ to the second word at $k2^{-i}$ 
and the first word at $(k+1)2^{-i}$. Observe that $\pi(W_r,z)$ is well-defined
because $\pi(\Phi_{(u,v)}(\cdot)_1,z)=\pi(\Phi_{(u,v)}(\cdot)_2,z)$.

\begin{lemma}\label{lemma:interpolate_lambda_props}
Suppose that $$k2^{-i} \le r_1,r_2 \le (k+1)2^{-i}$$
Then for the construction above, we have the estimate
$$|\pi(W_{r_1},z) - \pi(W_{r_2},z)| < 2\frac{|z|^{i+1}}{|1-z|}$$
\end{lemma}
\begin{proof}
By construction, $W_{k2^{-i},2}$ and $W_{(k+1)2^{-i},1}$ have a common 
prefix $w$ of length $i$, and thus for all $r$ with $k2^{-i} \le r \le (k+1)2^{-i}$, 
the pair of words comprising $W_r$ also has the prefix $w$.  Therefore, the difference 
$\pi(W_{r_1},z) - \pi(W_{r_2},z)$ is given by a power series in $z$ with coefficients 
in $\{-2,0,2\}$ whose first nonzero coefficient has degree at least $i+1$. The estimate
follows.
%
\end{proof}


\begin{proposition}\label{proposition:path_in_lambda}
 Suppose $u,v \in \partial \Sigma$ with $u_1 = f$ and $v_1 = g$ 
and $z$ is such that $\pi(u,z) = \pi(v,z)$.  Let $a,b \in \partial \Sigma$, and let $\mathrm{Dy}$ be the set of dyadic rational numbers. 
Using $u,v,a,b$ as input, construct $W$ as above, and define the map 
$w:\mathrm{Dy}\to \C$ given by $w:r\mapsto \pi(W_r,z)$. 
Then $w$ extends continuously to $[0,1]$, 
and satisfies
\[w(0) = a,\quad  w(1) = b,\quad\text{and}\quad w([0,1]) \subseteq \Lz.
\]  
Hence, $\Lz$ is path connected iff it is connected 
iff $f\Lz \cap g\Lz \ne \varnothing$.
\end{proposition}
\begin{proof}
It follows from Lemma~\ref{lemma:interpolate_lambda_props} that $w:\mathrm{Dy}\to \C$ is continuous, and its image is contained in $\Lz$, so $w$ extends continuously as a map $[0,1]\to \Lz$. 

To get the last assertion, observe that we have shown (3) $\Rightarrow$ (1); 
the other two implications (1) $\Rightarrow$ (2) and 
(2) $\Rightarrow$ (3) are obvious.
\end{proof}

\section{Limit sets of differences}\label{section:differences}

\subsection{Differences in $\Lz$}

As we saw in Section~\ref{section:topology_of_lambda}, the topology and 
geometry of $\Lz$ is controlled by the intersection $f\Lz \cap g\Lz$.
In particular, $f\Lz \cap g\Lz \ne \varnothing$ if and only if 
$\Lz$ is connected.  Said another way, the limit set is 
path connected if and only if $0$ lies in the set of differences
between points in $\Lz$.  It turns out that 
the set of differences between points in $\Lz$ is itself the limit 
set of an IFS.  This limit set features prominently in \cite{Solomyak} and in 
Section~\ref{section:renormalization}.

Define $\Gamma_z$ to be the limit set of the IFS generated 
by the three functions
\[
x\mapsto z(x+1)-1 \qquad x\mapsto zx \qquad x\mapsto z(x-1)+1
\]
That is, the IFS generated by dilations by $z$ centered at the 
points $-1,0,1$. 

\begin{lemma}\label{lemma:ifs_of_diffs}
We have $\Gamma_z = \{a-b\, | \, a,b\in\Lambda_z\}$.
\end{lemma}
\begin{proof}
Each point in $\Lambda_z$ is given by a power 
series in $z$ and associated with an infinite word in 
$\partial \Sigma$.  Thus, the set of differences of 
points in $\Lambda_z$ is associated with pairs of 
infinite words.  Given two words $x,y \in \partial \Sigma$, 
it is straightforward to compute the power series giving the 
value $\pi(x,z) - \pi(y,z)$ recursively, as follows.
Suppose that $x$ begins with $f$, so $x = fx'$ and $y$ begins with $g$, 
so $y = gy'$.  Then 
\[
\pi(x,z) - \pi(y,z) = f(\pi(x',z)) - g(\pi(y',z)) = 
z(\pi(x',z) - \pi(y',z) +1) - 1
\]
In other words, the difference associated to the pair of words $fx'$ and
$gy'$ is obtained from the difference associated to the pair of words
$x'$ and $y'$ by the transformation
$d\mapsto z(d+1)-1$. Similarly, prefixing the pair of words $(x',y')$ with the
pair of letters $(f,f)$, $(g,g)$, $(g,f)$ 
transforms the differences by
\[
d \mapsto zd, \quad d \mapsto zd, \quad z \mapsto z(d-1)+1,
\]
respectively. Note that two of these transformations are the same.
Therefore, the limit set of the semigroup 
generated by these three transformations
is precisely the set of differences in $\Lz$.  
But that limit set is $\Gamma_z$.
\end{proof}

Notice that the proof of Lemma~\ref{lemma:ifs_of_diffs} effectively 
shows that the set of differences of any IFS generated from a 
regular language as in Section~\ref{section:regular} is itself an IFS 
generated from a regular language.

\subsubsection{Iterated Mandelbrot sets}

The set of differences between points in $\Lz$ is $\Gamma_z$.  
We can iterate this procedure by taking the set of differences in $\Gamma_z$, 
and so on.
Let $\Gamma^k_z$ be the 
limit set of the IFS generated by $\{f_{-k},\ldots,f_k\}$, 
where $f_i$ is a dilation centered at $i$.
The set of differences 
of points in $\Lz$ is $\Gamma^1_z$.  
\begin{lemma}\label{lemma:diffs_in_lambda}
The set of differences between points in $\Gamma^{2^k}_z$ is 
$\Gamma^{2^{k+1}}_z$.
\end{lemma}
\begin{proof}
This is just a computation in the generators analogous to the proof of 
Lemma~\ref{lemma:ifs_of_diffs}.  We note
$(z(x-m)+m)-(z(y-n)+n) = z((x-y)-(m-n)) + (m-n)$, so 
acting by the dilations at $m$ and $n$ on two 
points acts on their difference by the dilation at $m-n$.
\end{proof}

If we define
\[
\SetA^k := \{z \in \D^*\, |\, \Gamma^k_z \textnormal{ is connected}\}
\]
and 
\[
\SetA^k_0 := \left\{z \in \D^* \, | \, 0 \in f_1\Gamma^k_z \right\},
\]
then
\begin{proposition}\label{proposition:iterated_SetA_SetB}
$\SetA^{2^k} = \SetA^{2^{k+1}}_0$.
\end{proposition}
\begin{proof}
To see that 
$\SetA^{2^{k+1}}_0 \subseteq \SetA^{2^k}$, suppose that 
$0 \in f_1\Gamma_z^{2^{k+1}}$, so
there is a pair of generators $f_n$, $f_{n+1}$ 
of $\Gamma_z^{2^{k}}$ such that 
$f_n\Gamma_z^{2^k} \cap f_{n+1}\Gamma_z^{2^k} \ne \varnothing$, 
and thus this holds for all $n$, so the limit set $\Gamma_z^{2^k}$ 
is connected.  Conversely, if 
$\Gamma_z^{2^k}$ is connected, then 
$f_{2^k}\Gamma_z^{2^k}$ intersects  $f_j\Gamma_z^{2^k}$ for some $j$.  
But the images $f_j\Gamma_z^{2^k}$ are translates of the same path 
connected set by multiples of the same vector, 
so it must be that $f_{2^k}\Gamma_z^{2^k}$ intersects 
the translate which is closest, i.e. 
$f_{2^k}\Gamma_z^{2^k}\cap f_{2^k-1}\Gamma_z^{2^k} \ne \varnothing$ .  Hence 
$0 \in f_1\Gamma_z^{2^{k+1}}$.  
\end{proof}
\begin{remark}
Note that the last step in the proof of Proposition~\ref{proposition:iterated_SetA_SetB} is
essentially the same as Bousch's proof of Proposition~\ref{proposition:square_Set_B}.
\end{remark}

In general, if our IFS is a set of dilations by $z$ 
at points $\{c_1,\ldots, c_k\}$, then the IFS which 
generates the differences in our IFS is the set of dilations 
by $z$ with centers at all differences of the $c_i$.  The fact that $f_1$ 
appears in the definition of 
$\SetA_0^k$ (as opposed to $f_i$ for another $i$) is natural because 
the number $1$ is always the generator of the lattice of centers.

\begin{question}
What sequences of sets arise as iterated differences?
What properties do these iterated IFS have? 
\end{question}

\section{Interior points in \texorpdfstring{$\SetA$}{M}}\label{section:interior}

We have already seen that $\SetA$ contains many interior points; in fact, the entire annulus
$1/\sqrt{2}\le |z|\le 1$ is in $\SetA$. In this section we develop the method of {\em traps}
to certify the existence of many interior points in $\SetA$, and examine the closure of the set of interior points. The result is quite surprising: the closure of
the interior is all of $\SetA$ \dots {\em except} for some subset of
the two real whiskers!

This assertion is Theorem~\ref{theorem:interior_dense} below, which is the affirmation of
Bandt's Conjecture (i.e.\/ Conjecture~\ref{conjecture:Bandt}). 
In Section~\ref{section:holes} these techniques are used to
certify the existence of (infinitely many) small holes in $\SetA$ --- i.e.\/ exotic components
of Schottky space.

\subsection{Short hop paths and Traps}\label{subsection:traps_holes}

In this subsection we give a method to certify the existence of
open subsets of $\SetA$. Abstractly, to certify that $z$ is an interior point of $\SetA$
is to give a proof that $z \in \SetA$ that depends on properties of $z$ which are stable
under perturbation. Showing that $z\in \SetA$ is equivalent to showing that $f\Lambda_z$
intersects $g\Lambda_z$, so our strategy is to show that this intersection is inevitable for
some {\em topological reason} (depending on $z$). Proving that sets intersect in topology
is accomplished by homology (or, more crudely, separation or linking properties). But
the homological properties of $\Lambda_z$ depend on its connectivity, which is what we are
trying to establish! So our method is first to consider precisely chosen {\em neighborhoods}
of $\Lambda_z$ (which may be presumed to be connected for some open set of $z$), and then to
consider homological properties of the configuration of the images of these neighborhoods under
$f$ and $g$ which force an intersection.

The key to our method is the existence of {\em short hop paths} and {\em traps}. 

\begin{definition}[Short hop path]\label{definition:short_hop_path}
Let $p,q \in \Lz$, let $\epsilon>0$ and let $D$ be a disk containing
$p$ and $q$. An {\em $(\epsilon,D)$-short hop path} from $p$ to $q$ is a sequence
$e_0,e_1,\cdots,e_m$ in $\partial \Sigma$ with $\pi(e_0,z)=p$ and $\pi(e_m,z)=q$ so that
$d(\pi(e_i,z),\pi(e_{i+1},z))<\epsilon$ and $\pi(e_i,z)\in D$ for all $i$.
\end{definition}
\noindent The existence of Short Hop Paths is guaranteed by the Short Hop Lemma; in particular,
we have:

\begin{proposition}[Short hop paths exist]\label{proposition:short_hop_paths_exist}
Let $u$ and $v$ be right-infinite words with a common prefix $w$ of length $n$,
and suppose that there are points in $f\Lz$ and $g\Lz$ which are distance
at most $\delta$ apart. Let $D$ be a disk containing the $|z|^n\delta/2$-neighborhood 
of $w\Lz$. Then there is a $(|z|^n\delta,D)$-short hop path from $\pi(u,z)$ to $\pi(v,z)$.
\end{proposition}
\begin{proof}
Let $u=wu'$ and $v=wv'$, and let $D'$ be any disk containing the
$\delta/2$-neighborhood of $\Lz$. By Lemma~\ref{lemma:short_hop} there 
is a $(\delta,D')$-short hop path from $\pi(u',z)$ to $\pi(v',z)$. Now apply $w$ to this short hop path.
\end{proof}

\noindent We now give the definition of a trap:

\begin{definition}[Trap]\label{definition:trap}
Let $D$ be a closed topological disk containing $\Lz$ in its interior.
We say that a pair of words $u,v\in \Sigma$ are a {\em trap} for $(z,D)$ if
the following are true:
\begin{enumerate}
\item{$u$ starts with $f$ and $v$ starts with $g$;}
\item{there are points $p^\pm$ in $u\Lz - vD$ and $q^\pm$ in $v\Lz - uD$
such that for some paths $\alpha \subseteq uD$ with endpoints $p^\pm$ and
$\beta \subseteq vD$ with endpoints $q^\pm$ the algebraic intersection number of
$\alpha$ and $\beta$ is nonzero; and}
\item{there are points in $f\Lz$ and $g\Lz$ within distance $\epsilon$
of each other, where the $\epsilon/2$ neighborhood of $\Lz$ is contained in $D$.}
\end{enumerate}
\end{definition}

The definition of a trap depends on a choice of paths $\alpha$ and $\beta$ which intersect;
but a homological argument shows that the property does not depend on the choice:

\begin{lemma}[Any paths suffice]\label{lemma:any_paths}
Suppose $u,v$ are a trap for $(z,D)$, and let $p^\pm \in u\Lz -vD$
and $q^\pm \in v\Lz - uD$ be as in Definition~\ref{definition:trap}.
Then {\em any} paths $\alpha \subseteq uD$ with endpoints $p^\pm$ and $\beta \subseteq vD$
with endpoints $q^\pm$ must intersect.
\end{lemma}
\begin{proof}
Any two paths $\alpha,\alpha'$ joining $p^\pm$ and contained in $uD$ are freely
homotopic relative to endpoints in the complement of $q^\pm$, and similarly for any two
$\beta,\beta'$ joining $q^\pm$ and contained in $vD$. Thus the classes 
$[\alpha] \in H_1(\C-q^\pm,p^\pm)$ and $[\beta] \in H_1(\C-p^\pm,q^\pm)$ are
well-defined, and therefore so is their intersection product.
\end{proof}

\begin{example}
\label{example:trap}
Figure~\ref{figure:trap} shows a trap in $\Lz$ with $z=-0.43+0.54i$ which is visible
to the naked eye.  We have drawn 
$D_{12}$ for a disk $D$ with $fD, gD \subseteq D$, so it is guaranteed that 
(1) $f\Lz$ and $g\Lz$ are contained inside the blue and orange sets, respectively 
and (2) there are points in $\Lz$ inside every disk drawn.  The computer also runs 
Algorithm~\ref{algorithm:disconnectedness} to verify that $D_{12}$ is connected. 
These facts, and the (visually evident) fact that the highlighted disks satisfy the
linking condition, proves the existence of points $p^{\pm}$ and $q^{\pm}$ 
inside the disks which give a trap.

\begin{figure}[htpb]
\centering
\includegraphics[scale=0.4]{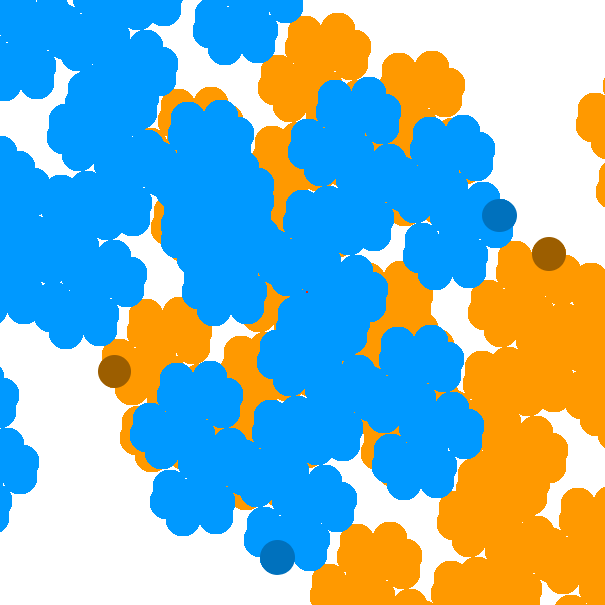}
\caption{An excerpt from $\Lambda$ with $z=-0.43+0.54i$.  
This picture proves the existence of a trap for this parameter, as explained in Example~\ref{example:trap}.}
\label{figure:trap}
\end{figure}
\end{example}

The next Proposition shows that the existence of a trap for $z$ shows
that $z$ is in the interior of $\SetA$.

\begin{proposition}[Traps in $\SetA$]\label{proposition:traps_A}
Let $u,v$ be a trap for $(z,D)$ for some disk $D$. Then $z$ is in the interior of $\SetA$.
\end{proposition}
\begin{proof}
Suppose that $\delta$ is the distance from $f\Lz$ to $g\Lz$. Then 
any two points in $\Lz$ can be joined by a $(\delta,D)$-short hop path.
This is a sequence of points with gaps of size at most $\delta$; such a sequence is
necessarily contained in the $\delta/2$-neighborhood of $\Lz$. 
It follows that $p^+$ can be joined to $p^-$ by
a path $\alpha$ in $uD$, every point on which is within distance $\delta|z|^n/2$ of
some point in $u\Lz$, where $u$ has length $n$.  
Similarly, $q^+$ can be joined to $q^-$ by a path
$\beta$ in $vD$, every point on which is within distance $\delta|z|^m/2$ of
some point in $v\Lz$, where $v$ has length $m$. 
But $\alpha$ and $\beta$ must intersect, by the
defining property of a trap, and Lemma~\ref{lemma:any_paths}. Thus
the distance from $u\Lz$ to $v\Lz$ is at most $\delta|z|^{\min(n,m)}$. But
$\delta|z|^{\min(n,m)}<\delta$ because $n,m\ge 1$.  This is contrary to the definition 
of $\delta$ unless $\delta=0$.
\end{proof}

It is interesting to note that while the existence of a trap for $z$ certifies that
$z$ is in the interior of $\SetA$ and thus that $f\Lz \cap g\Lz \ne \varnothing$, 
it is difficult to use it to algorithmically produce a point of intersection: 
as we decrease $\delta$, the intersecting $\delta$-short hop paths 
need not converge or intersect ``nicely''.

\subsection{Traps are (almost!) dense}

In this subsection we demonstrate the theoretical utility of traps by proving that
traps are dense in $\SetA$ away from the real axis. 
Since traps have nonempty interior, it follows that
the interior of $\SetA$ is dense in $\SetA$, again away from the real axis. This
was conjectured by Bandt in \cite{Bandt}, p.~7 and some partial results were obtained by
Solomyak-Xu \cite{Solomyak_Xu}, who proved the conjecture for points in a neighborhood
of the imaginary axis.

It is interesting that 
the proof depends on a complete analysis of the set of $z$ for which the limit set 
$\Lz$ is convex (Lemma~\ref{lemma:convex_zonohedra}). It turns out that the $z$ with
this property are exactly the union of dyadic ``spikes'' --- points of the form 
$re^{\pi i p/q}$ for coprime integers $p,q$ and $r$ real with $r\ge 2^{-1/q}$. For
$q>1$ these spikes are already in the interior of the solid annulus $r\ge 2^{-1/2}$ which
is entirely contained in $\SetA$; only the real ``whiskers'' protrude from this annulus, and
this is why these are the only points in $\SetA$ which are not in the closure of the interior.

%
%

\begin{definition}[Cell-like, trap-like]\label{definition:cell_trap}
A compact connected subset $X\subseteq \C$ is {\em cell-like} if its complement is connected. 
Let $X$ be cell-like. A complex number $w$ is {\em trap-like} for $X$ if the following hold:
\begin{enumerate}
\item{the union $X \cup (X+w)$ is connected (equivalently, $X$ intersects $X+w$); and}
\item{there are 4 points in the outermost boundary of $X\cup (X+w)$ that alternate between points
in $X - (X+w)$ and points in $(X+w) - X$.}
\end{enumerate}
\end{definition}

\begin{lemma}[Nonconvex cell has trap]\label{lemma:nonconvex_cell_trap}
Let $X$ be cell-like. There there is some trap-like vector $w$ for $X$ if and only if $X$ is nonconvex. 
\end{lemma}
\begin{proof}
If $X$ is convex, then the set of points in the boundary of $X\cup X+w$ in $X-(X+w)$ is connected,
and similarly for those points in $(X+w)-X$, so no $w$ is trap-like.

Conversely, suppose $X$ is nonconvex, and let $\ell$ be a supporting line such that
$\ell \cap X = P \cup Q$ both nonempty (not necessarily connected), 
and separated by an open interval $I$. The existence of such $P$ and $Q$ is guaranteed precisely
by the hypothesis that $X$ is not convex. After composing
with an isometry of the plane, we can assume that $\ell$ is the horizontal axis, 
oriented positively, so that $X$ is on the side of $\ell$ with negative $y$ coordinates.

Let $V$ be a small
open disk in $\C-X$ containing the midpoint of $I$, and choose $v \in V-\ell$ on the side 
of $\ell$ with negative coordinates (i.e.\/ the side containing $X$). 
Let $p$ denote the point of $P$ with biggest $x$ coordinate. Then $w=v-p$ is trap-like for $X$.
See Figure~\ref{figure:nonconvex_trap}.

\begin{figure}[htb]
\begin{center}
\labellist
\small\hair 2pt
 \pinlabel {$\ell$} at -4 45
 \pinlabel {$p$} at 31 48
 \pinlabel {$V$} at 70 51
\endlabellist
\includegraphics[scale=1.5]{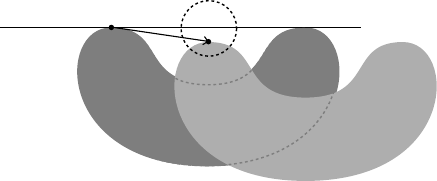}
\caption{A nonconvex compact full set has a trap-like vector.  See the proof of 
Lemma~\ref{lemma:nonconvex_cell_trap}.}
\label{figure:nonconvex_trap}
\end{center}
\end{figure}

This can be seen just by
looking at the foliation of $\C$ by vertical lines $\mu(t)$ with $x$-coordinate $t$, 
and for each line finding the point of $X\cup (X+w)$ with largest $y$-coordinate 
(where this is nonempty). Let $q \in Q$ be arbitrary, let $t$ be the maximum number such that
$\mu(t) \cap X$ is nonempty, and let $r$ be the point on $\mu(t)\cap X$ with largest $y$-coordinate.
Then the four points $p,v,q,r+w$ are the highest points of $X\cup (X+w)$ 
on their respective vertical lines $\mu(t_1),\mu(t_2),\mu(t_3),\mu(t_4)$ for $t_1<t_2<t_3<t_4$,
and alternate between the sets $X$ and $X+w$.
\end{proof}

We would like to apply Lemma~\ref{lemma:nonconvex_cell_trap} to the cell-like set $X_z$ one
obtains from a limit set $\Lz$. Thus, it is important to characterize $z$ for which
the cell-like set $X_z$ obtained from $\Lz$ is convex. 

\begin{lemma}[Convex polygon]\label{lemma:convex_zonohedra}
Let $z$ be in $\SetA$, and let $X_z$ be obtained from $\Lz$ by filling in bounded
complementary components, so that $X_z$ is the smallest cell-like set containing $\Lz$.
Then $X_z$ is convex if and only if $z=re^{\pi i p/q}$ for coprime integers $p,q$ and $r$ real
with $r\ge 2^{-1/q}$, in which case $X_z=\Lz$ is a convex polygon.
\end{lemma}
\begin{proof}

We make use of the following two facts: first, that
$X_z$ has rotational symmetry of order 2 about the point $1/2$; and second, that $\Lz$
is the union of $f\Lz$ and $g\Lz$, obtained from $\Lz$ by scaling by $z$ and
translated relative to each other by $1-z$. Suppose $X_z$ is convex, and consider the collection
of straight segments in the boundary of $X_z$. This collection is nonempty; for, if $p$ is an
extremal point for $X_z$ tangent to the supporting line in the direction $(1-z)/z$ then 
$fp$ and $gp$ are extremal points for $X_z$ tangent to the same supporting line in the direction
$(1-z)$, and then the entire segment between these points is in the line. Now, if
$\sigma$ is a straight segment in the boundary in the direction $w$, then if $w \ne 1-z$,
there is a straight segment in the boundary of the form $f^{-1}\sigma$ or $g^{-1}\sigma$
of length $|z|^{-1}\sigma$. It follows that there is a chain of straight segments
$$\sigma_0,\sigma_1,\sigma_2,\cdots,\sigma_{q-1}$$
where each $\sigma_j$ is in a direction $z^j$ relative to $\sigma_0$, and has length
$|z|^j|\sigma_0|$. But then $f(\sigma_{q-1})$ and $g(\sigma_{q-1})$
must be in the $1-z$ or $z-1$ direction, so that their union is either equal to $\sigma_0$ or
the image of $\sigma_0$ under the symmetry of order 2. It follows
that the argument of $z$ is of the form $\pi p/q$ for some integers $p/q$, and furthermore
that $|z|^q \ge 2$. This proves one direction of the claim.

The converse direction --- that limit sets $\Lz$ for $z$ of this kind really are convex
--- can be seen directly. In fact, these limit sets are zonohedra, the shadows of a linear
semigroup acting in high dimensional space. Let $R_q$ be the parallelepiped in $\R^q$ consisting of
vectors $v:=(v_0,\cdots,v_{q-1})$ whose coordinates satisfy $0 \le v_{pj}\le r^j$ with indices
taken mod $q$.  Note this is simply a rectangular box inside the positive orthant 
with one corner at the origin and edges along the coordinate axes.
Let 
$f:\R^q \to \R^q$ be the composition $f:v \mapsto \sigma^p(rv)$ where
$rv$ means multiply the coordinates of $v$ by $r$, and $\sigma$ is the finite order
rotation $\sigma:v \mapsto (v_{q-1},v_0,\cdots,v_{q-2})$. 
So $f$ rotates and scales the box $R_q$ to another box along the coordinate axes.
Similarly, let $g:v \mapsto \sigma^p(rv) + t$
where $t$ is the vector $(t_0,0,0,\cdots,0)$ for which $r^q+t_0=1$. 
The map $g$ acts in the same way as $f$, except it translates the box up along the 
first coordinate by $t_0$.  The height of the box in the first coordinate is $1$, and 
the heights of the acted-upon boxes $fR_q$ and $gR_q$ in the first coordinate are both $r^q$.
Hence, providing
$r^q\ge 1/2$, $fR_q \cup gR_q = R_q$, so 
the parallelepiped $R_q$ is the limit set of the contracting semigroup $\langle f,g\rangle$.
See Figure~\ref{figure:zonohedron}.

\begin{figure}[htb]
\begin{center}
\labellist
\small\hair 2pt
 \pinlabel {$R^q$} at 162 219
 \pinlabel {$gR_q$} at 248 133
 \pinlabel {$fR_q$} at 240 73
\endlabellist
\includegraphics[scale=0.5]{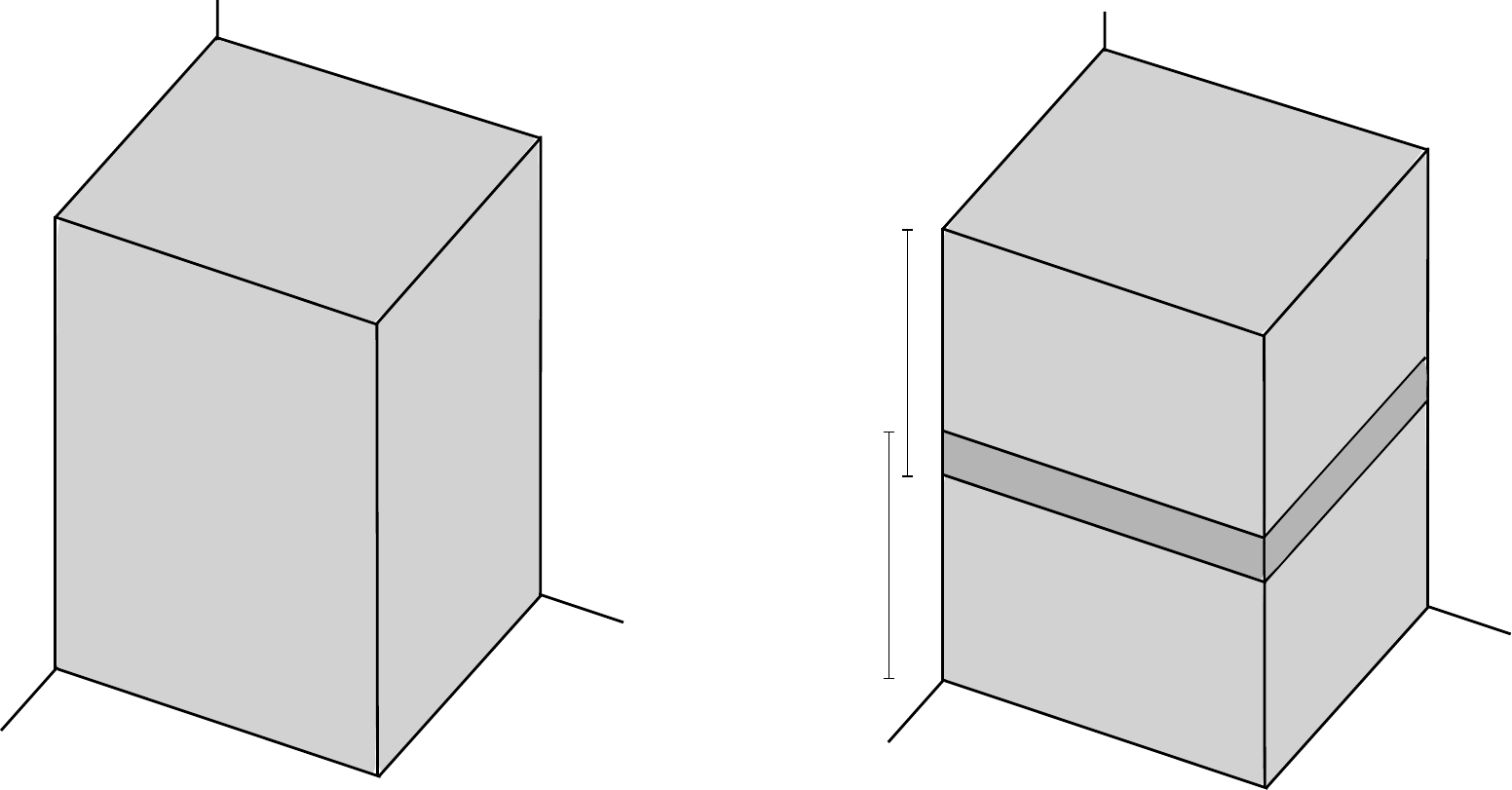}
\end{center}
\caption{The box $R_q$ in the proof of Lemma~\ref{lemma:convex_zonohedra}.  
The projection of $R_q$ to the plane is the limit set $\Lz$, so $\Lz$ is convex 
and, in particular, a zonohedron.}
\label{figure:zonohedron}
\end{figure}

Projecting $R_q$ to the plane so that the vectors $(0,0,\cdots,1,\cdots,0)$ are projected
to the $2q$th roots of unity defines a semiconjugacy from this semigroup to $\Gz$
where $z=re^{\pi i p/q}$, taking $R_q$ to $\Lz$. 
\end{proof}

\begin{example}[Hexagonal limit set]\label{example:hexagon}
Take $z=2^{-1/3}e^{2\pi i/3}\approx 0.396157+0.687364i$
then $\Lambda$ is a hexagon with angles $120^\circ$, and side lengths in the ratio
$1:2^{1/3}:2^{2/3}$. See Figure~\ref{hexagonal_limit_set}.

\begin{figure}[htpb]
\centering
\includegraphics[scale=0.15]{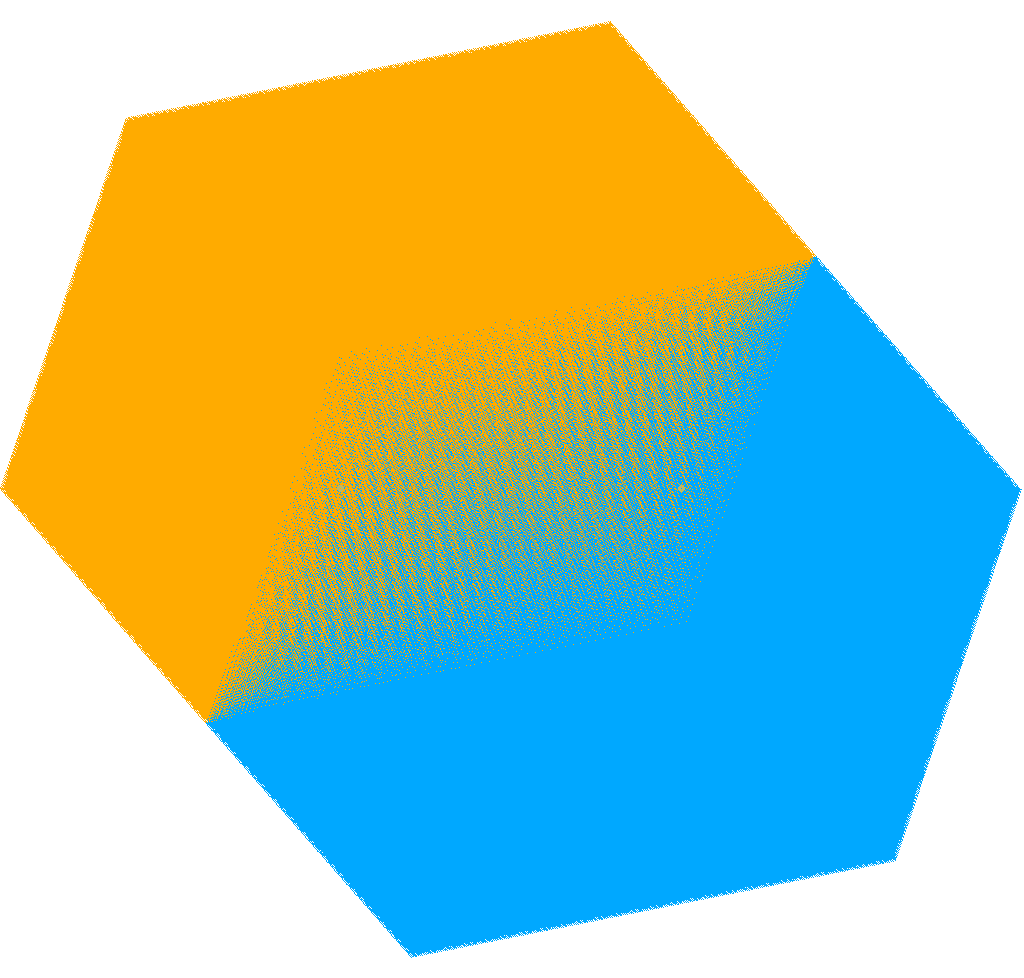}
\caption{A hexagonal limit set for $z=2^{-1/3}e^{2\pi i/3}$.}
\label{hexagonal_limit_set}
\end{figure}

\end{example}


\begin{lemma}[Surjective perturbation]\label{lemma:surjective_perturb}
Let $z_0$ be in $\SetA$, and let $u,v \in \partial \Sigma$ be such that $\pi(u,z_0)=\pi(v,z_0)$.
Then for any $\epsilon>0$ there is $\delta>0$ and integer $M$ so that if $u_m,v_m\in \Sigma$ denote
the prefixes of length $m$ for any $m\ge M$, and $T_m$ denotes the map 
$$T_m:z \mapsto u_m(z,1/2) - v_m(z,1/2)$$
then for any complex $w$ with $|w|<\delta$ there is $z_1$ with $|z_1-z_0|<\epsilon$, and
$T_m(z_1)=w$.
\end{lemma}
\begin{proof}
The functions $T_m$ converge uniformly to the limit $T_\infty:z \mapsto \pi(u,z)-\pi(v,z)$, 
which is holomorphic in $z$. Moreover, this limit could be constant only if $u=f^\infty$ and
$v=g^\infty$, in which case  $z\mapsto \pi(f^\infty,z)$ is identically $0$ and $z\mapsto \pi(g^\infty,z)$ is identically $1$, so $T_\infty(z)\equiv -1$; however, $T_\infty(z_0)=0$. The limit is therefore nonconstant. 
Thus $T_\infty$ takes the ball of radius
$\epsilon$ about $z_0$ to a set containing the ball of radius $2\delta$ about $0$ for some
positive $\delta$, and  the conclusion of the lemma is satisfied for sufficiently big $m$.
\end{proof}
\begin{corollary}\label{corollary:surjective_perturb}
Suppose $u,v\in\partial \Sigma$ and $z_0 \in \SetA$ such that $\pi(u,z_0) = \pi(v,z_0)$.
Then for any complex number $w$, and any positive $\epsilon$, we can 
find an $m$ and a $z_1$ with $|z_1-z_0|<\epsilon$ so that
$$z_1^{-m}\left( u_m(z_1,1/2) - v_m(z_1,1/2) \right)= w.$$
\end{corollary}
\begin{proof}
For sufficiently large $m$, the map $z\mapsto u_m(z,1/2) - v_m(z,1/2)$ is surjective onto a neighborhood of $0$, and the claim follows.   
\end{proof}

We now complete the proof of Bandt's conjecture:

\begin{theorem}[Interior is almost dense]\label{theorem:interior_dense}
The set of interior points is dense in $\SetA$ away from the real axis; that is
\[
\SetA = \overline{\mathrm{int}(\SetA)}\cup (\SetA\cap\R).
\]
\end{theorem}
\begin{proof}
Let $z_0$ be in $\SetA$, and suppose the limit set $\Lambda_{z_0}$ is not convex. Let
$X_{z_0}$ be the region bounded by $\Lambda_{z_0}$, so that $X_{z_0}$ is cell-like. Since $\Lambda_{z_0}$ is not
convex, neither is $X_{z_0}$, and by Lemma~\ref{lemma:nonconvex_cell_trap} there is some $w$
which is trap-like for $X_{z_0}$. Let $p_1,p_2\in X_{z_0}$, and let $q_1,q_2\in X_{z_0}+w$ be the four points from part $(2)$ in Definition \ref{definition:cell_trap}. Since $\partial X_{z_0} \subseteq \Lambda_{z_0}$, the points $p_i,q_i$ 
lie in $\Lambda_{z_0}$. There is an $\epsilon$ so that the closed $\epsilon$-neighborhood of $X_{z_0}$ is connected, 
$p_1,p_2 \in \Lambda_{z_0} - \overline{N}_{\epsilon}(X_{z_0}+w)$, and $q_1,q_2 \in (\Lambda_{z_0}+w) - \overline{N}_{\epsilon}(X_{z_0})$.  Furthermore, 
these conditions are open, so there is a $\delta>0$ such that they hold for $X_z$ for all $z$ with
$|z-z_0|<\delta$.

Now, since $z_0$ is in $\SetA$, there are $u,v\in \partial \Sigma$ starting with $f$ and $g$
respectively with $\pi(u,z_0)=\pi(v,z_0)$.
By Corollary~\ref{corollary:surjective_perturb}, we can find some $u_m,v_m$ prefixes of $u$ and $v$
of length $m$, and $z_1$ with $|z_1-z_0|<\delta$ so that $z_1^{-m}(u_m(z_1,1/2) - v_m(z_1,1 /2)) = w$.
We obtain a trap for $(z_1,D)$ where $D=u_m(z_1)^{-1}(\overline{N}_\epsilon(X_{z_1}))$. This follows from the three conditions above. 

We therefore find interior points of $\SetA$ within
distance $\delta$ of $z_0$. Since $z_0$ was arbitrary, we are done in the case that $\Lambda_{z_0}$
is not convex.

If $\Lambda=\Lz$ is convex and $|z|<2^{-1/2}$ then $z$ is totally
real, by Lemma~\ref{lemma:convex_zonohedra}. If $|z|> 2^{-1/2}$ then we are already
in the interior, by Lemma~\ref{lemma:inner_outer}. This completes the proof.
\end{proof}

\begin{figure}[htpb]
\centering
\includegraphics[scale=0.4]{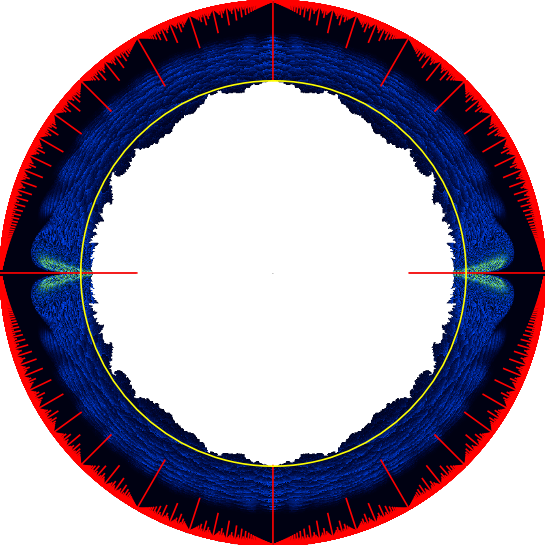}
\caption{The set of $z$ with $\Lz$ convex (in red) overlaid on $\SetA$. The
yellow circle indicates $|z|=2^{-1/2}$.}
\label{convex_on_M}
\end{figure}

Figure~\ref{convex_on_M} shows the set of $z$ with convex $\Lz$ overlaid on $\SetA$.
The picture of $\SetA$ in a neighborhood of the real axis is surprisingly complicated; partial
progress in understanding it was made by Shmerkin-Solomyak \cite{Shmerkin_Solomyak}; 
we describe some of their results in Section~\ref{subsection:whiskers_isolated}, and explain how
the method of  traps can be modified to certify the existence of interior points
in $\overline{\SetA - \R}\cap \R$.

\section{Holes in \texorpdfstring{$\SetA$}{M}}\label{section:holes}

In this section we rigorously certify the existence of holes
in $\SetA$ (i.e.\/ exotic components of Schottky space). 
 Holes in $\SetA$ were first observed experimentally by
Barnsley and Harrington \cite{Barnsley_Harrington}, and the existence of one hole was
rigorously proved by Bandt \cite{Bandt}. However, our technique is quite different from
Bandt's and our proof of the existence of holes is new. Furthermore, we shall show in
Section~\ref{section:renormalization} that our techniques generalize to prove the existence of 
{\em infinitely many} holes in $\SetA$. 

\subsection{An example}

In this section, we give an example of an apparent hole in $\SetA$, an 
intuitive explanation of why the hole is truly an exotic component of 
Schottky space, and the output of our program rigorously certifying 
the hole.  In the next section, we give a careful justification of the 
algorithm.

Figure~\ref{figure:certified_hole_apparent_hole} depicts an apparent collection of holes 
in $\SetA$ centered at $0.459650+0.459654i$.  The diameter of the 
large hole is approximately $0.000002$.

\begin{figure}[htpb]
\centering
\includegraphics[scale=0.4]{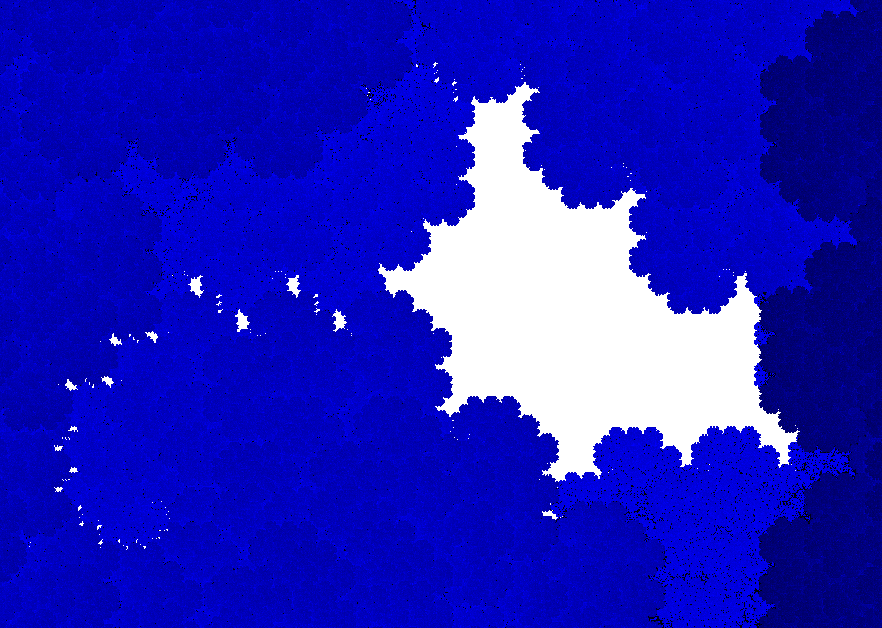}
\caption{Apparent holes in $\SetA$ centered at $0.459650+0.459654i$.}
\label{figure:certified_hole_apparent_hole}
\end{figure}

\begin{figure}[htb]
\begin{center}
\begin{minipage}{0.33\textwidth}
\begin{center}
\includegraphics[scale=0.2]{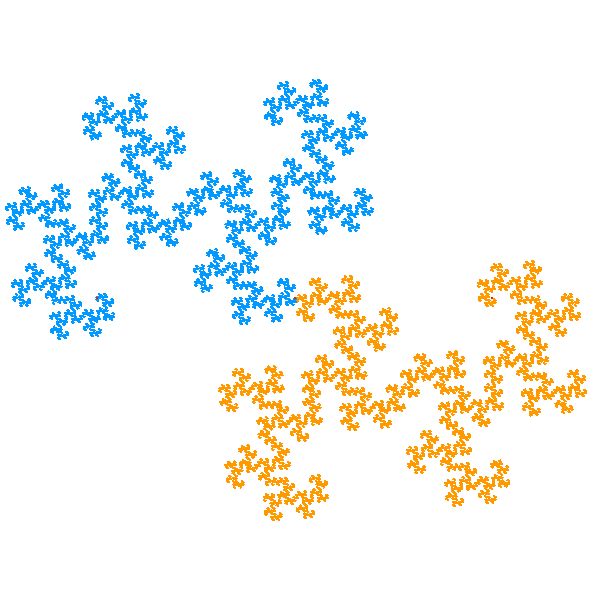}
\end{center}
\end{minipage}
\begin{minipage}{0.65\textwidth}
\includegraphics[scale=0.45]{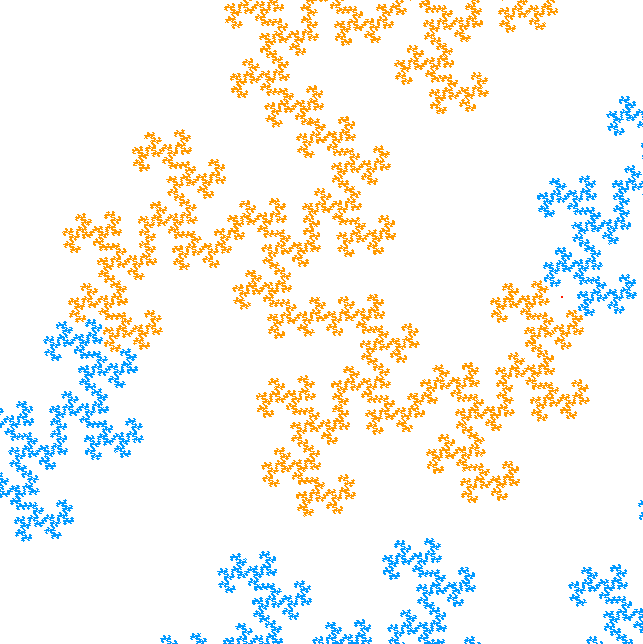}
\end{minipage}
\caption{The limit set for a parameter inside the large hole shown 
in Figure~\ref{figure:certified_hole_apparent_hole}, left, and a zoomed 
view, right.  The two components $f\Lz$ and $g\Lz$ 
(blue and orange, respectively)
cannot be unlinked with a rigid motion without intersecting.}
\label{figure:certified_hole_limit_set}
\end{center}
\end{figure}

The limit set corresponding to a parameter inside the large hole is 
shown in Figure~\ref{figure:certified_hole_limit_set}.  
The sets $f\Lz$ and $g\Lz$ are indeed disjoint, 
but they come very close.  
If one imagines that $f\Lz$ and $g\Lz$ are rigid, connected objects,
then it is clear that one cannot unlink them by a rigid motion
without the two sets intersecting at some intermediate step.  However, movement in 
parameter space does not produce exactly rigid motion of the limit set, so in order 
to prove that this ``hole'' in $\SetA$ is not, in fact, part of the large component 
of Schottky space, we need a more careful analysis.

\begin{figure}[htb]
\begin{center}
\begin{minipage}{0.55\textwidth}
\begin{center}
\includegraphics[scale=0.4]{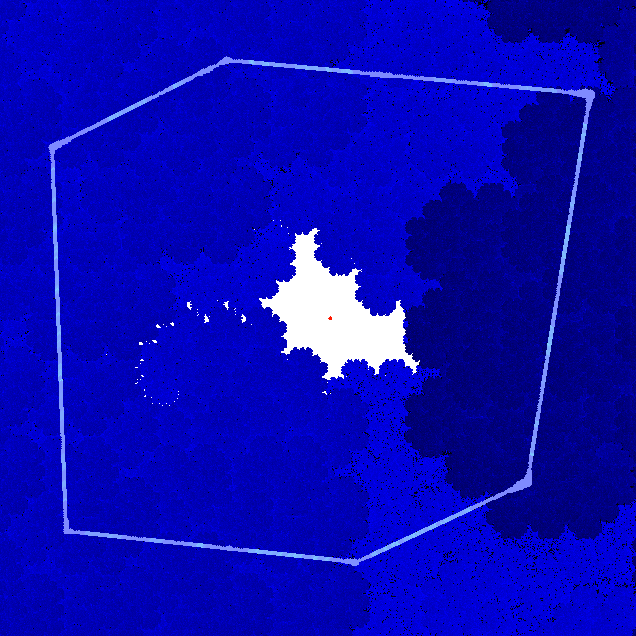}
\end{center}
\end{minipage}
\begin{minipage}{0.33\textwidth}
\begin{center}
\includegraphics[scale=0.2]{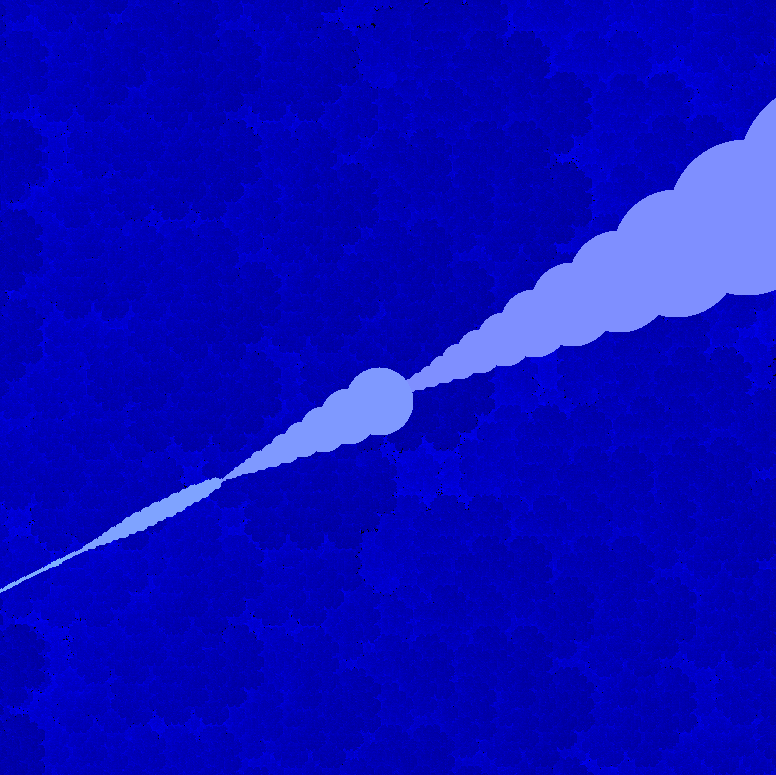}
\end{center}
\end{minipage}
\caption{A loop of traps encircling the holes from 
Figure~\ref{figure:certified_hole_apparent_hole}.  A zoomed view of 
part of the loop shows how the program overlaps rigorous trap balls to 
produce a path inside the interior of $\SetA$.}
\label{figure:certified_hole_trap_loop}
\end{center}
\end{figure}

Recall that the existence of a trap for parameter $z$ is an open condition --- there is 
some $\delta >0$ so that a trap for $z$ persists in $B_\delta(z)$.  
We call this a \emph{ball of traps}.
Our program certifies a putative hole in $\SetA$ by producing overlapping balls 
of traps along a closed path encircling the hole, as shown in 
Figure~\ref{figure:certified_hole_trap_loop}.  This proves that the closed 
path lies completely inside the interior of $\SetA$.  
Some technical remarks are in order.  First, to complete the proof of 
the existence of a hole, we must certify some parameter $z$ on the inside of the 
loop as Schottky.  But since the connectedness of a pixel in 
Figure~\ref{figure:certified_hole_trap_loop} 
is decided using Algorithm~\ref{algorithm:disconnectedness} applied to 
some parameter inside that pixel, a white pixel is guaranteed to 
contain some parameter which is Schottky, so this step is complete.  
Also, we note that the loop of trap balls in 
Figure~\ref{figure:certified_hole_trap_loop} appears to encircle many 
separate holes, but the output of this particular 
run of the program says nothing about whether these holes are actually distinct.  
We would need to run the program separately on loops encircling each of the holes 
we wished to rigorously separate.  In Section~\ref{section:renormalization}, 
we extend our algorithm to prove the existence of infinitely many holes.

\subsection{Numerical trap finding and loop certification}

In this section, we describe our trap-finding algorithm in detail, including 
various numerical estimates.  The details are important, because the output of 
a particular run of this algorithm serves as a rigorous certificate that 
there are multiple connected components of Schottky space, or equivalently, holes 
in $\SetA$.

The program will typically be required to produce 
a sequence of trap balls along a loop.
Thus, we will be interested in finding a large number of traps in a given small 
region of parameter space.  The algorithm takes advantage of this by 
separating the work into two pieces: a more computationally intensive 
piece of one-time work to find \emph{trap-like balls}, and a fast check 
to produce a single ball of traps.  Note that a trap-like ball (of vectors) 
and a ball of traps are not the same thing.

\subsubsection{Finding trap-like balls}

\begin{remark}
This section is full of messy definitions and computations.  These are 
necessary because we are in search of trap-like vectors similar to 
those found in Definition~\ref{definition:cell_trap} but which work for 
\emph{all} $z$ in a given region, so we need to carefully estimate how 
the limit set changes as we change $z$.  As a reward for this tedium, 
we get to compute these trap-like vectors (which is hard) only 
once, but we get to use them over an entire region.
\end{remark}

Suppose that we will be searching for traps in a 
square region $B\subseteq \D^*$ of parameter space centered at $z_0$ and with side length $2d$.
Let $n$, the \emph{hull depth}, be given.
Let $r_z = |z-1|/2(1-|z|)$; this is the minimal radius such that 
a disk of radius $r_z$ centered at $1/2$ is mapped inside itself under both $f$ and $g$.  
Let $D(p)$ denote the disk of radius $p$ centered at $1/2$.  
Typically, we compute $\Sigma_n(z,D(r_z))$, the union of images of $D(r_z)$ 
under all words of length $n$, to study $\Lz$.  However, we need to control $\Sigma_n(z,D(r_z))$ over all 
$z\in B$, so we need to understand how it changes as 
we vary $z$.  For this, we need some constants.  The reader might 
consult Lemma~\ref{lemma:lambda_stays_inside} and 
Figure~\ref{figure:lambda_stays_inside} for motivation before 
working through the technical details.

\begin{enumerate}
\item Let $K$ be an upper bound for $r_z$ in $B$.  We can assume $|z| \le 1/\sqrt{2}$, 
so the value $K = 2.92 > \sup_{|z|=1/\sqrt{2}}\; r_z$ will always work.

\item Let $C$ be such that for any word $u\in\partial\Sigma$ and $z \in B$, 
we have $|\pi(u,z_0) - \pi(u,z)| < C|z_0-z|$.  Since $u$ can be expressed 
as a power series in $z$ with coefficients in $\{0,\pm 1\}$, an upper bound 
for the derivative in terms of $z$ is given by 
$\sum_{i=1}^\infty i|z|^{i-1} = 1/(|z|-1)^2$, so 
a valid value of $C$ is given by $\sup_{z \in B}1/(|z|-1)^2$.
As previously mentioned, we can assume that $B$ lies within the disk of 
radius $1/\sqrt{2}$ by Lemma~\ref{lemma:inner_outer}, so the uniform 
value of $C = 11.67$ will always work.

\item Let $A$ be an upper bound for $|z|/|z_0|$ over $B$.  Because 
$1/2 < |z|,|z_0| < 1/\sqrt{2}$, we have $|z|/|z_0| < \sqrt{2}$.  For the previous 
two constants, a uniform upper bound like this is acceptable.  In this case, though, 
we will be raising $A$ to a large power, so it is critical to make $A$ as close to $1$ as possible.
\end{enumerate}

Next set: 
\[
R_{z_0} = A^nK + 4K + 3|z_0|^{-n}C\sqrt{2}d
\]
and for $z \in B$, define
\[
R_z = \frac{|z_0|^n}{|z|^n} R_{z_0} - |z|^{-n}C|z-z_0|.
\]

\begin{lemma}\label{lemma:lambda_stays_inside}
Suppose that $\Sigma_n(z_0,D(r_{z_0}))$ is connected.  Then for 
any $z \in B$, we have 
\begin{enumerate}
\item $\Sigma_n(z,D(R_z)) \subseteq \Sigma_n(z_0,D(R_{z_0}))$.
\item $\Sigma_n(z,D(R_z))$ contains an $\epsilon$-neighborhood of $\Lz$ for 
some $\epsilon$ such that there are two points 
$p_1 \in f\Lz$, $p_2 \in g\Lz$ such that $|p_1-p_2| < \epsilon$.
\end{enumerate}
\end{lemma}
Note that Algorithm~\ref{algorithm:disconnectedness} shows that (2) implies 
$\Sigma_n(z,D(R_z))$ is connected.

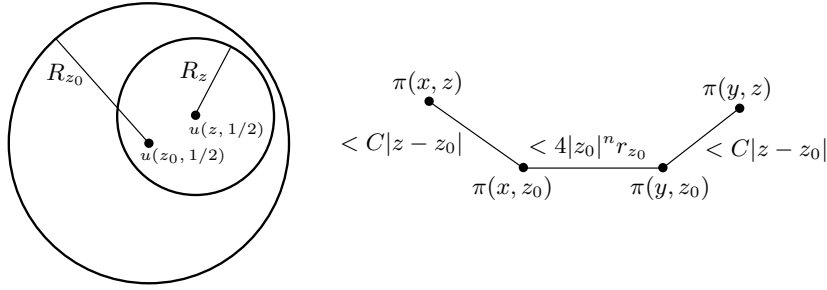
\begin{figure}[htb]
\begin{center}
\labellist
\small\hair 2pt
\tiny
 \pinlabel {$u(z_0,1/2)$} at 60 44
 \pinlabel {$u(z,1/2)$} at 75 53
\small
 \pinlabel {$R_z$} at 64 73
 \pinlabel {$R_{z_0}$} at 20 72
 \pinlabel {$<4|z_0|^nr_{z_0}$} at 199 47
 \pinlabel {$<C|z-z_0|$} at 260 46
 \pinlabel {$<C|z-z_0|$} at 135 48
 \pinlabel {$\pi(x,z)$} at 144 69
 \pinlabel {$\pi(y,z)$} at 250 67
 \pinlabel {$\pi(x,z_0)$} at 173 33
 \pinlabel {$\pi(y,z_0)$} at 227 33
\endlabellist
\includegraphics[scale=1.1]{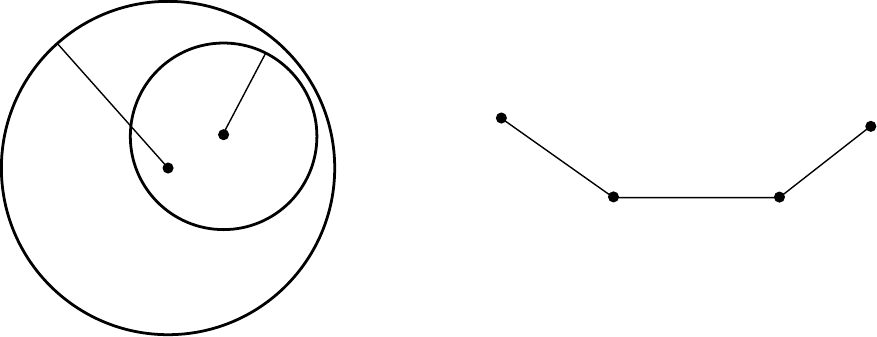}
\end{center}
\caption{The proof of Lemma~\ref{lemma:lambda_stays_inside}
just verifies that when we change the parameter $z_0$ to $z$, 
each disk $u(z,D(R_z))$ lies inside $u(z_0, D(R_{z_0}))$ 
and $\Lz$ still contains points that are close together.  
The figure on the right shows that we can prove (2) by proving that 
for each word $u$, $u(z,D(R_z))$ contains a 
$(4|z_0|^nr_{z_0} + 2C|z-z_0|)$-neighborhood of $u(z,D(r_z))$.
In the figure, $u\in\Sigma_n$ and $x,y \in \partial\Sigma$.}
\label{figure:lambda_stays_inside}
\end{figure}

\begin{proof}
The set $\Sigma_n(z,D(R_z))$ is the union of disks of radius $|z|^nR_z$ centered at 
the images $u(z,1/2)$ over all words $u$ of length $n$, and the set 
$\Sigma_n(z_0,D(R_{z_0}))$ is a similar union of disks of radius $|z_0|^nR_{z_0}$ 
centered at the images $u(z_0,1/2)$.  We prove (1) by showing that 
for each $u \in \Sigma_n$, the disk of radius $|z|^nR_z$ at $u(z,1/2)$ 
lies inside the disk of radius $|z_0|^nR_{z_0}$ at $u(z_0,1/2)$; we just compute 
from the definition of $R_z$:
\[
|z|^nR_z + C|z-z_0| = |z_0|^nR_{z_0}
\]
and by the definition of $C$, we have $|u(z,1/2) - u(z_0,1/2)| < C|z-z_0|$, so 
(1) follows.

To prove (2), first note that since $\Sigma_n(z_0,D(R_{z_0}))$ is connected, by 
Algorithm~\ref{algorithm:disconnectedness}, there are words $u',v'$ starting with $f,g$, 
respectively, such that the disks of radius $|z_0|^nR_{z_0}$ centered at $u'(z_0,1/2)$ 
and $v'(z_0,1/2)$ intersect.  Therefore, since these disks contain points in $\Lambda_{z_0}$, 
there are right-infinite words $u,v$ starting 
with $f,g$, respectively, such that $|\pi(u,z_0) - \pi(v,z_0)| < 4|z_0|^nr_{z_0}$.  
Therefore, 
\[
|\pi(u,z)- \pi(v,z)| < 4|z_0|^nr_{z_0} + 2C|z-z_0|,
\]
since $\pi(u,z)$ and $\pi(v,z)$ can each move by at most $C|z-z_0|$.  So there 
are two points in $\Lz$ which are closer than $\epsilon$, where $\epsilon$ is 
the right hand side of the inequality.  

Now we must show that $\Sigma_n(z,D(R_z))$ contains an $\epsilon$-neighborhood of $\Lz$.  
We know that $\Sigma_n(z,D(r_z))$ contains $\Lz$, so it suffices to show that 
the difference between the radii of the disks in $\Sigma_n(z,D(r_z))$ and the disks in $\Sigma_n(z,D(R_z))$ 
is at least $\epsilon$.  We compute
\begin{align*}
|z|^nR_z - |z|^nr_z &= |z_0|^nR_{z_0} - C|z-z_0| - |z|^nr_z \\
                      &= |z_0|^n\left( A^nK + 4K + 3|z_0|^{-n}C\sqrt{2}d \right) - C|z-z_0| - |z|^nr_z\\
                      &\ge ((|z_0|A)^n-|z|^n) r_z + 4|z_0|^nr_{z_0} + 2C|z-z_0|  \\
                      &\ge 4|z_0|^nr_{z_0} + 2C|z-z_0| \\
                      &= \epsilon
\end{align*}
Where we have used $r_z,r_{z_0} < K$ and $|z-z_0| < \sqrt{2}d$.  Also, because 
$|z_0|A > |z|$, we have $(|z_0|A)^n - |z|^n > 0$.
\end{proof}

Let $T$ be a component of the complement of $\Sigma_n(z_0,D(R_{z_0}))$ inside the 
convex hull of $\Sigma_n(z_0,D(R_{z_0}))$.  Note that the boundary of $T$ contains 
a line segment along a supporting hyperplane for the convex hull (the ``outside'' 
boundary of $T$).  There are two distinguished 
disks in $\Sigma_n(z_0,D(R_{z_0}))$ which lie on either end of this line segment and 
are centered at images of $1/2$ under two words in $\Sigma_n$.  Let 
these disks be centered at $p_1 = u_1(z_0,1/2)$ and $p_2 = u_2(z_0,1/2)$.  
Next, let $q$ be a point in $T$ which is distance $\alpha'$ from 
$\Sigma_n(z_0,D(R_{z_0}))$, and suppose that $\alpha > 0$, where 
\begin{align*}
\alpha &= \alpha' - (|z_0|A)^nK - 4|z_0|^nK - 5C\sqrt{2}d \\
       &= \alpha' - |z_0|^nR_{z_0} - 2C\sqrt{2}d
\end{align*}
\begin{definition}\label{definition:trap_like_balls}
In the above notation, the balls $B_\alpha(p_1-q)$ and $B_\alpha(p_2-q)$ 
are \emph{trap-like balls} for the region $B$.
\end{definition}

\begin{figure}[htb]
\begin{center}
\includegraphics[scale=0.3]{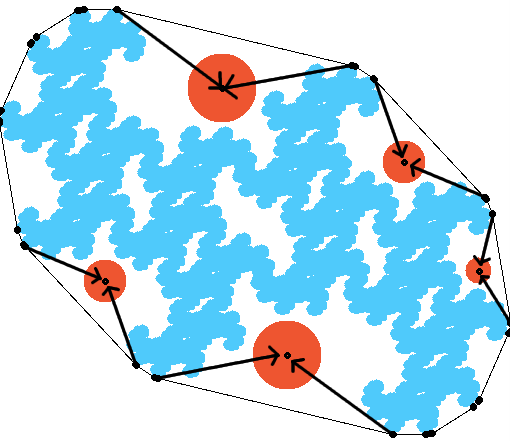}~~
\includegraphics[scale=0.24]{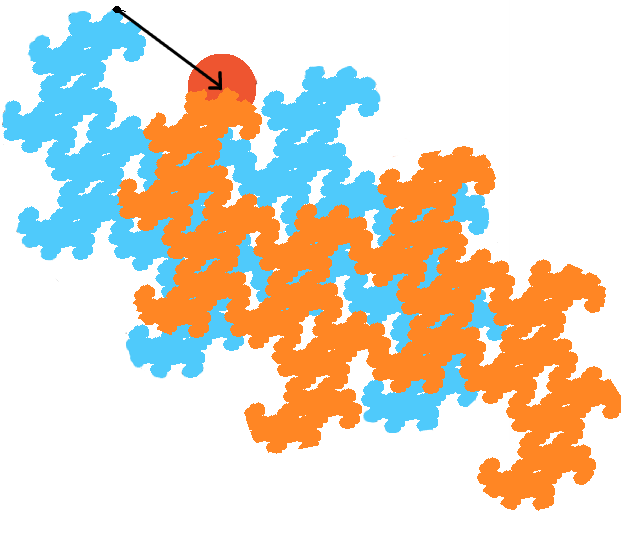}
\end{center}
\caption{
A supporting hyperplane of the convex hull of $\Sigma_n(z_0, D(R_{z_0}))$ 
intersects $\Sigma_n(z_0, D(R_{z_0}))$ is two balls.  Vectors which 
translate these balls \emph{inside} the convex hull but \emph{outside}
$\Sigma_n(z_0, D(R_{z_0}))$ are trap-like (left).
Translating by a trap-like vector moves $\Sigma_n(z_0, D(R_{z_0}))$ 
transverse to itself and produces a trap (right).}
\label{figure:trap_like_balls}
\end{figure}

That is, a trap-like ball for $B$ is a ball of vectors which translate a disk 
at a vertex of the convex hull of $\Sigma_n(z_0,D(R_{z_0}))$ an appreciable 
amount into the region inside the convex hull but outside the set.
See Figure~\ref{figure:trap_like_balls}.

\begin{remark}
One might wonder whether we should expect any trap-like balls to exist at all, 
since it's not immediately clear why $\alpha$ should be positive.  
Recall that $|z_0| < 1/\sqrt{2}$, so for $n$ large enough, 
$\alpha \approx \alpha' - 5C\sqrt{2}d$, and $d$ is probably tiny 
compared to the scale of $\Sigma_n(z_0,D(R_{z_0}))$ (which is approximately the limit set $\Lambda_{z_0}$).
\end{remark}

\begin{lemma}\label{lemma:trap_like_balls_are_symmetric}
If $B_\alpha(v)$ is a trap-like ball for $B$, then $B_\alpha(-v)$ is also a 
trap-like ball.
\end{lemma}
\begin{proof}
The set $\Sigma_n(z_0,D(R_{z_0}))$ is rotationally symmetric under a rotation 
of order $2$ about the point $1/2$.  A trap-like ball is taken to a 
trap-like ball under this rotation, and in the above notation, it 
will negate the vectors $p_1-q$ and $p_2-q$.
\end{proof}

\subsubsection{Finding a ball of traps centered at a parameter $z$}

In this section, we show how to use the trap-like balls 
produced in the previous section to verify the existence of a 
ball of traps at $z$.  We fix notation as in the previous section, 
so we have a square region $B$ in parameter space 
with side length $2d$ and centered at $z_0$.  We let $K$ be an upper 
bound for $r_z$ over $B$ and $C$ be such that 
$|u(z,1/2) - u(z_0,1/2)| < C|z-z_0|$ for all $u \in \Sigma_n$ and $z \in B$.

\begin{lemma}\label{lemma:ball_of_traps}
Let $u,v \in \Sigma_m$ be such that $u$ starts with $f$ and $v$ starts 
with $g$ and $z^{-m}(u(z,1/2) - v(z,1/2)) \in B_\alpha(p)$, 
where $B_{\alpha}(p)$ is a trap-like ball for $B$.  Let $Z$ be a lower 
bound for $|z|$ over $B$.  Then there exists a 
trap for every $z' \in B_\epsilon(z)\cap B$, where 
\[
\epsilon = \frac{Z^m}{2C}(\alpha - |z^{-m}(u(z,1/2) - v(z,1/2)) - p|)
\]
\end{lemma}
\begin{proof}
We will check the 3 hypotheses of Definition~\ref{definition:trap} on the 
words $u$ and $v$ with the topological disk $\Sigma_n(z,D(R_z))$.  
First, $u,v$ start with $f,g$ by construction, so the first condition is verified.
The third condition, that there are points in $f\Lz$ and $g\Lz$ 
within distance $\epsilon$, where the $\epsilon/2$ neighborhood 
of $\Lz$ is contained in $\Sigma_n(z,D(R_z))$, is conclusion (2) 
of Lemma~\ref{lemma:lambda_stays_inside}.

Now we need to verify condition (2) in the definition of a trap.
This is the more difficult verification. After a suitable rescaling,
the problem becomes more tractable.
Consider the unions
$z^{-m} (u\Sigma_n(z,D(R_z)))$ and $z^{-m}(v\Sigma_n(z,D(R_z)))$.  
These sets have a pair of intersecting paths as in condition (2) if and 
only if the original sets $u\Sigma_n(z,D(R_z))$ and $v\Sigma_n(z,D(R_z))$ do.  
Furthermore, the sets $z^{-m} (u\Sigma_n(z,D(R_z)))$ and $z^{-m}(v\Sigma_n(z,D(R_z)))$ 
are exactly the same, up to translation, as 
$\Sigma_n(z,D(R_z))$ and the translated set
\[
\Sigma_n(z,D(R_z)) + z^{-m}(u(z,1/2) - v(z,1/2)).
\]
In other words, we translate the set $\Sigma_n(z,D(R_z))$ off of itself by the vector 
$w = z^{-m}(u(z,1/2) - v(z,1/2))$.  If we can find the interlocking paths 
of condition (2), we are done.  

We start by considering $\Sigma_n(z_0,D(R_{z_0}))$ and then thinking about what 
can happen as we change $z_0$ to $z$.  By hypothesis, $w$ lies in a 
trap-like ball for $B$.  These are four distinguished disks 
associated with $w$, as follows.  The vector $w$ is associated with a 
component of the complement of $\Sigma_n(z_0,D(R_{z_0}))$ inside the 
convex hull of it.  This component has one side which lies along a supporting 
hyperplane $H$ of the convex hull, and $H$ intersects two disks $P_1,P_2$ which 
sit on either side of the component.  By the definition of the 
trap-like balls, the disk $P_1$, which has radius $|z_0|^nR_{z_0}$, is 
translated by $w$ to a disk $Q_1$, which is distance at least $2C\sqrt{2}d+2C\epsilon$ 
away from both $H$ and $\Sigma_n(z_0,D(R_{z_0}))$.  Also note that $P_1,P_2$ 
are this same distance away from the translated set.
Let $H'$ be a hyperplane perpendicular to $w$, and translate 
it so it is supports $\Sigma_n(z_0,D(R_{z_0}))$.
It lies tangent to some 
disk $P_3$.  Now $P_3$ is translated by $w$ to a disk $Q_2$ 
which is distance at least $2C\sqrt{2}d + 2C\epsilon$ away from the slid $H'$, and 
thus that distance away from $\Sigma_n(z_0,D(R_{z_0}))$.  The pairs of 
disks $(P_1,P_2)$ and $(Q_1,Q_2)$ are linked. See 
Figure~\ref{figure:trap_ball_movement}.

\begin{figure}[htb]
\begin{center}
\includegraphics[scale=0.4]{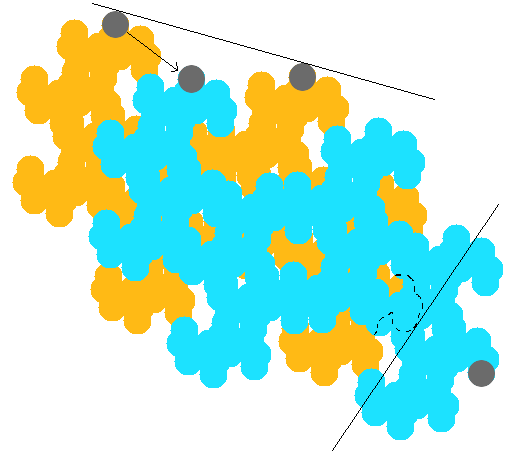}
\begin{picture}(0,0)
\put(-165,185){$P_1$}
\put(-92,163){$P_2$}
\put(-132,159){$Q_2$}
\put(-35,140){$H$}
\put(-15,106){$H'$}
\put(-15,23){$Q_2$}
\end{picture}
\end{center}
\caption{The picture of the proof of Lemma~\ref{lemma:ball_of_traps}.  
If we change the parameter slightly, the marked balls $P_1,P_2,Q_1,Q_2$ 
cannot move much and thus still give a trap.}
\label{figure:trap_ball_movement}
\end{figure}
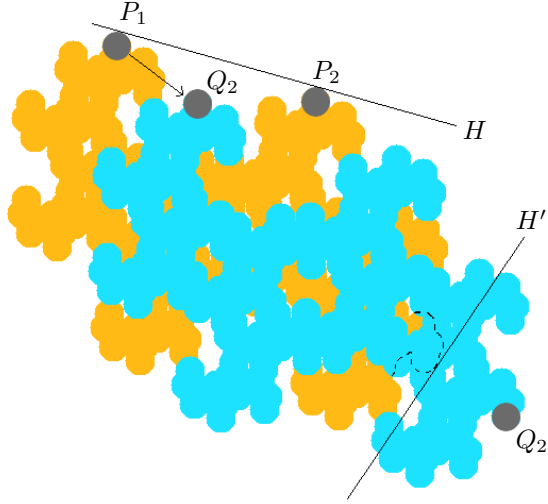


Now change the parameter from $z_0$ to $z$, and consider $\Sigma_n(z,D(R_z))$.  
Because $|z|^nR_z < |z_0|^nR_{z_0}$, the disks $P_1,P_2,Q_1,Q_2$ can only 
shrink.  And every disk, and the supporting hyperplanes $H,H'$, can move at 
most distance $C|z-z_0| < C\sqrt{2}d$.  Therefore, these four disks are 
still disjoint from the opposing copy of $\Sigma_n(z,D(R_z))$, and each contains 
points in $\Lz$, and by Lemma~\ref{lemma:lambda_stays_inside}, 
$\Sigma_n(z,D(R_z))$ remains connected, so these points can be connected 
by paths with algebraic intersection number $1$.  This verifies condition (2) 
of the trap definition.  

This shows that there exists a trap for parameter $z$, but recall that we 
desire a trap for every $z$ in $B_\epsilon(z)$.  To see that this is true, 
observe that at the point $z$, the disks $P_1,P_2,Q_1,Q_2$ are still distance 
at least $2C\epsilon|z|^{-m}$ away from the opposing copy of $\Sigma_n(z,D(R_z))$.  
So we can change $z$ again by at most $\epsilon$ while retaining this trap.
All of this is contingent on the parameter $z$ remaining in $B$, so 
the final ball of traps produced is $B_\epsilon(z) \cap B$.
\end{proof}

\subsection{Certifying holes}

We now summarize this section.  To certify a hole in $\SetA$ which lies 
completely within some square region $B$, we compute $R_{z_0}$ and 
the set $\Sigma_n(z_0,D(R_{z_0}))$ for a reasonable-sized $n$ (say, $15$).  
Then we compute the convex hull and some trap-like balls.  
Then, at the initial point of a path $\gamma$ encircling the hole, we apply 
Lemma~\ref{lemma:ball_of_traps} to produce a ball $B_1$ of traps.  
Then we go along $\gamma$ to the edge of $B_1$, and find another ball of 
traps $B_2$, and so on.  The balls overlap, so together they produce an 
open set inside set A containing $\gamma$.

We do all computations to double precision, which has a precision 
of at least $15$ decimal digits.  Therefore, as long as no number in the 
computation ever requires more than, say, $10$ digits of precision, this is rigorous.
In practice, this is never an issue.

\begin{question}\label{question:hole_combinatorics}
Is there a combinatorial way to distinguish holes in $\SetA$?
\end{question}

\section{Infinitely many holes in \texorpdfstring{$\SetA$}{M} 
and renormalization}\label{section:renormalization}

In this section we describe a certain family of natural operators on the parameter plane
which account for much of the observed self-similarity in the structure of $\SetA$ and $\SetB$.
Similar ideas and some similar results already appear in the work of Solomyak \cite{Solomyak}, 
although our approach is sufficiently different (and enough of the results we obtain are new)
that it is worth including here.

The first main result we obtain is the existence of {\em infinitely many holes} in $\SetA$,
arranged in certain spirals. The proof of this fact does not technically need 
the theoretical apparatus of renormalization; but the phenomenon is not properly 
explained without it. We defer the explanation until after a description and rigorous proof
of the phenomenon, so that the techniques and definitions we then introduce are 
sufficiently motivated.

\subsection{Infinitely many holes}

Numerical exploration of $\SetA$ quickly reveals many interesting phenomena, of which one of
the most interesting is the appearance of apparent ``spirals'' of holes. One of the
most prominent is centered at the point $\omega \sim 0.371859+0.519411i$. 
See Figure~\ref{figure:hexahole_spiral}. The figure also illustrates part of $\SetB$ (in purple),
and exhibits the  limit as the ``tip'' of a spiral of $\SetB$. Techniques of Solomyak
\cite{Solomyak} certify that $\SetB$ is self-similar at the  limit, and is asymptotically
similar to the limit set $\Lambda_\omega$.

\begin{figure}[htpb]
\centering
\includegraphics[scale=0.4]{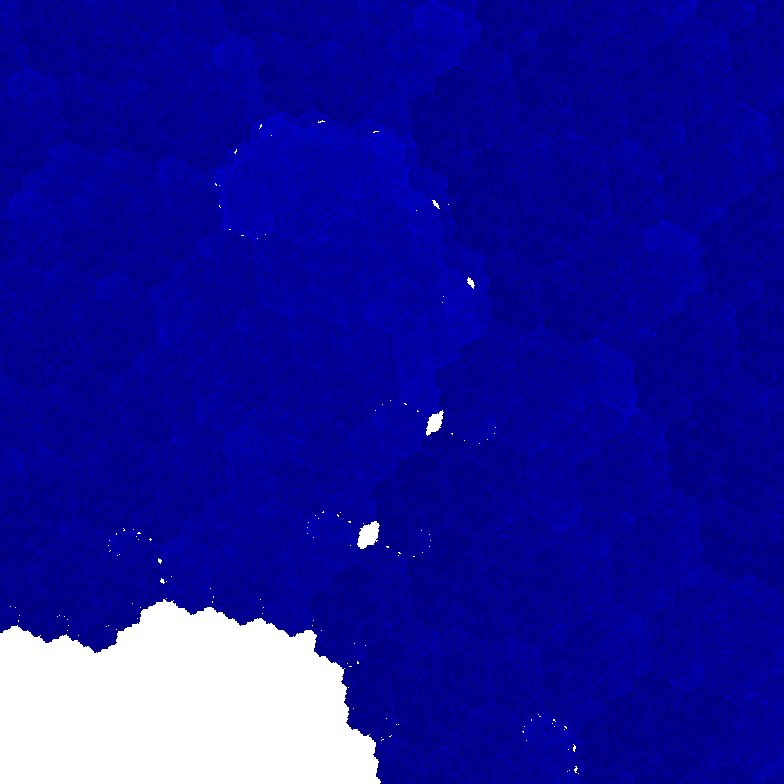}
\caption{Spiral of holes converging to $\omega \sim 0.371859+0.519411i$.}
\label{figure:hexahole_spiral}
\end{figure}

The main theorem
we prove in this section is the following:

\begin{theorem}[Limit of holes]\label{theorem:hole_limit}
Let $\omega \sim 0.371859+0.519411i$ be the root of the
polynomial $1-2z+2z^2-2z^5+2z^8$ with the given approximate value. Then
\begin{enumerate}
\item{$\omega$ is in $\SetA$, $\SetB$ and $\SetC$; in fact, the intersection of $f\Lambda_\omega$
and $g\Lambda_\omega$ is exactly the point $1/2$;}
\item{there are points in the complement of $\SetA$ arbitrarily close to $\omega$; and}
\item{there are infinitely many rings of concentric loops in the interior of $\SetA$ which
nest down to the point $\omega$.}
\end{enumerate}
Thus, $\SetA$ contains infinitely many holes which accumulate at the point $\omega$.
\end{theorem}

We refer informally to the holes accumulating on $\omega$ as {\em hexaholes} 
(because of their approximate shape), and to $\omega$ itself as the {\em hexahole limit}.

The first step in the proof is to give a more combinatorial description of the hexahole
limit $\omega$, and prove the first bullet of Theorem~\ref{theorem:hole_limit}.

\begin{lemma}
The set of $z$ for which $\pi(fgfffgggf^\infty, z) = \pi(gfgggfffg^\infty, z)$ are exactly the roots of
$1-2z+2z^2-2z^5+2z^8$. 
\end{lemma}
\begin{proof}
By Proposition~\ref{proposition:power_series}, the power series associated to these two infinite
words are actually finite polynomials; equating them gives the identity
$$z-z^2+z^5-z^8 = 1-z+z^2-z^5+z^8$$
so that $1-2z+2z^2-2z^5+2z^8=0$.
\end{proof}

The next step in the proof requires us to certify the existence of points in the complement
of $\SetA$, arbitrarily close to $\omega$. Any given value of $z$ can be numerically certified
as being in the complement of $\SetA$ by Algorithm~\ref{algorithm:disconnectedness}, but we would
like to apply this algorithm uniformly to an infinite collection of $z$ of a particular form.

First recall the form of the algorithm: given $z$ as input, and a cutoff depth, we first
load the number $1-z^{-1}$ into a ``stack'' $V$, and then recursively replace the content
of the stack at each stage with a set of {\em viable children}. More precisely, for each
$\alpha \in V$, there are three children $z^{-1}\alpha$, $z^{-1}(\alpha + z - 1)$, and
$z^{-1}(\alpha - z + 1)$. A child is {\em viable} if its absolute value is less than a
constant $R$ depending only on the initial value $z$ (in fact, we can take $R$ to be fixed
throughout a neighborhood of a given $z$), and at each stage of the algorithm we replace each
number in $V$ with the set of its viable children. The algorithm halts whenever the stack
$V$ is empty (in which case we certify that $z$ is in the complement of $\SetA$) or if we 
exceed the ``run time'' (i.e.\/ the cutoff depth) allocated in advance.

Let's imagine running our algorithm on an ideal machine without imposing a cutoff depth, so
that the algorithm halts if and only if $\Lz$ is disconnected. At each successive
time step $d$, the stack $V$ consists of a finite list of numbers. If it happens that
the content of $V$ is eventually {\em periodic} (and nonempty) as a function of $d$, then
of course the algorithm never halts --- certifying that $z$ is in fact in $\SetA$. Now, the
numbers in $V$ at each finite stage are all Laurent polynomials in $z$ of degree bounded by
$d$, so if $V$ is eventually periodic, then $z$ must be algebraic. 

If we apply the algorithm to the number
$\omega$ defined above, then we indeed can certify that $V$ is eventually periodic, and
in fact becomes {\em constant} after $d=12$.

\begin{example}[Stack contents for $\omega$]
For $\omega \sim 0.371859+0.519411i$ the root of $1-2z+2z^2-2z^5+2z^8=0$ 
we can take $R = 2.257$. Unfortunately, the full stack over all $12$ 
steps is somewhat unwieldy, so we do not list it here.  However, on step 
$9$, the stack contains the number $1$, and this is the only 
stack entry with viable descendants.  The tree diagram of the algorithm 
at this point becomes periodic; we show it in Figure~\ref{figure:periodic_stack}.

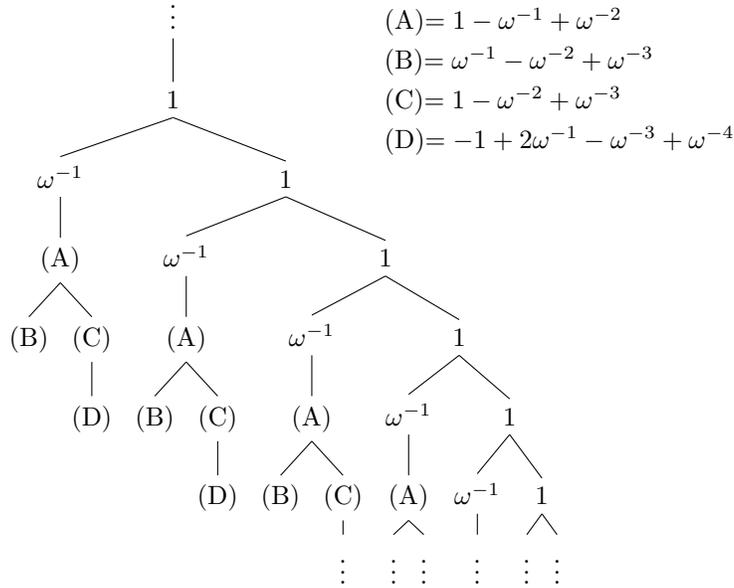
\begin{figure}[htb]
\begin{center}
\Tree [.$\vdots$ 
      [.$1$ [.$\omega^{-1}$ 
              [.(A) (B) [.(C) (D) ] ]
            ] 
            [.$1$ 
              [.$\omega^{-1}$ 
		        [.(A) (B) [.(C) (D) ] ]
		      ] 
		      [.$1$ 
                [.$\omega^{-1}$ 
		          [.(A) (B) [.(C) $\vdots$ ] ]
		        ] 
		        [.$1$ 
                  [.$\omega^{-1}$ 
		            [.(A) $\vdots$ $\vdots$ ]
		          ] 
		          [.$1$ 
                    [.$\omega^{-1}$ 
                       $\vdots$
		            ] 
		            [.$1$ $\vdots$ $\vdots$ ]
		          ] 
		        ] 
		      ]  
		    ] 
      ]]
\begin{picture}(0,0)
\put(-70,0){(A)$ =1-\omega^{-1}+\omega^{-2}$}
\put(-70,-15){(B)$ =\omega^{-1}-\omega^{-2}+\omega^{-3}$ }
\put(-70,-30){(C)$ =1-\omega^{-2}+\omega^{-3}$ }
\put(-70,-45){(D)$ =-1+2\omega^{-1}-\omega^{-3}+\omega^{-4}$ }
\end{picture}
\end{center}
\caption{ 
The periodic stack of the disconnectedness algorithm on the input $\omega$.}
\label{figure:periodic_stack}
\end{figure}

It is clear from the tree diagram that the stack becomes constant at 
step $12$.
\end{example}

By analyzing precisely which children have indefinitely viable descendents, we get a 
precise description of the intersection $f\Lambda_\omega\cap g\Lambda_\omega$. In this case,
we can readily observe that there is a unique pair of infinite words $u,v$ where $u$ starts
with $f$ and $v$ with $g$ so that $\pi(u,\omega) = \pi(v,\omega)$; these words are in fact related under the canonical involution, so that
the intersection consists exactly of the point $1/2$. This proves the first bullet in Theorem~\ref{theorem:hole_limit}.

The second step in the proof of Theorem~\ref{theorem:hole_limit} is to certify the existence
of points in the complement of $\SetA$ arbitrarily close to $\omega$. These points will all be
of the form $\omega + C\omega^\ell$ for sufficiently big $\ell$, and for a fixed constant 
$C = 0.29946137 - 0.48972405i$.

\begin{proposition}[$\omega$ limit of Schottky]\label{proposition:omega_limit}
For $C = 0.29946137 - 0.48972405i$, there is $\ell$ so that point $z=\omega+C\omega^{\ell}$ 
is Schottky for sufficiently large $\ell$.
\end{proposition}
In order to prove Proposition~\ref{proposition:omega_limit},
we are going to formally run the disconnectedness algorithm on $z$ and 
show that we can understand the contents of the stack as long as $\ell$ is 
large enough.  The stack will essentially be the same as the stack
for $\omega$ for a long time, followed by a uniformly bounded (in $\ell$) number of steps 
which prove disconnectedness.  This discussion is 
elementary, but it requires taking things to infinity 
in a careful order.

We first prove a general lemma which provides the stack contents; 
proving the proposition then reduces to doing a numerical computation for 
the given $C$ value.  
To set up the lemma, we need to do a computation.
Recall that when running the disconnectedness algorithm on 
$\omega$, at every step there is a single stack entry (i.e. ``$1$'')
which has infinitely many descendants.  The true, unsimplified
version of this entry in the stack at step $n$ is 
\[
p_{1,n}(z) = 1 + z^{-n}(-1+2z-2z^2+2z^5-2z^8)
\]
At every time step, there are $5$ other polynomials on the stack, 
which are 
the finitely many children of $p_{1,n-1}(z)$, $p_{1,{n-2}}(z)$, 
and $p_{1,n-3}(z)$ which have not yet died.  These polynomials are:
\begin{align*}
p_{2,n}(z) &= z^{-n}(-1+2z-2z^2+2z^5-2z^8+z^{n-4}-z^{n-3}+2z^{n-1}-z^n) \\
p_{3,n}(z) &= z^{-n}(-1+2z-2z^2+2z^5-2z^8+z^{n-3}-z^{n-2}+z^{n-1})      \\
p_{4,n}(z) &= z^{-n}(-1+2z-2z^2+2z^5-2z^8+z^{n-3}-z^{n-2}+z^n)          \\
p_{5,n}(z) &= z^{-n}(-1+2z-2z^2+2z^5-2z^8+z^{n-2}-z^{n-1}+z^n)          \\
p_{6,n}(z) &= z^{-n}(-1+2z-2z^2+2z^5-2z^8+z^{n-1})                      \\
\end{align*}
It is important to note that obtaining these polynomials 
involves \emph{no} simplification at any stage.  Thus, heuristically, 
for fixed $n$ and $z$ values close to $\omega$, these polynomials 
should give the stack contents.  This is the idea 
Lemma~\ref{lemma:contents_of_stack} explores in detail.

We 
compute the values of the $p_{i,n}(z)$ polynomials at the point 
$z=\omega + C\omega^{m+k}$ 
(it is pedagogically helpful to split $\ell$ in 
Proposition~\ref{proposition:omega_limit} into the two 
variables $\ell = m+k$).
This computation is just an expansion, simplified using 
the polynomial of which $\omega$ is a root. 
For compactness, we denote 
$p_{i,n}(\omega + C\omega^{m+k})$ by $p_{i,n}^{m+k}$. 
All of these polynomials have a large ``remainder'' term, which 
we will denote by 
\[
R_{n,m,k} = \frac{\omega^{m+k}}{(\omega+C\omega^{m+k})^n}(2C-4C\omega + 
10C\omega^4-16C\omega^7 + O(\omega^m))
\]
Now we list the polynomials:
\begin{align*}
p_{1,n}^{m+k} &= 1 + R_{n,m,k}
 \\
p_{2,n}^{m+k} &= -1 + \frac{1}{(\omega+C\omega^{m+k})^2}
                    - \frac{1}{(\omega+C\omega^{m+k})^3}
                    + \frac{1}{(\omega+C\omega^{m+k})^4} +R_{n,m,k} \\
p_{3,n}^{m+k} &= \frac{1}{\omega+C\omega^{m+k}}
                -\frac{1}{(\omega+C\omega^{m+k})^2}
                +\frac{1}{(\omega+C\omega^{m+k})^3} +R_{n,m,k} \\
\end{align*}
\begin{align*}
p_{4,n}^{m+k} &= 1 - \frac{1}{(\omega+C\omega^{m+k})^2} 
                   + \frac{1}{(\omega+C\omega^{m+k})^3} +R_{n,m,k} \\
p_{5,n}^{m+k} &= 1 -\frac{1}{\omega+C\omega^{m+k}} 
                   +\frac{1}{(\omega+C\omega^{m+k})^2} +R_{n,m,k} \\
p_{6,n}^{m+k} &= \frac{1}{\omega+C\omega^{m+k}} + R_{n,m,k} 
\end{align*}

\begin{lemma}\label{lemma:contents_of_stack}
For any $C$, there are constants $k$ and $M$ such that for 
all $m>M$ and $12<n\le m$, the 
contents of the stack of the disconnectedness algorithm at step $n$ 
when run on $\omega + C\omega^{m+k}$ is exactly 
the set of $p_{i,n}^{m+k}$ for $1\le i\le 6$.
\end{lemma}
\begin{proof}
In order to prove this lemma, first consider running the algorithm on $\omega$: 
the stack beyond step $12$ is constant at 
\[
\{1, -1+\frac{1}{\omega^2} -\frac{1}{\omega^3} + \frac{1}{\omega^4} ,
\frac{1}{\omega} - \frac{1}{\omega^2} + \frac{1}{\omega^3},
1-\frac{1}{\omega^2} + \frac{1}{\omega^3},
1-\frac{1}{\omega}+\frac{1}{\omega^2},
\frac{1}{\omega}\} = 
\{p_{i,n}(\omega)\}_{i=1}^6
\]
Now think of varying the input from 
$\omega$ to $\omega + C\omega^{m+k}$.  In order to prove that 
the contents of the 
stack are as claimed, we need to show two things (1) the 
polynomials $p_{i,n}^{m+k}$ 
stay on the stack for all $n\le m$ and (2) every child 
which was discarded for $\omega$ 
through step $n$ still gets discarded.

First note that for $n\le m$,
\[
\left\vert \frac{\omega^{m+k}}{(\omega+C\omega^{m+k})^n}\right\vert \le 
\left\vert \frac{\omega^{m+k}}{(\omega+C\omega^{m+k})^m}\right\vert,
\]
and furthermore, $\omega^{m+k}/(\omega+C\omega^{m+k})^m$ converges to $\omega^{k}$ 
from below as $m\to\infty$.  Therefore, the absolute 
value $|R_{n,m,k}|$ is uniformly (in $n$) bounded above 
by the ``worst case'' $|R_{m,m,k}|$ where $n=m$:
\[
|R_{n,m,k}| \le |R_{m,m,k}| \le \left| \omega^{k}\left( 2C - 
4C\omega + 10C\omega^4 - 16C\omega^7 + O(\omega^m) \right)\right|
\]
So for example, 
the variation $|p_{6,n}(\omega)-p_{6,n}^{m+k}|$ is 
uniformly (in $n$) bounded 
by the expression:
\[
\left| \frac{1}{\omega} - \frac{1}{\omega+C\omega^{m+k}}\right| + 
\left| \omega^{k}\left( 2C - 4C\omega + 10C\omega^4 - 16C\omega^7 + O(\omega^m) \right)\right|.
\]
There are similar expressions for $|p_{i,n}(\omega)-p_{i,n}^{m+k}|$
for each $i$.
Now if we make $k$ large and bound $m$ from below, we can make all these expressions 
as small as we like, and hence small enough so that the $p_{i,n}^{m+k}$ remain on the 
stack for all $n\le m$.

To prove that these are the \emph{only} things on the stack, we compute expressions 
for the children of $p_{i,n}^{m+k}$ and do exactly the same thing to prove that a 
large enough $k$ and $m$ make the worst-case deviation from the children 
of $p_{i,n}(\omega)$ small for $n\le m$ and 
hence these children will leave the stack exactly as the children of 
$p_{i,n}(\omega)$ do.  This computation is the same, so we omit it.
\end{proof}

We note that for a specific value of $C$ (such at the one given in the proposition), 
it is possible to actually numerically compute $k$.  As an example, we show 
how to compute a value of $k$ which ensures $p_{1,n}^{m+k}$ remains on the stack 
for sufficiently large $m$.  As $m\to \infty$, the difference $|p_{1,n}(\omega) - p_{1,n}^{m+k}|$ 
is bounded above by the limit 
\[
\left| \omega^{k}\left( 2C - 4C\omega + 10C\omega^4 - 16C\omega^7 \right)\right|
\]
Hence if we compute $p_{1,n}(\omega)$ and observe how far away it is from getting 
cut off the stack (remember things get removed if their absolute value is too large), 
we can choose $k$ so that the expression above is small enough that $p_{1,n}^{m+k}$ 
remains on the stack for $m> M$ and $n\le m$, (where $M$ can depend on $k$).
In order to compute a value of $k$ which actually works for Lemma~\ref{lemma:contents_of_stack}, 
it is necessary to 
consider $p_{i,n-1}(\omega)$ over all $i$ and all their children and make sure 
that $k$ is large enough to accept or reject them appropriately.  

\begin{proof}[Proof of Proposition~\ref{proposition:omega_limit}]

Doing the computation above for the specified value 
$C = 0.29946137 - 0.48972405i$  
(this is an exact value)
shows that Lemma~\ref{lemma:contents_of_stack} holds with $k=12$.
Therefore, for all $m$ sufficiently large, the contents of the 
stack at time $n\le m$ will be as claimed in the lemma.  By taking 
$m$ large, we can get the stack contents at step $n=m$ as close as we 
like to the limits, which we denote by $p_{i,\infty}^k$.
\begin{align*}
p_{1,\infty}^k &= 1 + \omega^{k}\left( 2C - 4C\omega + 10C\omega^4 - 16C\omega^7\right) \\ 
\
p_{2,\infty}^k &= -1+\frac{1}{\omega^2} -\frac{1}{\omega^3} + \frac{1}{\omega^4} + 
\omega^{k}\left( 2C - 4C\omega + 10C\omega^4 - 16C\omega^7\right)\\
\
p_{3,\infty}^k &= \frac{1}{\omega} - \frac{1}{\omega^2} + \frac{1}{\omega^3} + 
\omega^{k}\left( 2C - 4C\omega + 10C\omega^4 - 16C\omega^7\right) \\
\
p_{4,\infty}^k &= 1-\frac{1}{\omega^2} + \frac{1}{\omega^3} + \omega^{k}\left( 2C - 4C\omega + 10C\omega^4 - 16C\omega^7\right) \\
\
p_{5,\infty}^k &= 1-\frac{1}{\omega}+\frac{1}{\omega^2} +  \omega^{k}\left( 2C - 4C\omega + 10C\omega^4 - 16C\omega^7\right) \\
\
p_{6,\infty}^k &= \frac{1}{\omega} +  \omega^{k}\left( 2C - 4C\omega + 10C\omega^4 - 16C\omega^7\right) \phantom{\frac{1}{x}} 
\end{align*}

We want to continue running the disconnectedness algorithm at this point.
Recall we require a radius outside which we 
discard children.  By taking $m$ large, we may assume this radius 
is the one for $\omega$, i.e. $2|\omega-1|/2(1-|\omega|) < 2.26$ and 
that the algorithm replaces $\alpha$ with the three children
\[
\omega^{-1}\alpha, \qquad \omega^{-1}(\alpha + \omega -1), \qquad \omega^{-1}(\alpha - \omega +1).
\]
Now start this algorithm with the given (numerical, with $k=12$) 
stack contents $\{p_{i,\infty}^{k}\}_{i=1}^6$; it 
terminates (with an empty stack) in $20$ steps.  Therefore, 
for $k=12$, there is some $M$ such that for all $m>M$, the disconnectedness 
algorithm run on the input $\omega+C\omega^{m+k}$ certifies 
disconnectedness at step $m+20$.  This completes the proof.
\end{proof}

This proves the second bullet in Theorem~\ref{theorem:hole_limit}. Note that by means of this
method we can numerically certify any $C\in \C$ for which
the points $\omega + C\omega^n$ are Schottky for all sufficiently large $n$. However, when
this method of certification fails, we cannot conclude that the corresponding points are
all (eventually) in $\SetA$; a different method is necessary for that.

The last step in the
proof of Theorem~\ref{theorem:hole_limit} is to certify the existence of infinitely many
rings of concentric loops in the interior of $\SetA$ which nest down to the point $\omega$. This
depends on an analysis of how trap vectors transform under certain combinatorial and numerical
operations. We discuss this in the remainder of the section.

Let $R\subseteq \C$ be a small region containing $\omega$.  Recall from 
Section~\ref{section:holes} that we can produce a collection of 
trap-like balls for the region $R$ such that if $z \in R$ and 
there exist $u,v\in \Sigma_m$ starting with $f,g$, respectively, 
such that $z^{-m}(u(z,1/2)-v(z,1/2))$ lies in a trap-like ball, 
then there exists a trap at $z$, and $z$ lies in the interior of $\SetA$.
We will use this to show that for $z$ of the form $\omega + C\omega^n$, we 
can certify the existence of a trap for $z$ for \emph{all} 
sufficiently large $n$.

Given two words $u,v \in \Sigma_m$, not necessarily starting with $f,g$, 
recall that we can write 
\[
u(z,x) = xz^m + p_u(z) \quad \textnormal{and} \quad v(z,x) = xz^m + p_v(z)
\]
for some polynomials $p_u(z)$ and $p_v(z)$ in $z$.  For example, $p_g(z) = -z+1$ 
because $g(z,x) = z(x-1) + 1$.  Define words
\[
U_n = fgfffgggf^nu \quad \textnormal{and} \quad V_n = gfgggfffg^nv
\]

\begin{lemma}\label{lemma:limit_trap}
In the above notation, if the vector 
\[
t = 2C\omega^{-m-8}(1-2\omega+5\omega^4-8\omega^7) + \omega^{-m}(p_u(\omega)-p_v(\omega)+1)
\]
lies in a trap-like ball $B_\alpha(p)$ for the region $R$, then for sufficiently large $n$, 
the words $U_n, V_n$ give a trap for $\omega + C\omega^n$.
Furthermore, if we let 
\[
\epsilon = \frac{|\alpha - (p-t)|}{|2\omega^{-m-8}(1-2\omega+5\omega^4-8\omega^7)|},
\]
then for any compact subset $S$ of $B_\epsilon(C)$, there is an $N$ such that 
for any $C' \in S$ and $n>N$, the words $U_n,V_n$ give a trap for $\omega + C'\omega^n$.
\end{lemma}

\begin{proof}
The proof is primarily a computation.
By applying the definitions of $f$ and $g$, we compute:
\[
U_n(z,1/2) = z - z^2 + z^5 - z^8 + \frac{1}{2}z^{m+n+8} + z^{n+8} p_u(z)
\]
\[
V_n(z,1/2) = 1 - z + z^2 - z^5 + z^8 - z^{n+8} + \frac{1}{2}z^{m+n+8} + z^{n+8} p_v(z),
\]
so
\[
U_n(z,1/2) - V_n(z,1/2) = p_\omega(z) + z^{n+8}(p_u(z) - p_v(z)+1),
\]
where $p_\omega(z) = -1 + 2 z - 2 z^2 + 2 z^5 - 2 z^8$.  Recall 
from the definition of $\omega$ that $p_\omega(\omega) = 0$.  
Since $U_n$ and $V_n$ have length $m+n+8$, to show that 
this pair gives a trap-like vector for some $z$, we'll be considering 
the expression
\[
z^{-m-n-8}(U_n(z,1/2) - V_n(z,1/2)) = \frac{p_\omega(z)}{z^{m+n+8}} + z^{-m}(p_u(z)-p_v(z)+1).
\]

We now show how to certify that this vector is trap-like for $z$ of 
the form $\omega + C\omega^n$, for sufficiently large $n$.  
We therefore consider
\[
\frac{p_\omega(\omega+C\omega^n)}{(\omega + C\omega^n)^{m+n+8}} 
   + (\omega+C\omega^n)^{-m}(p_u(\omega+C\omega^n)-p_v(\omega+C\omega^n)+1).
\]
Note that the right summand converges to $\omega^{-m}(p_u(\omega)-p_v(\omega)+1)$ as $n\to\infty$. 
We claim the left summand converges as well.  To see this, we expand it out 
using the definition of $p_\omega$:
\begin{align*}
\frac{p_\omega(\omega+C\omega^n)}{(\omega + C\omega^n)^{m+n+8}} & = 
(-1 + 2\omega - 2\omega^2 + 2\omega^5 - 2\omega^8)\frac{1}{(\omega+C\omega^n)^{m+n+8}} \\
& + 2C\omega^{-m-8}(1-2\omega+5\omega^4-8\omega^7)\frac{\omega^n}{(\omega+C\omega^n)^{n}}\\
& +\frac{- 2C^2\omega^{2n} + 2C^5\omega^{5n} - 2C^8\omega^{8n} + 20C^2\omega^{3+2n} - 56C^2\omega^{6+2n}}{(\omega+C\omega^n)^{m+n+8}} \\
& +\frac{20C^3\omega^{2+3n} - 112C^3\omega^{5+3n} + 10C^4\omega^{1+4n}-140C^4\omega^{4+4n}}{(\omega+C\omega^n)^{m+n+8}} \\
& +\frac{-112C^5\omega^{3+5n}-56C^6\omega^{2+6n}-16C^7\omega^{1+7n}}{(\omega+C\omega^n)^{m+n+8}} \\
\end{align*}
The first line is $0$ because, recall, $p_\omega(\omega) = 0$, and it's straightforward 
to see that $\lim_{n\to\infty}\omega^n/(\omega+C\omega^n)^n = 1$, so the last 
three lines converge to $0$ and the second line converges to $2C\omega^{-m-8}(1-2\omega+5\omega^4-8\omega^7)$.

Therefore, if the hypothesis of the lemma holds, then for sufficiently large $n$, 
the words $U_n$ and $V_n$ are trap like for $\omega + C\omega^n$, as claimed.

To get the last statement of the lemma, observe that the vector $t$ varies 
linearly with $C$, so certainly for any $C' \in B_\epsilon(C)$, the hypotheses of 
the lemma are satisfied.  But note that all the expressions above are uniformly 
continuous in $C$ on compact subsets, so given any compact subset, there is a 
uniform bound on the value of $n$ required.
\end{proof}

To complete the proof of Theorem~\ref{theorem:hole_limit}, then, it suffices to 
exhibit a loop of overlapping balls output by Lemma~\ref{lemma:limit_trap} 
encircling $\omega$.  Because there are finitely many balls, there is a uniform 
$N$ such that for $n>N$, there exists a trap for $\omega+C\omega^n$ for every $C$ 
in every ball in this loop.  In other words, the image of this loop under the 
map $x \mapsto \omega(x-\omega) +\omega$ lies in the interior of $\SetA$ for 
all sufficiently large iterates. Figure \ref{figure:limit_trap_loop} 
shows the loop of trap balls which we computed.

\begin{figure}[htpb]
\centering
\includegraphics[scale=0.5]{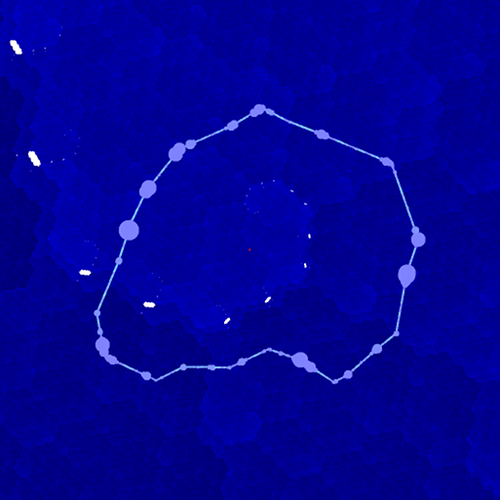}
\caption{A loop of limit trap balls encircling $\omega$.}
\label{figure:limit_trap_loop}
\end{figure}

\begin{remark}
Lemma~\ref{lemma:limit_trap} only states that this loop is \emph{eventually} in the 
interior of $\SetA$ (under a large enough iterate of the map $x \mapsto \omega(x-\omega) + \omega$).  
However, experimentally, this loop lies in the interior for all iterates. 
The primary evidence for this is that a picture of limit traps near $\omega$ 
looks the same as a picture of regular traps.
\end{remark}

\subsection{Renormalization}

\subsubsection{Introduction}

In this section, we place the above example in a more formal 
context and explain the relationship with the work of Solomyak in~\cite{Solomyak}.
We first give a heuristic explanation of some of our definitions.
We would \emph{like} to define a renormalization operator $R: \Sigma \times \Sigma \times \D^* \to \D^*$ 
such that $R(u,v,z)$ is the parameter $w$ such that the limit set for 
$w$ is the same, in some sense, as 
the union $u(z,\Lambda_z) \cup v(z,\Lambda_z)$.  The right definition 
for this operator is elusive.  However, we show below that we can 
understand what the fixed points of renormalization should be, 
and at these fixed points, there is a sensible definition of a 
limiting trap.  For certain renormalization points, 
we give a new interpretation of a result of Solomyak~\cite{Solomyak}.
 
Let $u$ and $v$ be given words of the same length.  These 
will be our prefixes.  Let $s$ and $t$ be two other words 
of the same length.  We are interested in the appearance of 
the set $us^n(z,\Lambda_z) \cup vt^n(z,\Lambda_z)$ 
and renormalization with respect to the words $us^n$ and $vt^n$ 
as $n \to \infty$.
As $n$ gets large, renormalization at $us^n$, $vt^n$ 
should converge to a locally-defined holomorphic 
function which, abusing notation, we'll call 
renormalization at (the now infinite words) $us^\infty$, $vt^\infty$.
Parameters $z$ for which $\pi(us^\infty,z) = \pi(vt^\infty,z)$ 
should be the fixed points of this renormalization.

Therefore, we say that a parameter $z$ is a \emph{renormalization point}
if there are words $u,v,s,t$ as above such that 
$\pi(us^\infty,z) = \pi(vt^\infty,z)$.
We will show that there is a notion of a limit trap at a 
renormalization point and that this can sometimes 
give an asymptotic self-similarity.

\subsubsection{A computation}

This section is essentially concerned with the behavior 
of the limit set $\Lz$ at infinitesimal scales for renormalization 
points.  That $\omega$ is a renormalization point means that 
$f\Lambda_\omega \cap g\Lambda_\omega \ne \varnothing$ 
and in fact there are two eventually 
periodic words $u,v$ so $\pi(u,\omega) = \pi(v,\omega)$.  We want to zoom 
in on this point of intersection.  Recall that for a finite (or infinite) 
word $u \in \Sigma_n$, we can write $u(z,x) = xz^n + p_u(z)$, where 
$p_u$ is a polynomial of degree $n$ (if $u$ is infinite, $p_u(z)$ is the 
power series $\pi(u,z)$).  We take the convention that 
if $u$ has length $0$, then $p_u(z) \equiv 0$.  If $u,v \in \Sigma_n$, then $u(z,\Lz)$ and 
$v(z,\Lz)$ are translates of each other, and the displacement vector is 
$p_u(z) - p_v(z)$.  A more useful quantity turns out to be the displacement 
relative to the sizes of the sets $u(z,\Lz)$ and $v(z,\Lz)$, that is
\[
z^{-n}(p_u(z)-p_v(z)).
\]
We have already encountered this expression several times.
As in the proof of Theorem~\ref{theorem:hole_limit}, we will need to 
compute its value for parameters of the form $\omega + C\omega^n$ for 
long words.  This section contains a rather tedious computation 
which will be necessary for its generalization.

\begin{lemma}\label{lemma:sum_computation}
Let $u,v$ have length $a$; let $s,t$ have length $b$; 
and let $x,y$ have length $c$.  Let $\omega$ be a renormalization 
point for $u,v,s,t$.  Write $P(z) = p_{us^\infty}(z) - p_{vt^\infty}(z)$, 
so $P(\omega) = 0$.  Then as $n \to \infty$, the 
quantity
\[
(\omega + C\omega^{nb})^{-(a+bn+c)}\left( p_{us^nx}(\omega + C\omega^{nb}) - 
                                            p_{vt^ny}(\omega + C\omega^{nb}
                                         \right)
\]
converges to 
\[
\omega^{-a-c}(p_u(\omega) - p_v(\omega)) + \omega^{-c}(p_x(\omega) - p_y(\omega)) + \omega^{-a-c}CP'(\omega)
\]
\end{lemma}
\begin{proof}[Proof of Lemma~\ref{lemma:sum_computation}]
First, some notation.  Write $d_i$ for the coefficients of the 
power series $P(z)$, so $P(z) = \sum_{i=1}^\infty d_i z^i$.
Note that $d_i$ is periodic with period $b$ for large enough $i$; 
write $P_a(z)$ to mean the eventually periodic part of $P(z)$, 
shifted by $a$, so 
\[
P_a(z) = \sum_{i=0}^\infty d_{i+a+bn} z^i.
\]
Where $n$ is taken large enough that the coefficients are 
constant in $n$.  
If we take a finite power for $s$ and $t$, the resulting polynomial 
(which has degree $a+bn$)
will agree with $P(z)$ to the term with degree $a+bn-1$, so define 
$r \in \{\pm 2, \pm 1, 0\}$ so that 
\[
p_{us^n}(z)-p_{vt^n}(z) = \sum_{i=0}^{a+bn-1}d_iz^i + rz^{a+bn}.
\]
Observation of the power series $P$ shows the facts (the third following from the 
first two):
\begin{align*}
P_a(z) &= p_{s^\infty}(z) - p_{t^\infty}(z) + r \\
P(\omega) &= 0 = p_u(\omega) - p_v(\omega) + \omega^{a}(p_{s^\infty}(\omega) - p_{t^\infty}(\omega)) \\
r - P_a(\omega) &= \omega^{-a}(p_u(\omega) - p_v(\omega))
\end{align*}

We will soon encounter some rather large expressions, and it will 
be helpful to use some small notation.  We denote the 
expression in the lemma by $E_n$, so 
\[
E_n = (\omega + C\omega^{nb})^{-(a+bn+c)}\left( p_{us^nx}(\omega + C\omega^{nb}) - 
                                            p_{vt^ny}(\omega + C\omega^{nb}\right),
\]
and we denote $\omega+C\omega^{bn}$ by $\Omega_n$.
Recall that $\lim_{n\to\infty}\omega^{bn}/\Omega_n^{bn} =1$.
We expand using the fact that 
$p_{us^nx}(z) = z^{a+bn}p_x(z) + p_{us^n}(z)$:
\begin{align*}
E_n & = \Omega_n^{-(a+bn+c)}\left( \Omega_n^{a+bn}(p_x(\Omega_n)-p_y(\Omega_n)) + 
        p_{us^n}(\Omega_n) - p_{vt^n}(\Omega_n) \right) \\
& = \Omega_n^{-c}(p_x(\Omega_n) - p_y(\Omega_n)) + r\Omega_n^{-c}+ 
\Omega_n^{-(a+bn+c)} \sum_{i=0}^{a+bn-1} d_i \Omega_n^i 
\end{align*}
The first part trivially converges to 
$\omega^{-c}(p_x(\omega)-p_y(\omega)) + r\omega^{-c}$ 
as $n\to\infty$.  We will show that 
\[
\Omega_n^{-(a+bn+c)} \sum_{i=0}^{a+bn-1} d_i \Omega_n^i  \quad \longrightarrow \quad
-\omega^{-c}P_a(\omega) + C\omega^{-(a+c)}P'(\omega).
\]
To do this, we expand the term
$\Omega_n^i = (\omega + C\omega^{bn})^i$ using the binomial theorem:
\begin{align*}
\Omega_n^{-(a+bn+c)} \sum_{i=0}^{a+bn-1} d_i \Omega_n^i &= \\
&\phantom{=\,\,} \Omega_n^{-(a+bn+c)}\sum_{i=0}^{a+bn-1} d_i\omega^i  \numberthis\\
&+ \Omega_n^{-(a+bn+c)}\sum_{i=1}^{a+bn-1} d_i i C\omega^{bn+i-1} \numberthis\\
&+ \Omega_n^{-(a+bn+c)}\sum_{i=2}^{a+bn-1} d_i \sum_{j=0}^{i-2}\binom{i}{j}C^{i-j}\omega^{bn(i-j)+j} \numberthis\\
\end{align*}
We handle these summand-by-summand.  First, we rewrite (1) using the fact that 
$P(\omega) = 0$ so $\sum_{i=0}^{a+bn-1}d_i\omega^i = -\sum_{i=a+bn}^\infty d_i\omega^i$, so 
\begin{align*}
\Omega_n^{-(a+bn+c)}\sum_{i=0}^{a+bn-1} d_i\omega^i 
&= -\Omega_n^{-(a+bn+c)}\sum_{i=a+bn}^\infty d_i\omega^i \\
&= -\frac{\omega^a}{\Omega_n^{a+c}} \frac{\omega^{bn}}{\Omega_n^{bn}} \sum_{i=0}^\infty d_{i+a+bn} \omega^i\\
&\to -\omega^{-c}P_a(\omega)
\end{align*}

Next, summand (2):
\begin{align*}
\Omega_n^{-(a+bn+c)}\sum_{i=1}^{a+bn-1} d_i i C\omega^{bn+i-1} 
&= \Omega_n^{-(a+c)}\frac{\omega^{bn}}{\Omega_n^{bn}}C\sum_{i=1}^{a+bn-1}d_ii\omega^{i-1} \\
&\to \omega^{-(a+c)} C P'(\omega)
\end{align*}
Finally, summand (3).  We will prove that it converges to $0$.
First, we bound the absolute value of the 
innermost sum.  To do this, we pull out terms from the binomial 
coefficient to re-express it as a different binomial coefficient, 
so we can collapse the sum into a power.  In the first line, we use 
the fact that $\binom{i}{j} = i(i-1)(i-j)(i-j-1)\binom{i-2}{j}$, and 
$i-1,i-j,i-j-1\le i$:
\begin{align*}
\left|\sum_{j=0}^{i-2}\binom{i}{j}C^{i-j}\omega^{bn(i-j)+j} \right|
&\le i^4|C|^2|\omega|^{2bn}\sum_{j=0}^{i-2}\binom{i-2}{j}\left|C^{i-2-j}\omega^{bn(i-2-j)+j}\right| \\
&= i^4|C|^2|\omega|^{2bn}(|\omega| + |C\omega^{bn}|)^{i-2}
\end{align*}
So the entire summand (3) is bounded in absolute value by
\begin{align*}
&\hspace{5mm}|\Omega_n|^{-(a+bn+c)}\sum_{i=2}^{a+bn-1} |d_i| i^4|C|^2|\omega|^{2bn}(|\omega| + |C\omega|^{bn})^{i-2} \\
&= |\Omega_n|^{-(a+c)}\frac{|\omega|^{bn}}{|\Omega_n|^{bn}}|\omega|^{bn} |C|^2\sum_{i=2}^{a+bn-1} |d_i|i^4 (|\omega| + |C\omega|^{bn})^{i-2} \\
\end{align*}
Let $H(z) = \sum_{i=2}^\infty |d_i|i^4z^{i-2}$.  Using the root 
test, it is easy to see that $H(z)$ is uniformly convergent for 
$|z|<1$, so $H$ is uniformly convergent in a neighborhood of 
$|\omega|$.  Therefore, as $n \to \infty$, the above expression 
converges to
\begin{align*}
&\to |\omega|^{-(a+c)}\left(\lim_{n\to\infty}\frac{|\omega|^{bn}}{|\Omega_n|^{bn}}\right)\left(\lim_{n\to\infty}|\omega|^{bn}\right)|C|^2H(|\omega|)\\
&= |\omega|^{-(a+c)}(1)(0)|C|^2H(|\omega|) \\
&=0 
\end{align*}
We have now shown that as $n\to\infty$ 
\[
E_n \to \omega^{-c}(p_x(\omega)-p_y(\omega)) + r\omega^{-c}  
-\omega^{-c}P_a(\omega) + C\omega^{-(a+c)}P'(\omega).
\]
Using the observations about $P$ at the beginning of the proof, this 
expression rearranges into the statement of the lemma.
\end{proof}

\subsubsection{Similarity}
Recall from Section~\ref{section:differences} 
that the set of differences between points in $\Lz$ is
$\Gamma_z$, the limit set generated by the 
three contractions
\[
x \mapsto z(x+1)-1 \qquad x \mapsto zx \qquad x \mapsto z(x-1)+1
\]
\begin{theorem}[Renormalizable traps]\label{theorem:similarity}
Suppose that $\omega$ is a renormalization point for 
$u,v,s,t$, where $s,t$ have length $b$. 
 Let $P(z) = p_{us^\infty}(z) - p_{vt^\infty}(z)$.
Let $T_\omega$ denote 
$-\frac{p_u(\omega)-p_v(\omega)}{P'(\omega)} - 
\frac{\omega^a}{P'(\omega)}\Gamma_\omega$, the translated, scaled copy of $\Gamma_\omega$
\begin{enumerate}
\item If $C \in T_\omega$, then for all $\epsilon>0$, there is a 
$C'$ such that $|C-C'|<\epsilon$ and for all sufficiently 
large $n$, there is a trap for $\omega+C'\omega^{bn}$.
\item If there is a unique pair of infinite words  
$U,V \in \partial \Sigma$  such that $p_U(\omega) = p_V(\omega)$ 
(i.e. $U=us^\infty$, $V=vt^\infty$), 
then there is $\delta>0$ such that for all $C \notin T_\omega$ 
with $|C| <\delta$, the limit set for the parameter 
$\omega + C\omega^{bn}$ is disconnected for all sufficiently large $n$.
\end{enumerate}
\end{theorem}

\begin{remark}
A version of part (2) of Theorem~\ref{theorem:similarity}
still holds if there are finitely many such infinite $U,V$, 
as long as they are eventually 
periodic.  In this case, we need to replace $T_\omega$ with a 
union of multiple scaled, translated copies of 
$\Gamma_\omega$.
\end{remark}

\begin{remark}
We can think of Theorem~\ref{theorem:similarity} as the verification of a kind of
``Renormalized Bandt's Conjecture''. It says that at a renormalizable point $\omega$, there
is an increasing union of open subsets of renormalizable traps, limiting to the
asymptotically scaled copy of $\SetA$ centered at $\omega$.
It implies (but is stronger than) one of the main consequences of
Theorem~2.3 from Solomyak \cite{Solomyak}, that suitable
neighborhoods of zero in $T_\omega$ converge in the sense of Hausdorff distance 
to suitably scaled neighborhoods of $\omega$ in $\SetA$.

In contrast to Solomyak, our argument is more closely expressed in the language of
algorithms, since one of our aims was always to use this theorem to obtain numerical
certificates of the existence of hole spirals. This is stated carefully in
Lemma~\ref{lemma:limit_traps_general}.
\end{remark}

\begin{lemma}\label{lemma:limit_traps_general}
Let $u,v$ have length $a$; let $s,t$ have length $b$, and 
let $x,y$ have length $c$. Let $\omega$ be a 
renormalization point for $u,v,s,t$.  Write 
$P(z) = p_{us^\infty}(z) - p_{vt^\infty}(z)$.  Suppose that $C$ is such that 
the vector
\[
\omega^{-a-c}(p_u(\omega)-p_v(\omega)) +
 \omega^{-c}(p_x(\omega)-p_y(\omega)) + \omega^{-a-c}CP'(\omega)
\]
is trap-like for $\omega$.  Then the words $us^nx$, $v,t^ny$ 
give a trap for $\omega + C\omega^{bn}$ for all sufficiently 
large $n$.
\end{lemma}
\begin{proof}
This is essentially immediate from Lemma~\ref{lemma:sum_computation}, 
which says that the vector which determines whether 
$us^nx$, $vt^ny$ give a trap converges to the above expression 
as $n$ gets large.  Hence, if the above is trap like, 
we get a trap for $\omega + C\omega ^{bn}$ for all $n$ large enough.
\end{proof}

If Lemma~\ref{lemma:limit_traps_general} holds for 
some point $\omega$ and $C$, we say that $C$ admits a 
\emph{limit trap} for $\omega$.

\begin{proof}[Proof of Theorem~\ref{theorem:similarity}]
We first prove part (1).  Let us be given $C \in T_\omega$.
By Lemma~\ref{lemma:limit_traps_general}, if the vector:
\[
\omega^{-a-c}(p_u(\omega)-p_v(\omega)) +
 \omega^{-c}(p_x(\omega)-p_y(\omega)) + \omega^{-a-c}KP'(\omega)
\]
is trap-like for $\omega$, then $K$ admits a limit trap.
Let $T$ be a trap-like vector.  Then we can solve for the 
associated value $C'$ which admits a limit trap:
\[
C' = \omega^{a+c}\frac{T}{P'(\omega)} - 
\frac{p_{u}(\omega)-p_{v}(\omega)}{P'(\omega)} -  
\frac{\omega^{a}}{P'(\omega)}(p_x(\omega)-p_y(\omega))
\]
As $c$ grows and $x$ and $y$ vary over all words of length $c$, 
the first summand goes to zero, and the 
second two together converge (in the Hausdorff topology, say, but 
quite regularly) to $T_\omega$.  
Hence if $C \in T_\omega$, then for any $\epsilon>0$, 
there are words $x,y \in \Sigma_c$ so that 
$C'$ admitting a limit trap as above has $|C-C'|<\epsilon$.
This completes the proof of part (1).

Now we prove part (2).  
When we run Algorithm~\ref{algorithm:disconnectedness} on $\omega$, 
the stack 
entries at stage $a+bn$ are exactly the scaled differences 
$\omega^{-a-bn}(p_x(\omega)-p_y(\omega))$ 
between centers of words $x,y$ of length $a+bn$ (when these 
differences are small enough to remain on the stack).
If there is a unique pair of 
words $U,V$ such that $p_U(\omega) = p_V(\omega)$, then
there is a single 
stack entry 
with infinitely viable children, and it is
$\omega^{-a-bn}(p_{us^n}(\omega)-p_{vt^n}(\omega))$. 
Rewriting this as in the proof of Lemma~\ref{lemma:sum_computation}, 
we see that by making $n$ large, this expression is as close as 
we'd like to $\omega^{-a}(p_u(\omega) - p_v(\omega))$.

When we vary $\omega$ to $\omega + C\omega^{a+bn}$, 
and make $n$ large, then by Lemma~\ref{lemma:sum_computation}, 
we can make this stack entry
as close as we like to 
\[
\omega^{-a}(p_u(\omega) - p_v(\omega)) + \omega^{-a}CP'(\omega)
\]
Therefore, there is a $\delta>0$ as in the statement of the 
theorem such that if $|C|<\delta$, then when we run the disconnectedness 
algorithm on the input $\omega + C\omega^{a+bn}$, the stack at 
step $a+bn$ has the entry (as close as we want to) 
$\omega^{-a}(p_u(\omega) - p_v(\omega)) + 
\omega^{-a}CP'(\omega)$, and every other stack entry has 
children which are eliminated in finite time.  The value for $\delta$ 
can be found by checking how far the limiting entry 
$\omega^{-a}(p_u(\omega) - p_v(\omega))$ 
is from the cutoff; then make $\delta$ 
small enough so that adding the term $\omega^{-a}CP'(\omega)$ 
does not push anything off of or onto the stack.

Now, compute all possible 
children after $c$ more steps; by Lemma~\ref{lemma:sum_computation}, 
we get 
\begin{align*}
X_{x,y} = \omega^{-a-c}(p_u(\omega)-p_v(\omega)) +
 \omega^{-c}(p_x(\omega)-p_y(\omega)) + \omega^{-a-c}CP'(\omega),
\end{align*}
where $x,y$ vary over all words of length $c$.
We rearrange:
\[
\omega^{a}\frac{X_{x,y}}{P'(\omega)} = 
\omega^{-c}\left( C - \left( -\frac{p_u(\omega)-p_v(\omega)}{P'(\omega)} 
-\frac{\omega^a}{P'(\omega)}(p_x(\omega)-p_y(\omega)\right)\right)
\]
However, the fact that $C$ is not in $T_\omega$ 
means that as we increase $c$, the minimum value of 
quantity on the right above goes to infinity.  
Thus, $\min_{x,y}X_{x,y} \to\infty$.  Hence, at 
some finite $c$, every one of these children has left the stack.

Recall the stack entries above are limits of the real 
stack entries we see for step $a+bn+c$, but by choosing 
$n$ large enough, we can make the computation valid (because $c$ is 
some finite number, so there are finitely many quantities to 
bring close to their limits).
Hence for $n$ large enough,
the disconnectedness algorithm certifies that 
the limit set for $\omega+C\omega^{bn}$ is disconnected.
\end{proof}

Figure~\ref{figure:C_values_for_hexahole} shows an example of 
$T_\omega$ near $0$ for the renormalization point 
in Theorem~\ref{theorem:hole_limit}.  See also 
the pictures in \cite{Solomyak}.

\begin{figure}[htb]
\begin{center}
\includegraphics[scale=0.27]{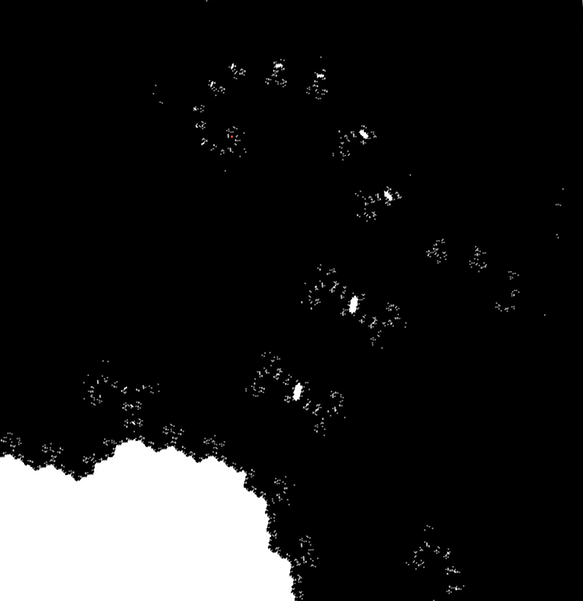}
~~\includegraphics[scale=0.276]{hexahole_spiral}
\end{center}
\caption{A portion of the limit set $T_\omega$ near $0$ (left) for 
$\omega \approx 0.371859+0.519411i$ and 
set $\SetA$ near $\omega$ (on right).} 
\label{figure:C_values_for_hexahole}
\end{figure}

We end this section by proposing two (related) conjectures:

\begin{conjecture}\label{conjecture:algebraic_points}
The algebraic points in $\partial \SetA$ are dense in $\partial \SetA$.
\end{conjecture}

\begin{conjecture}\label{conjecture:limits_of_holes}
Every point in $\partial \SetA$ not on the real axis is a 
limit of a sequence of holes with diameters going to zero.
\end{conjecture}

We believe that fixed points of renormalization are the key to both conjectures;
such fixed points are on the one hand algebraic, and on the other hand points where
$\SetA$ is asymptotically self-similar, and asymptotically similar to the limit
set of a 3-generator IFS. It is very easy for a connected limit set of
a 3-generator IFS to fail to be simply-connected: irregularities in the frontiers of
the translates overlap each other in complicated ways, cutting off tiny holes. Once there
is one tiny hole, there will be infinitely many, accumulating densely in the boundary of the
limit set; thus one expects the corresponding point in $\SetA$ to be a limit of tiny holes.

The experimental evidence for Conjecture~\ref{conjecture:limits_of_holes} is ambiguous.
On the one hand, a computer-aided search (using {\tt schottky}) will only reveal the holes at
any scale that are big enough to see, so one must develop heuristics to identify promising
regions for exploration. On the other hand, failure to find holes near some given frontier
point does not rule out the possibility that they might exist, but be very elusive.

In a private communication, Boris Solomyak suggested that there might be no tiny holes
accumulating at the point $i/\sqrt{2}$ in $\partial \SetA$; this is an especially good candidate
counterexample to Conjecture~\ref{conjecture:limits_of_holes}, 
since although it is algebraic --- and in fact a fixed point
of renormalization --- the limit set of the corresponding 3-generator IFS {\em is} full,
and in fact {\em convex}. Thus one could not hope to prove the existence of a renormalization
sequence of holes, certified by loops of limit traps, limiting to $i/\sqrt{2}$. On the other
hand, very small holes {\em can} be found by hand, as close to $i/\sqrt{2}$ as the resolution
allows --- the (numerically certified) hole at $0.02269108+0.70320806i$ is a good example.

\section{Whiskers}\label{section:whiskers}

In this section we discuss the subtle problem of the structure of $\SetA$ and $\SetB$ near the
real axis.

\subsection{Whiskers are isolated}\label{subsection:whiskers_isolated}

In light of Theorem~\ref{theorem:interior_dense} it might be surprising that the structure of
$\SetA$ and $\SetB$ near the real axis can be very complicated. 
In fact, as was already observed by Barnsley-Harrington \cite{Barnsley_Harrington},
there is an open neighborhood of the points $\pm 1/2$ in $\pm [1/2,1/\sqrt{2}]$ in which $\SetA$
is totally real. We give an elementary proof of this fact, using
the description of the limit set $\Lz$ as the values of certain power
series in $z$, as described in Section~\ref{subsection:coefficients}. Getting a better estimate
depends on analyzing a real 2-dimension IFS introduced by Shmerkin-Solomyak 
\cite{Shmerkin_Solomyak} which we discuss and study in Section~\ref{subsection:2d_IFS}.

\begin{lemma}[Whiskers isolated]\label{lemma:whiskers_isolated}
There is some $\alpha>1/2$ so that the intersection of $\SetA$ with some open subset of $\C$ is
equal to the interval $[1/2,\alpha)$.
\end{lemma}
\begin{proof}
Recall that for $e \in \partial \Sigma$ the image $\pi(e,z) \in \Lz$ is the value of the power
series $\pi(e,z):= a_0 + a_1z + a_2z^2 + \cdots$ where the coefficients $a_i$ are determined recursively
from the infinite word $e$ by the method in Proposition~\ref{proposition:power_series}. The
key point is that the nonzero coefficients alternate between $1$ and $-1$, starting with $1$.

Let $z=1/2+\epsilon$ be real, for some small
positive $\epsilon$. The limit set $\Lz$ is exactly equal to the unit interval, and
$f\Lz = [0,z]$, $g\Lz = [1-z,1]$ so that the intersection is exactly the
interval $[1/2-\epsilon,1/2+\epsilon]$. The words $e$ with $\pi(e,z)$ in the overlap all
start with $fg^n$ or $gf^n$ for some big $n$ (depending only on $n$) so that the power
series are of the form $z-z^{n+1}+\cdots$ or $1-z^{n+1}+z^{m} -\cdots$ depending whether $e$ starts
with $f$ or $g$, and in the second case $m>n+1$ (we include the possibility that $m=\infty$). 
In the first case, $d\pi(e,z)/dz = 1 - (n+1)z^{n+1} + \cdots > 0.1$, while in
the second case $d\pi(e,z)/dz = -(n+1)z^n + mz^{m-1} - \cdots < 0$ for big $n$ and any fixed
$z<1$. Since the derivative is holomorphic in $z$, this means that if we perturb $z$ to
$z+i\delta$ for some small positive $\delta$, the imaginary part of $\pi(e,z)$ becomes {\em positive}
for $e$ beginning with $f$, and {\em negative} for $e$ beginning with $g$ (at least for
$\pi(e,z)$ close to the interval $[1/2-\epsilon,1/2+\epsilon]$), so that the two
sets $f\Lz$ and $g\Lz$ are disjoint, and we are in the complement of $\SetA$.
\end{proof}

\subsection{A 2-dimensional IFS}\label{subsection:2d_IFS}

We will push this argument further by analyzing the pairs 
$(\pi(u,z),d\pi(u,z)/dz)$ and $(\pi(v,z),d\pi(v,z)/dz)$ for
left-infinite words $u,v \in \partial \Sigma$ starting with $f$ and $g$ respectively, and showing that
for all real $z$ in the interval $[0.5,0.6684755]$ the pairs are disjoint. 

Shmerkin-Solomyak \cite{Shmerkin_Solomyak} introduce a 2-dimensional {\em real} IFS acting
on $\R^2$ whose limit set is precisely the pairs $(\pi(u,z),d\pi(u',z)/dz)$ for $u\in \partial \Sigma$. Explicitly, 
for real $z \in (-1,1)$, define
$$f^{(1)}: (x,y) \to (zx, x+zy), \quad g^{(1)}:(x,y) \to (z(x-1)+1,x-1+zy)$$
and let $\tdLz$ denote the limit set of the IFS 
generated by $f^{(1)}$ and $g^{(1)}$ (the notation is supposed to suggest the action of our
familiar $f$ and $g$ on 1-jets). 
Analogous to our standard notation, we will write $u(z,x)$ for the 
action of the word $u\in\Sigma$ on $x \in \R^2$ for a 
parameter $z \in \R$.  Also, we write 
$\pi(u,z) = \lim_{n\to\infty}u_n(z,x)$, where the limit does not 
depend on $x$.

\begin{lemma}\label{lemma:2d_ifs}
Let $z \in \R$ and suppose $f^{(1)}(z,\tdLz)$ and $
g^{(1)}(z,\tdLz)$ are disjoint. Then 
$\SetA$ is totally real in an open neighborhood of $z$. 
\end{lemma}
\begin{proof}
This is the same argument as that in used in the proof of  
Lemma~\ref{lemma:whiskers_isolated}.
\end{proof}

Since this condition is open, it can be certified numerically.
Thus, if we define $\Omega_2$ to be the subset of $z \in (-1,1)$ for which $\tdLz$ is
connected, then $\overline{\SetA - \R} \cap \R \subseteq \Omega_2$. One can characterize
$\Omega_2$ as the set of real numbers $z$ of absolute value at most $1$ for which there is
some power series $\zeta(z):= 1 + \sum_{n=1}^\infty a_nz^n$ where each $a_n \in \lbrace -1,0,1\rbrace$
for which $\zeta(z)=\zeta'(z)=0$.
We discuss later the question of whether there are points in $\Omega_2$ 
which do not lie in the closure of the interior of $\SetA$.

Analogous to $\Omega_2$, one can study the subset $\Xi_2 \subseteq (-1,1)$ consisting of $z$ for which 
$\tdLz$ contains the point $(1/2,0)$, and then $\overline{\SetB - \R}\cap \R \subseteq
\Xi_2$.

Shmerkin-Solomyak \cite{Shmerkin_Solomyak} define $\alpha$ to be the smallest positive
real number in $\Omega_2$, and $\tilde{\alpha}$ to be the smallest real number such that
$[\tilde{\alpha},1) \subseteq \Omega_2$.
Experimentally they obtained estimates 
$$\alpha \sim 0.6684755, \quad \tilde{\alpha} \sim 0.67$$ 
We improved the estimate of $\alpha$ to
$$\alpha \sim 0.6684755322100605954110550451436814$$ 
Getting a rigorous estimate of $\tilde{\alpha}$ is much harder, but experimentally
we obtain $\tilde{\alpha} \sim 0.6693556$.

To obtain these estimates, we used an algorithm which is perfectly analogous to 
Algorithm~\ref{algorithm:disconnectedness}, and is proved in essentially the same way.
To describe this algorithm, we use the following shorthand:
$$A:=\begin{pmatrix} z^{-1} & 0 \\ -z^{-2} & z^{-1} \end{pmatrix}, \quad \quad Z:=\begin{pmatrix} 1-z \\ -1 \end{pmatrix}$$
Furthermore, for a $2\times 1$ column vector $X$ we say $X$ is {\em small} if
$|X_1|<1$ and $|X_2| < \sup_{k\ge 1} 2k|z|^{k-1}$, where $k$ is an integer. Note that for $z$ real with 
$|z|<1$, this latter inequality reduces to the analysis of a small fixed number of cases
for $k$. In the regime in which we are interested, $z$ will be quite close to $0.66$, so the 
relevant cases are $k=2$ and $k=3$, and in practice the inequality reduces to
$|X_2|<2.681165$.

\begin{algorithm}[htpb]
\caption{No Multiple Roots$(z,\text{depth})$}\label{algorithm:2d_IFS}
\begin{algorithmic}
\State $V\gets \lbrace AZ \rbrace$
\State $d\gets 0$

\While{$V\ne \emptyset$ or $d<\text{depth}$}
	\State $W \gets \emptyset$
	\ForAll{$X \in V$}
		\If{$A(X-Z)$ is small}
			$W \gets W \cup A(X-Z)$
		\EndIf
		\If{$AX$ is small}
			$W \gets W \cup AX$
		\EndIf
		\If{$A(X+Z)$ is small}
			$W \gets W \cup A(X+Z)$
		\EndIf
	\EndFor
	\State	$V \gets W$
	\State $d \gets d+1$

\EndWhile

\If{$V=\emptyset$} 
\State \Return true
\Else
\State \Return false
\EndIf
\end{algorithmic}
\end{algorithm}

The justification for this algorithm is essentially the same as that of
Algorithm~\ref{algorithm:disconnectedness}.  
To ask whether $L_z$ is connected is equivalent to asking 
whether $f^{(1)}(z,L_z)\cap g^{(1)}(z,L_z) = \varnothing$, which is 
equivalent to asking whether the set of differences contains $0$.
Just as in Section~\ref{section:differences}, the set of 
differences between points in $L_z$ is a limit set itself.  
We denote the set of differences by $L'_z$, and we note it is 
the limit set of the IFS generated by the 
three maps 
\[
F_{-1}:X \mapsto BX - Z, \qquad F_0:X \mapsto BX, \qquad F_1:X \mapsto BX + Z,
\]
where 
\[
B:= \begin{pmatrix} z & 0 \\ 1 & z \end{pmatrix}.
\]
Note $B = A^{-1}$.
We obtain these maps by looking at how pairs of maps $(f,f)$,$(f,g)$,
$(g,f)$, 
and $(g,g)$ act on differences of points; there are only 
three distinct maps.  Since $F_1L'_z$ consists of differences between 
points in $L_z$ whose corresponding infinite words begin with $g^{(1)}$ and 
$f^{(1)}$, respectively, to check whether $L_z$ is connected 
it suffices to check whether $0 \in F_1L'_z$.

To determine whether $0 \in F_1L'_z$, we start with a box $R$ centered 
at $(0,0)$, which is sent inside itself under the 
three maps $F_{-1},F_0,F_1$.  We want to consider $F_1L'_z$, 
so first we apply $F_1$.  Next, we subdivide $F_1R$ 
into its three subboxes, which are $F_1F_{-1}R$, $F_1F_0R$, and 
$F_1F_{-1}R$, and discard 
those which cannot contain $0$.  We then subdivide again, and so on.
Suppose that $X$ is the center (image of $(0,0)$) of an image of $R$ 
under a 
word of length $n$.  Since the centers of $F_{-1,0,1}L'_z$ are 
at $-Z,0,Z$, respectively, the centers of the children of $X$
will be at the points 
$X-B^nZ$, $X$, $X+B^nZ$.  For simplicity, it makes sense to rescale the
problem at every step by $A = B^{-1}$.  Hence, we initialize the 
algorithm with the rescaled $Z$, i.e. $AZ$.  Then we add 
$-Z, 0, +Z$, and rescale by $A$ again, and so on.  Any 
child which lies too far from the origin can be discarded, which is 
exactly what the smallness condition guarantees.  
The precise constants in the smallness condition follow 
from an analysis of how the rectangle $|X_1|\le a$, $|X_2|\le b$ behaves 
under the maps $F_{-1,0,1}$.  It is easy to see that the 
infinite strip $|X_1|\le 1$ is sent inside itself, so $L'_z$ 
lies inside this strip.  Then if we consider the 
images of the four points $(-1,-b)$, $(1,-b)$, $(1,b)$, $(-1,b)$, 
we find that the image with the largest second coordinate 
is $(1,b)$ under the word $F_{-1}^kF_1^\infty$, and this image has
second coordinate $2k|z|^{k-1}$.  Therefore, if we find the $k$ 
maximizing that expression and set $b = 2k|z|^{k-1}$, then the limit 
set must lie in the rectangle $[-1,1]\times [-b,b]$.

If we run Algorithm~\ref{algorithm:2d_IFS} on our numerical value for $\alpha$, the output is
quite interesting. For the correct theoretical value of $\alpha$, 
the set $V$ of children viable to each depth will never be empty, and the same must be true
for our numerical approximation (of course, this is how we find the approximation in the first
place). But what is not obvious from the definition (although it is intuitively plausible)
is the experimental fact that the size of $|V|$ is uniformly bounded independently of the depth $d$, 
and there is apparently a {\em unique} lineage viable to infinite depth. If we denote
the children $A(X-Z)$, $AX$, $A(X+Z)$ of the vector $X$ by $L,M,R$ respectively, then the
(numerically) unique viable descendent of the initial vector $AZ$ to 194 generations is of
the form 
$$L^3 \prod_i (R^i M) \text{ for } i = 1\, 2\, 2\, 3\, 3\, 2\, 7\, 5\, 6\, 6\, 2\, 5\, 1\, 8\, 1\, 6\, 3\, 3\, 5\, 4\, 3\, 2\, 8\, 3\, 9\, 2\, 2\, 1\, 5\, 4\, 8\, 2\, 4\, 3\, 3\, 6\, 2\, 3\, 1\, 5$$
i.e.\/ the first few terms are $L L L R M R R M R R M R R R M \cdots$. One can think of the
values of $i$ as analogs of the terms in the continued fraction expansion of a number. In fact,
the analogy is quite good: if any viable sequence for an initial vector $AZ:=AZ(t)$ is 
eventually periodic, we obtain an identity of the
form $p_1(A)Z = p_2(A)Z$ for distinct 
polynomials $p_1,p_2$ with coefficients in $\lbrace -1,0,1\rbrace$, and therefore deduce that
$t^{-1}$ is a root of $p_1-p_2$ and is therefore algebraic.  The 
branching of the algorithm is shown in Figure~\ref{figure:2d_disconnected_branching}.


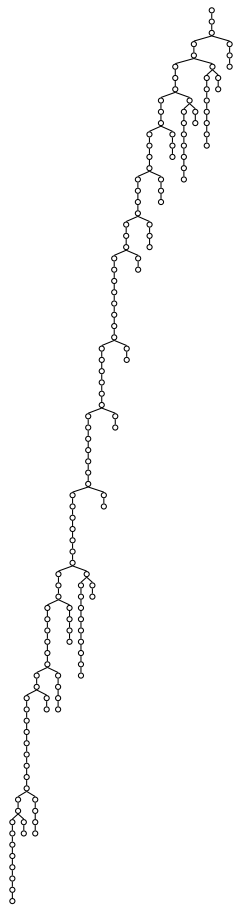
\begin{figure}
\begin{center}
\begin{tikzpicture}
\fontsize{1.5}{1.5}\selectfont
\tikzset{every tree node/.style={draw,circle},level distance=1.5mm}
\Tree [.{} [.{} [.{} [.{} [.{} [.{} [.{} [.{} [.{} [.{} [.{} [.{} [.{} [.{} [.{} [.{} [.{} [.{} [.{} [.{} [.{} [.{} [.{} [.{} [.{} [.{} [.{} [.{} [.{} [.{} [.{} [.{} [.{} [.{} [.{} [.{} [.{} [.{} [.{} [.{} [.{} [.{} [.{} [.{} [.{} [.{} [.{} [.{} [.{} [.{} [.{} [.{} [.{} [.{} [.{} [.{} [.{} [.{} [.{} [.{} [.{} [.{} [.{} [.{} [.{} [.{} [.{} [.{} [.{} [.{} [.{} [.{} [.{} [.{} [.{} [.{} [.{} [.{} [.{} {} ] ] ] ] ] ] ] [.{} {} ] ] ] [.{} [.{} [.{} {} ] ] ] ] ] ] ] ] ] ] ] ] [.{} {} ] ] ] [.{} [.{} [.{} {} ] ] ] ] ] ] ] ] ] [.{} [.{} [.{} {} ] ] ] ] ] ] [.{} [.{} [.{} [.{} [.{} [.{} [.{} [.{} [.{} {} ] ] ] ] ] ] ] ] [.{} {} ] ] ] ] ] ] ] ] ] [.{} {} ] ] ] ] ] ] ] ] [.{} {} ] ] ] ] ] ] ] [.{} {} ] ] ] ] ] ] ] ] ] [.{} {} ] ] ] ] [.{} [.{} {} ] ] ] ] ] ] [.{} [.{} {} ] ] ] ] ] ] [.{} [.{} {} ] ] ] ] ] [.{} [.{} [.{} [.{} [.{} [.{} [.{} {} ] ] ] ] ] ] [.{} {} ] ] ] ] ] [.{} [.{} [.{} [.{} [.{} [.{} [.{} {} ] ] ] ] ] ] [.{} {} ] ] ] ] [.{} [.{} {} ] ] ] ] ]
\end{tikzpicture}
\end{center}
\caption{The branching of 
Algorithm~\ref{algorithm:2d_IFS} on the (numerical) input 
$\alpha$.  The long vertical chains are all $R$, so reading down the left edge 
produces strings of $R$'s of lengths $1$,$2$,$2$,$3$,$3$,$2$,$7$, etc, agreeing 
with the ``continued fraction'' expansion of $\alpha$.}
\label{figure:2d_disconnected_branching}
\end{figure}

In view of our experimental evidence, it seems reasonable to make the following conjecture:

\begin{conjecture}[Unique lineage]\label{conjecture:unique_lineage}
For $\alpha$ as above, there is a unique child at every stage with viable descendents to
all future depths. Furthermore, this viable lineage consists of the initial segments in the sequence
$L^3 \prod_i (R^i M)$ for some sequence $i = 1\,2\,2\, \cdots$ as above, where the terms are
uniformly bounded.
\end{conjecture}

In a similar vein, we define $\beta$ to be the smallest positive real number in $\Xi_2$,
and $\tilde{\beta}$ to be the smallest real number such that $[\tilde{\beta},1) \subseteq \Xi_2$.
Using similar methods we obtain the following estimates
$$\beta \sim 0.67133041244176126776, \quad \tilde{\beta} \sim 0.728781$$
(the same caveat about $\tilde{\beta}$ applies). 
It is easy to modify Algorithm~\ref{algorithm:2d_IFS} to determine, for a given real $z$, 
when there are infinite words $u,v$ so that $\pi(u,z)=\pi(v,z)=1/2$ and 
$d\pi(u,z)/dz = d\pi(v,z)/dz=0$;
we need only consider children $AX-Z$ and $AX+Z$ for each $X$ in the stack $V$, and otherwise
the algorithm runs in exactly the same way.

Figure~\ref{numerical_alpha} gives numerical plots of the subset of the intervals
$[\beta,\tilde{\beta}] \cap \Xi_2$ and $[\alpha,\tilde{\alpha}] \cap \Omega_2$.

\begin{figure}[htpb]
\centering
\includegraphics[scale=0.4]{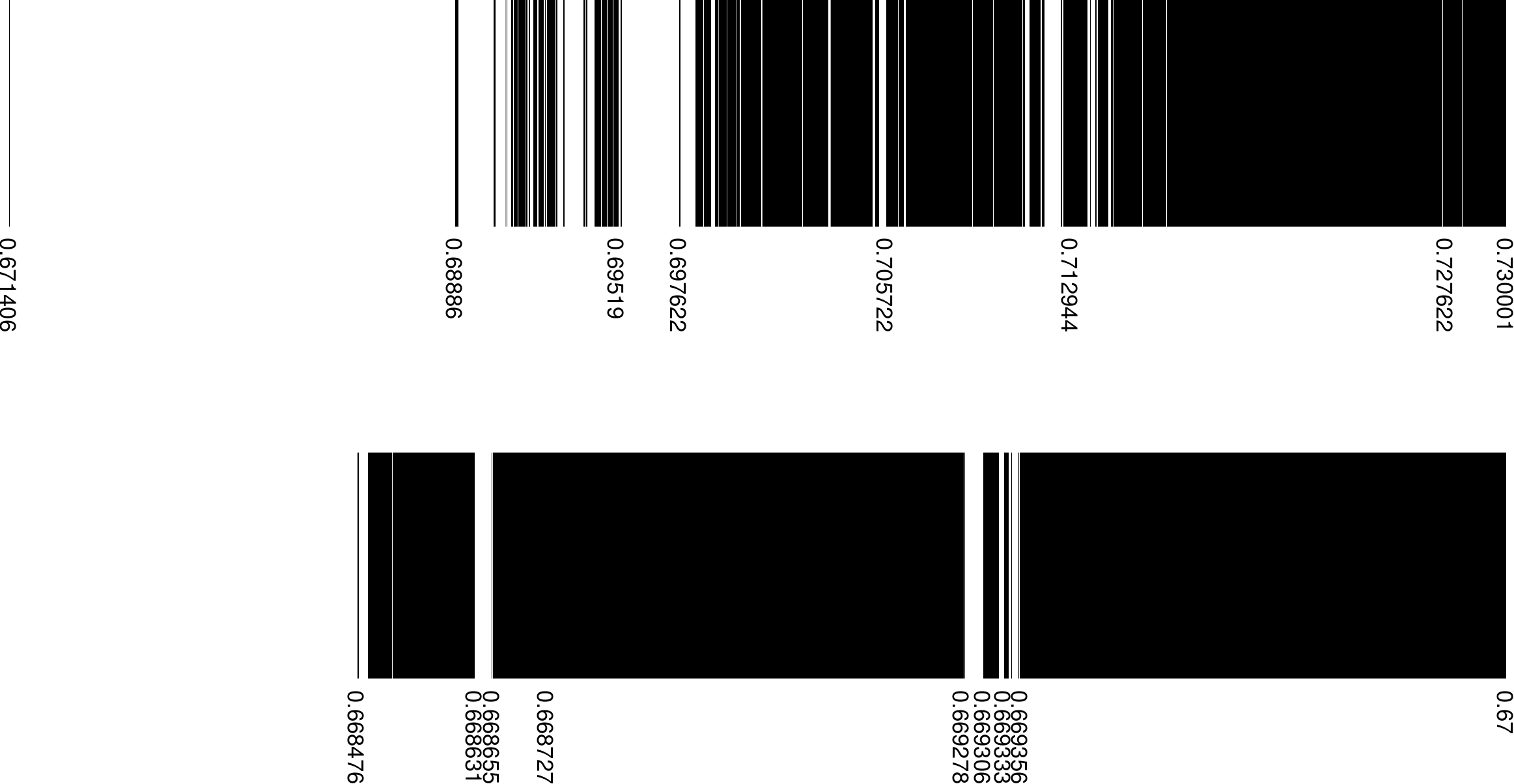}
\caption{Numerical plots of $\Xi_2$ (top) and $\Omega_2$ (bottom).}
\label{numerical_alpha}
\end{figure}

This figure strongly suggests that $\Xi_2 \cap [\beta,\tilde{\beta}]$ 
is totally disconnected, while $\Omega_2 \cap [\alpha,\tilde{\alpha}]$ 
appears to contain many solid intervals. In fact,
our method of  traps can be easily adapted to this more complicated
IFS, and in Section~\ref{subsection:intervals_omega_2} we give a method to
certify interior points in $\Omega_2$. 

\subsection{Intervals in $\Omega_2$}\label{subsection:intervals_omega_2}

Recall from Section~\ref{subsection:2d_IFS} that $\Omega_2$ is the set of positive
real numbers $z<1$ for which the IFS $\tdLz \subseteq \R^2$ generated by the
affine linear maps 
$$f^{(1)}: (x,y) \mapsto(zx,x+zy), \quad g^{(1)}:(x,y) \mapsto (z(x-1)+1,x-1+zy)$$ 
is connected. By abuse of notation, we denote these generators by $f$ and $g$ throughout this
section. Note that both generators have constant Jacobian
$$B(z):= \begin{pmatrix} z & 0 \\ 1 & z \end{pmatrix}$$
Throughout this section we restrict attention to real $z$ in the interval $[0.668,0.67]$.
The analog of Lemma~\ref{lemma:short_hop} is the following:

\begin{lemma}[Affine Short Hop Lemma]\label{lemma:affine_short_hop}
With $z \in [0.668,0.67]$, suppose that $f\tdLz$ and $g\tdLz$ 
contain points at 
distance $\delta$ apart in the $L^1$ metric on $\R^2$.
Then for any word $u$ of length at least $6$, the $0.9006\cdot\delta/2$ neighborhood of $u(z,\tdLz)$
in the $L^1$ metric is path connected.
\end{lemma}
\begin{proof}
The proof is identical to that of Lemma~\ref{lemma:short_hop}, except that one must take into
account the fact that $B(z)$ does not uniformly contract the $L^1$ metric. However, for
$z$ in the interval in question, $B(z)^n$ multiplies the $L^1$ metric by at most
$0.67^n + n\times 0.67^{n-1}$ which is $<0.9006$ for $n\ge 6$.
\end{proof}

The analog of Proposition~\ref{proposition:traps_A} is the following:

\begin{proposition}[Affine  traps]\label{proposition:affine_shrinking_trap}
Suppose for some $z \in \Omega_2$ that there are words $u,v$ beginning with $f$ and $g$
of length at least $6$ so that $u(z,\tdLz)$ and $v(z,\tdLz)$ cross transversely. Then $z$ is an
interior point in $\Omega_2$.
\end{proposition}
\begin{proof}
Since $u(z,\tdLz)$ and $v(z,\tdLz)$ cross transversely, the same is true for their $\epsilon$-neighborhoods,
for some sufficiently small fixed $\epsilon$. Thus the same is true for the
$\epsilon/2$-neighborhoods of $u(z',\tdL_{z'})$ and $v(z',\tdL_{z'})$ whenever $|z-z'|$ is small
enough, depending on $z$ and $\epsilon$. Thus, we choose such a $z'$, and
suppose $\delta$ is the $L^1$ distance from $f(z',\tdL_{z'})$ to $g(z',\tdL_{z'})$, where
$\delta \ll \epsilon$. Then the $0.9\delta/2$
neighborhoods of $u(z',\tdL_{z'})$ and $v(z',\tdL_{z'})$ are path connected, so by transversality, 
there is some point within $L^1$ distance $0.9\delta/2$ from both $u(z',\tdL_{z'})$ and $v(z',\tdL_{z'})$,
and consequently the $L^1$ distance from $u(z',\tdL_{z'})$ to $v(z',\tdL_{z'})$ is at most $0.9\delta$.
But then $\delta \le 0.9\delta$ so that $\delta=0$ and $z' \in \Omega_2$, as claimed.
\end{proof}

\begin{figure}[htpb]
\centering
\includegraphics[scale=0.45]{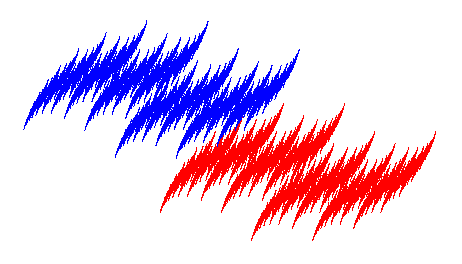}~
\includegraphics[scale=0.25]{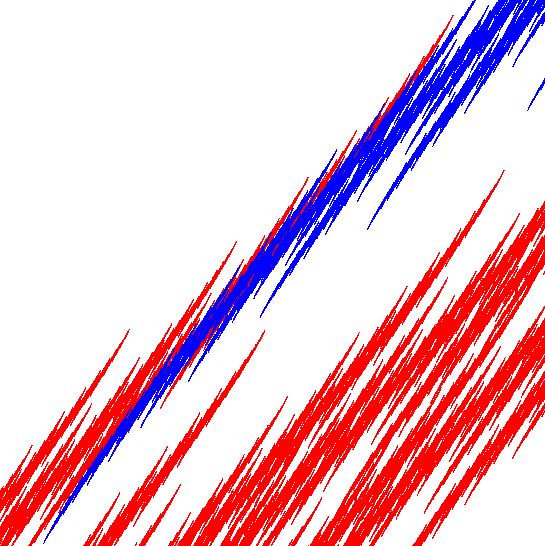}
\caption{The limit set for $z=0.669027$. 
The visually evident trap on the right certifies that this point 
lies in the interior of $\Omega_2$.}
\label{669}
\end{figure}

Such affine traps may be found and certified numerically; for example $z=0.669027$ satisfies
the proposition for the words 
\[
fgggfgfffffgfffffff \qquad \textnormal{and} \qquad gfffggggfgggggggfgg,
\]
and we deduce that $0.669027$ is an interior point in $\Omega_2$; see Figure~\ref{669}.
One might hope to prove an analog of Bandt's Conjecture (i.e.\/ Theorem~\ref{theorem:interior_dense})
for the set $\Omega_2$; that is, that the interior is dense in $\Omega_2$.  However
our proof of Theorem~\ref{theorem:interior_dense} uses in several ways the fact that 
points in the limit set are holomorphic functions of the parameter, which 
of course can no longer be true for the real parameter $z$.  Nevertheless, such a
proof does not seem beyond reach, and we comfortably conjecture:

\begin{conjecture}\label{conjecture:omega_2_interior}
Affine traps are dense in $\Omega_2$, and hence the 
interior of $\Omega_2$ is dense in $\Omega_2$.
\end{conjecture}

Recall that $\overline{\SetA - \R} \cap \R \subseteq \Omega_2$.  
It is not known whether there are any points in $\Omega_2$ which 
do not lie in the closure of the interior of $\SetA$.  
However, the following lemma relates this to 
Conjecture~\ref{conjecture:omega_2_interior}.  
For clarity, we write $u(z) = \pi(u,z)$.

\begin{lemma}\label{lemma:2dIFS_is_limit_of_nonreal}
For every $u \in \partial \Sigma$ and $b \in \D^* \cap \R$, we have that 
\[
\lim_{x+iy \to b}\frac{1}{y}\textnormal{Im}(u(x+iy)) \to  u'(b).
\]
The rate of this convergence does not depend on $u$.  Consequently, 
the limit set $\Lambda_{x+iy}$ scaled vertically by $1/y$ converges 
in the Hausdorff topology to $L_b$.
\end{lemma}
\begin{proof}
It is an easy calculus exercise to show the lemma if $x+iy$ approaches $a$ vertically, 
i.e. if $x$ is fixed at $b$.  However, we desire convergence in general, so we 
will need to look at power series.  Write $u(z)$ as the power series 
$u(z) = \sum_{k=0}^\infty a_k z^k$.  Then
\[
u(x+iy) = \sum_{k=0}^\infty a_k (x+iy)^k = \sum_{k=0}^\infty a_k \sum_{j=0}^k \binom{k}{j} (iy)^jx^{k-j}
\]
The terms which contribute to the imaginary part of this sum 
are exactly those for which $j$ is odd.  Hence
\[
\textnormal{Im}(u(x+iy)) = \sum_{k=0}^\infty a_k \sum_{\ell=1}^{2\ell+1 = k} \binom{k}{2\ell+1} (-1)^{\ell-1} y^{2\ell+1}x^{k-(2\ell+1)}
\]
The top limit on the inner sum indicates that we should run the 
inner sum until $2\ell+1$ is larger than $k$.  
Also note we are recording the imaginary part of $u(x+iy)$, 
so the $\pm i$ terms have disappeared.
Therefore,
\[
\frac{1}{y} \textnormal{Im}(u(x+iy)) = \sum_{k=0}^\infty a_k \left(kx^{k-1} + \sum_{\ell=2}^{2\ell+1 = k} \binom{k}{2\ell-1} (-1)^{\ell-1} y^{2\ell}x^{k-(2\ell+1)}\right)
\]
The entire sum is controlled in absolute value by 
$\sum_{k=0}^\infty|a_k||x+iy|^k \le \sum_{k=0}^\infty|x+iy|^k$, which is uniformly convergent for 
$x+iy \in \D^*$.  Therefore, as $x+iy\to b$, the entire sum converges, at a rate controlled 
independently of $u$, to $\sum_{k=0}^\infty a_k k b^{k-1} = u'(b)$.  The 
point in $L_b$ associated with $u$ has coordinates $(u(b), u'(b))$ in $\R^2$, and the 
point in the vertically scaled copy of $\Lambda_{x+iy}$ has coordinates 
$(\textnormal{Re}(u(x+iy)), \frac{1}{y}\textnormal{Im}(u(x+iy)))$, and the lemma follows.
\end{proof}

If affine traps are dense in $\Omega_2$, then 
near any point in $\Omega_2$, by 
Lemma~\ref{lemma:2dIFS_is_limit_of_nonreal} there are nonreal 
parameters which have a trap and therefore 
lie in the interior of $\SetA$.  So every point in $\Omega_2$ 
would be in the closure of the interior of $\SetA$; 
i.e. Conjecture~\ref{conjecture:omega_2_interior} 
implies $\overline{\SetA - \R} \cap \R = \Omega_2$.

\section{Holes in \texorpdfstring{$\SetB$}{M0}}
\label{section:holes_in_m0}

$\SetB$ is path connected \cite{Bousch1}, but Bousch's proof is somewhat indirect.
His strategy is to show that 
every point can be joined by a path to some parameter with absolute 
value close to $1$.  Since $\SetB$ contains an annulus around the unit circle, this 
gives path connectedness. He does not directly address what the paths in $\SetB$ actually
look like, or when a (polygonal) path near $\SetB$ can be approximated by a path
contained in $\SetB$.

In this section, we show how to certify the existence of a point in $\SetB$ 
in a neighborhood of a given point and how to certify a 
path in $\SetB$ in a neighborhood of a given polygonal path. If we can certify paths,
we can certify loops, and thus exotic holes in $\SetB$.  As with $\SetA$, 
by a \emph{hole} in $\SetB$, we mean a connected component of 
the complement which is distinct from the connected component of the complement 
which contains $0$.

Just as $\SetA$ closely resembles the limit set $\Gamma_z$ at many points,
$\SetB$ closely resembles $\Lz$. Thus the methods in this 
section are closely related to the methods we developed in Section~\ref{subsection:Lz_paths}
to construct paths in $\Lz$ (e.g.\/ Proposition~\ref{proposition:path_in_lambda}).

\subsection{Complex analysis}
%
%
%

In this section, we prove a lemma in complex analysis, 
but we first motivate it.  Suppose we have a holomorphic function 
$h(z)$, and we find that $h(z_0)$ is quite close 
to a value we desire $c$.  We would like to conclude that there is a $z_1$ 
near $z_0$ so that $h(z_1) = c$. If the derivative of $h$ is bounded away from $0$,
and does not vary much near $z_0$, then $h$ can be well approximated by a linear function,
and $z_1$ can be found.

Thus to certify the existence of such a $z_1$, and to prove the validity of the
certificate, is not technically difficult.
However, Lemma~\ref{lemma:nearby_root} is organized carefully to 
be of use to us later, and it can be confusing to read.  
One should understand the lemma as saying ``if there are 
four constants $r,C,C',\delta$ which satisfy the hypotheses, then the 
conclusion holds''.  Do not worry about where the constants come from at this stage.
This lemma is very similar to Lemma~3.1 in \cite{Solomyak}.

\begin{lemma}\label{lemma:nearby_root}
Let $h$ be a holomorphic function and $z_0,c\in \C$ 
with $|h(z_0)-c|< \epsilon$.  
Suppose there are $r,C,C' > 0$ and $0 < \delta < 1$ such that 
$C' \le |h'(z)| \le C$ for all $z$ 
with $|z-z_0| < r$, and 
\[
r \ge \frac{\epsilon}{\delta} \frac{1+\frac{\delta^2}{1-\delta}}{C' - C\frac{\delta}{1-\delta}} 
    = \frac{\epsilon(1-\delta+\delta^2)}{\delta((1-\delta)C' - \delta C)}
\]
Then there exists a unique $z_1 \in \C$ with 
$|z_0-z_1| \le \epsilon \frac{(1-\delta+\delta^2)}{(1-\delta)C' - \delta C}$
such that $h(z_1) = c$.
\end{lemma}
\begin{proof}
First, it suffices to prove the theorem with $c=0$ by translation, so we will 
make that assumption.  

Write $h_a(z)$ for the affine part of the power series for $h$, centered 
at $z_0$, i.e. $h_a(z) = h(z_0) + h'(z_0)(z-z_0)$.  Under $h_a$, the circle 
$z_0 + de^{i\theta}$ of radius $d$ is mapped to the circle 
$h(z_0) + |h'(z_0)|de^{i\theta}$.  Therefore, if $d|h'(z_0)| > \epsilon$, 
the image circle will enclose $0$, and hence $0 \in h_a(B_d(z_0))$, or equivalently
$h_a$ will have a zero within $B_d(z_0)$, the ball of radius $d$ centered at $z_0$. 

Now consider $h$; it might not be affine, and we record the remainder term as $R_1$:
\[
h(z) = h_a(z) + R_1(z) = h(z_0) + h'(z_0)(z-z_0) + R_1(z).
\]
Suppose that there were a radius $d$ such that for all $0 \le \theta \le 2\pi$, 
we had $|h'(z_0)|d - |R_1(z_0+de^{i\theta})| \ge \epsilon$.  
In other words, the error in the affine 
approximation is smaller than the radius of the affine image circle minus $\epsilon$.  
Then the image of the circle $z_0 + de^{i\theta}$ under $h$ would have to 
contain $B_\epsilon(h(z_0))$, and hence $h$ would have a zero in $B_d(z_0)$.
Additionally, this follows immediately from Rouche's theorem, which also 
gives the claimed uniqueness.

To prove the lemma, then, it suffices to find a $d$ such that 
$|h'(z_0)|d - \epsilon \ge |R_1(z_0+de^{i\theta})|$ for all $0 \le \theta \le 2\pi$.  
From Taylor's theorem and Cauchy's derivative estimates, there is an inequality
\[
|R_1(z_0 + de^{i\theta})| \le \frac{M_r d^2}{r^2 - rd} \le \frac{(\epsilon + Cr)d^2}{r^2 - rd},
\]
where $M_r = \max_{\theta}|h(z_0 + re^{i\theta})|$, and the estimate is valid whenever $d < r$, and 
we can also estimate $M_r \le \epsilon + Cr$.

Set $d = \delta r$.  Rearranging the inequality in the hypothesis of the lemma, we have
\[
C'\delta r - \epsilon \ge \frac{(\epsilon +Cr)\delta^2r^2}{r^2 - \delta r^2}.
\]
Since $|h'(z_0)| \ge C'$, and plugging in $d = \delta r$, we have
\[
|h'(z_0)|d - \epsilon \ge \frac{(\epsilon +Cr)d^2}{r^2 - rd} \ge \frac{M_r d^2}{r^2 - rd} 
\]
Therefore, $d=\delta r$ satisfies the necessary inequality, so 
there is $z_1 \in B_{\delta r}(z_0)$ with $h(z_1) = c$.
Since making $r$ smaller maintains the validity of the bounds $C,C'$ for 
$|h'(z)|$, we may shrink $r$ until the inequality in the lemma is an equality, 
so the claimed bound on $|z_0-z_1|$ holds.
\end{proof}

\begin{remark}
The hypotheses of Lemma~\ref{lemma:nearby_root} may seem somewhat technical, 
but in fact they are not difficult to check in practice.  
We set $r$ to be quite small but still large compared to $\epsilon$, 
and we get bounds on the derivative. Then $\delta$ can be found by trial and 
error or any minimum-finding algorithm.  In fact, Mathematica 
produces an explicit formula for the $\delta$ which minimizes 
the expression on the right of the inequality for $r$; this formula is rather large and 
unedifying, so we omit it.

One feature we will make use of is that Lemma~\ref{lemma:nearby_root} can be 
checked for large collections of elements in $\partial \Sigma$ at the same time, since two 
words with a large common prefix will satisfy the same $C,C'$ bounds with 
similar values of $\epsilon$.
\end{remark}

\begin{remark}[Derivative bounds]
Lemma~\ref{lemma:nearby_root} requires good derivative bounds on $h'(z)$ a given ball $B_{z_0}(r)$.
A naive way to approach this is to get a universal upper bound $K$ on the second 
derivative and then state that $|h'(z)| < |h'(z_0)| + Kr$ on $B_{z_0}(r)$.  This is 
typically a bad estimate because $r$ can be large compared to the potential change 
in $h'(z)$.
Here is a better way.  Since $|h'(z)|$ is holomorphic, 
its maximum will lie on the boundary of $B_{z_0}(r)$.  Cover the boundary 
circle of $B_{z_0}(r)$ with many (say, $100$) small balls, use the naive approach 
on these small balls, and take the maximum.  Because the radius on which we 
apply the naive approach is now quite small, our error will be much less.
\end{remark}

\subsection{Paths in $\SetB$}

In this section, we explain how to find paths in $\SetB$.  
These paths will be rather short, but by piecing them together, 
we can produce loops and thus certify holes in $\SetB$.

We now give some initial observations about paths in $\SetB$ to 
clarify the construction to follow.
To each point $z$ in $\SetB$, there is a set of distinguished 
words in $\partial \Sigma$; namely, the words $x$ such that 
$\pi(x,z) = 1/2$.  Therefore, if we have a path $\gamma:[0,1] \to \D^*$
such that the image of $\gamma$ lies in $\SetB$, there is 
a combinatorial map $\lambda:[0,1] \to \partial \Sigma$ such that 
$\pi(\lambda(t), \gamma(t)) = 1/2$.  Of course, $\lambda$ is not 
uniquely defined, as there may be 
more than one word mapping to $1/2$ for a given parameter.

In order to build paths in $\SetB$, we essentially go in the other direction, 
Given two words $a,b \in \partial \Sigma$, we first build a nice combinatorial 
path interpolating between $a$ and $b$.  Then, provided we are close enough to 
$\SetB$, we show how apply Lemma~\ref{lemma:nearby_root} 
to produce a path of parameters which drags this combinatorial path along $1/2$.

In this lemma, we recall the notation $p_w(z) = \pi(w,z)$, the power series 
associated with $w\in\partial \Sigma$.

\begin{lemma}\label{lemma:lambda_image_in_B}
Suppose there are $\epsilon,r,C,C'>0$, $0< \delta < 1$, and $z_0 \in \C$
such that 
\begin{enumerate}
\setlength{\itemsep}{2mm}
\item $|z_0|+r  < 1$
\item $\displaystyle
r \ge \frac{\epsilon(1-\delta+\delta^2)}{\delta((1-\delta)C' - \delta C)}.$
\item For all $v \in u\partial\Sigma$ we have $|p_v(z_0)-1/2| < \epsilon$.
\item For all $v \in u\partial\Sigma$  and $z \in B_v(z_0)$ we have $C' < |p_v'(z)| < C$.
\end{enumerate}
Then for all $v \in u\partial \Sigma$, there is a unique 
$Z(v) \in B_{\delta r}(z_0)$ such that $p_v(Z(v)) = 1/2$.  
Consequently, there is a map $Z:u\partial \Sigma \to \SetB \cap B_{\delta r}(z_0)$ 
such that $p_v(Z(v)) = 1/2$.  Furthermore, $Z$ is uniformly continuous 
and the image $Z(u\partial \Sigma)$ is path connected.
\end{lemma}
\begin{proof}
That the map $Z$ exists and is well-defined (single-valued) 
follows immediately from Lemma~\ref{lemma:nearby_root}, 
so the content of the lemma is the uniform continuity and path connectedness.
We first address the former.  This is with respect to the Cantor metric, 
so it suffices to show that if two words $w_1,w_2 \in u\partial \Sigma$ have a 
sufficiently long common prefix, then their images under $Z$ are close (independent 
of what the prefix is).  

Let $K$ be equal to $|z_0|+r$. 
We claim that there exists a constant $I$ such that if $w_1,w_2 \in u\partial \Sigma$ have a 
common prefix $w$ of length at least $I$, then
\[
|Z(w_1) - Z(w_2)| < \frac{2 K^{|w|}}{|1-K|}
                     \frac{(1-\delta+\delta^2)}{((1-\delta)C' - \delta C)}.
\]
We now prove this claim.  We remark that $u$ is a prefix of $w$, since 
$w_1,w_2$ already have the common prefix $u$.
By Lemma~\ref{lemma:diameter_bound}, for a given $z$, the limit set 
$\Lz$ is contained in a ball of radius $|1-z|/2(1-|z|) < 1/(1-|z|)$ centered 
at $1/2$, so if 
$u$ is a word of length $n$, then $u(z,\Lz)$ is contained in a ball of 
size $|z|^n/(1-|z|)$ centered at $u(z,1/2)$.  
In our situation, then, the limit set $w(Z(w^\infty), \Lambda_{Z(w^\infty)})$ is contained inside a 
ball of radius $\frac{K^{|w|}}{|1-K|}$.  Therefore, we have
\[
|p_{w_1}(Z(w^\infty)) - 1/2|, |p_{w_2}(Z(w^\infty))-1/2| < \frac{K^{|w|}}{|1-K|}.
\]
We are going to apply Lemma~\ref{lemma:nearby_root} to $w_1$ and $w_2$ to get 
nearby roots, but there is a slight subtlety.  We have derivative bounds on 
all words in $u\partial \Sigma$ and $z \in B_{r}(z_0)$, but to apply 
Lemma~\ref{lemma:nearby_root}, we need derivative bounds in a ball centered at 
$Z(w^\infty)$.  We can achieve these bounds in the following way.
Since $Z(w^\infty) \in B_{\delta r}(z_0)$, the derivative bounds $C'$ and $C$ 
must be valid over $B_{(1-\delta)r}(Z(w^\infty))$.
So if $|w| > I$ for $I$ sufficiently long enough, then
\[
(1-\delta)r \ge \frac{K^{|w|}}{|1-K|}\frac{(1-\delta+\delta^2)}{\delta((1-\delta)C' - \delta C)},
\]
so we can apply Lemma~\ref{lemma:nearby_root} to the words $w_1,w_2$ 
at the point $Z(w^\infty)$ with radius $(1-\delta)r$ and 
$\epsilon = \frac{K^{|w|}}{|1-K|}$; this gives nearby $z_1,z_2$ so $\pi(w_1,z_1) = 1/2$ 
and $\pi(w_2,z_2) = 1/2$.  But $Z$ is uniquely defined, so $Z(w_1) = z_1$ and 
$Z(w_2) = z_2$, and hence
\[
|Z(w_1) - Z(w^\infty)|, \, |Z(w_2)-Z(w^\infty)| < \frac{K^{|w|}}{|1-K|}\frac{(1-\delta+\delta^2)}{((1-\delta)C' - \delta C)}.
\]
The claim that $Z$ is uniformly continuous follows from the triangle inequality, and therefore the image 
of $Z$ is compact.  It remains to show that 
the image $Z(u\partial \Sigma)$ is path connected.

Analogous to the set $W$ we constructed to build paths through $\Lz$ in Section~\ref{section:topology_of_lambda}, 
given any two words $a,b \in u\partial \Sigma$,
we will construct a combinatorial path through $u\partial \Sigma$ interpolating 
between them, and then show that applying $Z$ to this path gives a 
continuous path in $\SetB$.
Given a finite word $w$, denote by $\bar{w}$ the word obtained 
from $w$ by swapping $f$ and $g$.  
Note that if $w$ is finite and there is a parameter $z$ such 
that $w(z,1/2)=1/2$, then $\bar{w}(z,1/2) = 1/2$, so $p_{w^\infty}(z) = 1/2$ 
and $p_{\bar{w}^\infty}(z) = 1/2$.  Additionally, for any infinite word 
$w^\infty_*$ obtained by taking an infinite power of $w$ and swapping 
arbitrary copies of $w$ for $\bar{w}$, we have $p_{w_*^\infty}(z) = 1/2$.  Therefore, 
$Z(w^\infty) = Z(w_*^\infty)$.

Now let $H$ be a set of pairs of elements of $u\partial \Sigma$ 
indexed by the dyadic rationals and constructed inductively as follows.  
First set $H_0 = (a,a)$ and $H_1 = (b,b)$.  Next, given 
$H_{k2^{-i}}$ and $H_{(k+1)2^{-i}}$, let $v$ be the maximal common 
prefix of $H_{k2^{-i},2}$ and $H_{(k+1)2^{-i},1}$, and let 
\[
H_{k2^{-i} + 2^{-(i+1)} } = \Phi_{(v^\infty,\bar{v}^\infty)}(H_{k2^{-i},2}, H_{(k+1)2^{-i},1})
\]
That is, $H_{k2^{-i} + 2^{-(i+1)} }$ is either $(vv^\infty, v\bar{v}^\infty)$ 
or $(v\bar{v}^\infty, vv^\infty)$ depending on the first letters of 
$H_{k2^{-i},2}$ and $H_{(k+1)2^{-i},1}$ after the initial prefix.
By the observation above, the map $Z$ is well-defined on the pairs in 
$H$ because each pair consists of two words of the form $w_*^\infty$ for 
the same $w$.

By induction, if $k2^{-i} \le r_1 \le r_2 \le (k+1)2^{-i}$, then 
$H_{r_1}$ and $H_{r_2}$ have a common prefix of length at least $|u| + i$.  
Here we say $H_{r_1}$ and $H_{r_2}$ have a common prefix of length $n$ 
if at least one of the four possible pairings of a word in $H_{r_1}$ and 
$H_{r_2}$ has a common prefix of length $n$.
Since $Z$ is uniformly continuous, this means that $Z(H_r)$ is continuous 
as a function of the dyadic rational $r$, so $Z(H_r)$ extends continuously 
to $r \in [0,1]$, and $Z(u\partial \Sigma)$ is compact, so the image $Z(H_r)$ is contained in 
$Z(u\partial \Sigma)$ and is a path beginning at $Z(a)$ and ending at $Z(b)$, and the lemma 
is proved.
\end{proof}

\subsection{Holes in $\SetB$}
By a \emph{hole} in $\SetB$, we mean a connected component of the complement 
which is distinct from the ``obvious'' large connected component containing 
the point $0$.
Lemma~\ref{lemma:lambda_image_in_B} shows how to find a map $Z$ 
which takes a set of words $u\partial \Sigma$ into $\SetB$ in a nice way.  
In order to find a hole in $\SetB$, we will find words 
$u_0, \ldots, u_{n-1} \in \partial \Sigma$ satisfying Lemma~\ref{lemma:lambda_image_in_B}, 
thus giving maps $Z_i:u_i\partial \Sigma \to \SetB$.  The images $Z(u_i\partial \Sigma)$ 
are path connected, and we will show, for all $i$ with $i$ taken modulo $n$, 
that we have $Z(u_i\partial \Sigma) \cap Z(u_{i+1}\partial \Sigma) \ne \varnothing$.
Thus, there is a path passing through each image in turn.  Furthermore, 
we'll show that the images encircle a point which is not in $\SetB$.  This 
will complete the proof of the existence of a hole in $\SetB$.

Lemma~\ref{lemma:lambda_image_in_B} does not say what the images 
$Z(u_i\partial \Sigma)$ will look like; it only gives balls which are guaranteed 
to contain them.  To get a more precise picture, we do the following:
enumerate all words $\Sigma_m$ of some length $m$, and apply 
Lemma~\ref{lemma:lambda_image_in_B} to $Z(u_ix\partial \Sigma)$ for every 
$x \in \Sigma_m$.  If all these computations succeed, we obtain 
$2^m$ balls, and we know that (1) there is a point in $Z(u_i\partial \Sigma) \subseteq \SetB$ 
inside each ball and (2) these points are connected by paths inside $Z(u_i\partial \Sigma)$.

Therefore, if we can use this technique to exhibit, for each $i$, that the sets 
$Z(u_i\partial \Sigma)$ and $Z(u_{i+1}\partial \Sigma)$ lie transverse to each other, in the 
sense of traps, then they intersect.

\begin{theorem}[Holes in $\SetB$]
There is a hole in $\SetB$.
\end{theorem}\label{theorem:hole_in_SetB}
\begin{proof}
After the discussion above, this proof reduces to showing the 
pictures shown in Figure~\ref{figure:hole_in_B} and asserting 
that they were produced using the method above.  Note that 
this produces a \emph{loop} in $\SetB$, and checking if a parameter is 
\emph{not} in $\SetB$ is rigorous, so it suffices to exhibit 
a single pixel in the middle of the putative hole which is not in $\SetB$.  Many such pixels are 
easily visible.
\end{proof}

\begin{figure}[htb]
\includegraphics[scale=0.375]{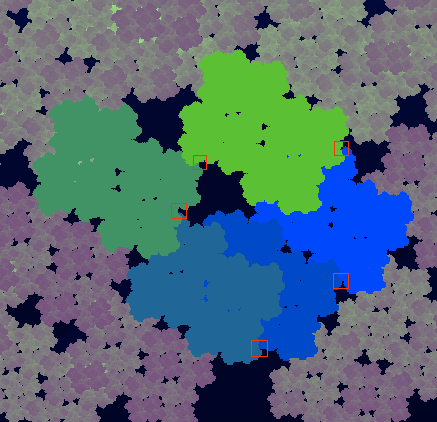}
~~\includegraphics[scale=0.295]{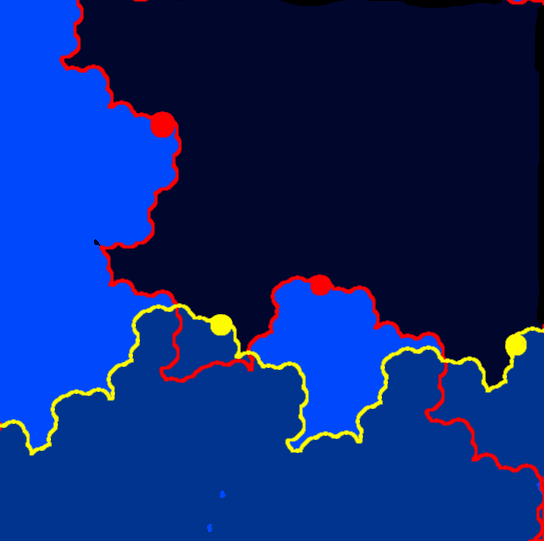}

\includegraphics[scale=0.3]{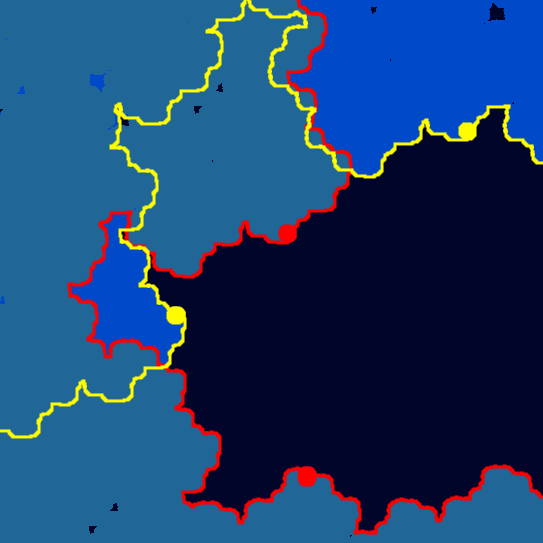}
~~\includegraphics[scale=0.3]{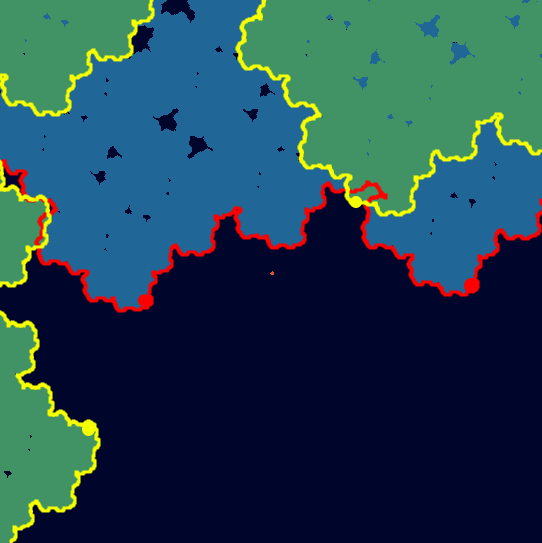}

\includegraphics[scale=0.3]{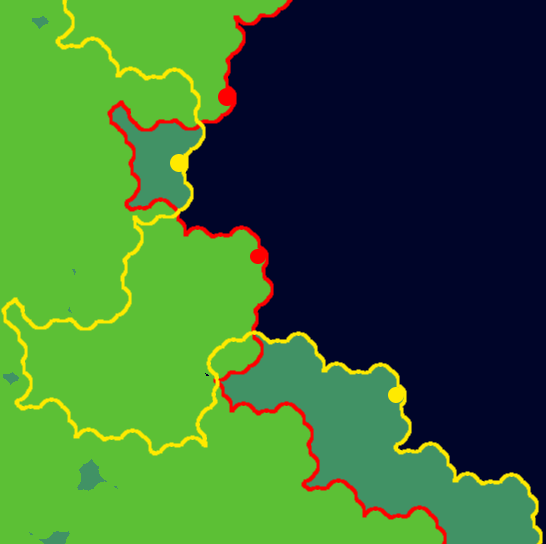}
~~\includegraphics[scale=0.3]{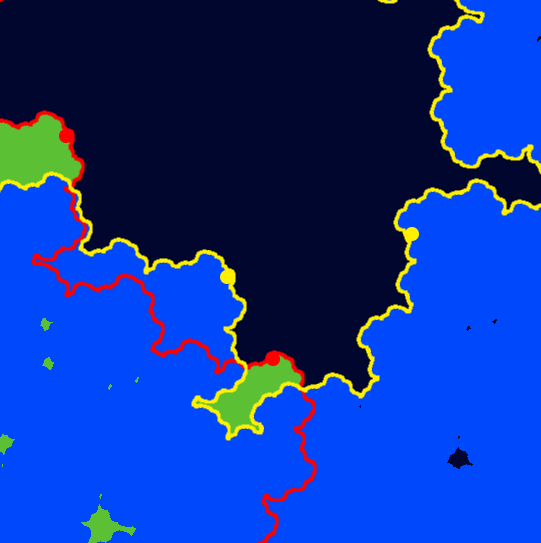}
\caption{The upper left picture shows the 
images of $Z(u_i\partial \Sigma)$ for $0\le i \le 4$, and 
the red boxes indicate the zoomed regions shown in the following 
pictures.  Each picture is made up of many small disks guaranteed to 
contain points in $\SetB$.  Four linked disks are highlighted in each 
picture to show that the various images of $Z$ must intersect, and 
each image is path connected, so there is a loop in $\SetB$.}
\label{figure:hole_in_B}
\end{figure}

\end{document}